\documentclass[12pt]{amsart}

\usepackage{geometry} 
\usepackage{amsmath,amsthm,amssymb,amscd, mathrsfs, mathtools, manfnt}
\usepackage{hyperref}
\usepackage[all, cmtip]{xy}

\newcommand{\ie}{i.e.\ }
\newcommand{\eg}{e.g.\ }
\newcommand{\inv}{^{-1}}

\newcommand{\Z}{\mathbb Z}
\newcommand{\N}{\mathbb N}

\newcommand{\D}{\mathbb D}
\newcommand{\E}{\mathbb E}

\newcommand{\bS}{\mathbb S}

\DeclareMathOperator{\K}{K} 

\newcommand{\Dcof}{{\D}_\text{cof}}
\newcommand{\Ecof}{{\E}_\text{cof}}

\newcommand{\ul}{\text{\smash{\raisebox{-1ex}{\scalebox{1.5}{$\ulcorner$}}}}}

\DeclareMathOperator{\dia}{dia}

\newcommand{\Sdot}{\mathbf{S}_\bullet}
\newcommand{\SD}[1]{\mathbf{S}_{#1}}
\newcommand{\Sn}{\mathbf{S}_n}
\DeclareMathOperator{\Ar}{Ar}

\newcommand{\Ord}{\mathbf{Ord}}

\DeclareMathOperator{\Ho}{Ho}
\DeclareMathOperator{\HO}{HO}

\newcommand{\op}{^\text{op}}
\newcommand{\pr}{\text{pr}}

\newcommand{\swtrans}{\mathbin{\rotatebox[origin=c]{225}{$\Rightarrow$}}}
\newcommand{\netrans}{\mathbin{\rotatebox[origin=c]{45}{$\Rightarrow$}}}

\newcommand{\Sp}{\mathbf{Sp}}

\newcommand{\cE}{\mathcal E}
\newcommand{\cW}{\mathcal W}
\newcommand{\cA}{\mathcal A}

\newcommand{\cC}{\mathcal C}
\newcommand{\cD}{\mathcal D}
\newcommand{\cB}{\mathcal B}
\newcommand{\cS}{\mathcal S}

\newcommand{\cM}{\mathcal M}

\newcommand{\Cat}{\mathbf{Cat}}
\newcommand{\StCat}{\mathbf{StCat}}
\newcommand{\CAT}{\mathbf{CAT}}

\newcommand{\Der}{\mathbf{Der}}

\DeclareMathOperator{\Hom}{Hom}
\DeclareMathOperator{\Fun}{Fun}
\DeclareMathOperator{\id}{id}

\DeclareMathOperator*{\colim}{colim}
\newcommand{\Dia}{\mathbf{Dia}}
\newcommand{\Dirf}{\mathbf{Dir_f}}

\renewcommand{\phi}{\varphi}

\numberwithin{equation}{section}
\numberwithin{figure}{section}

\theoremstyle{definition}
\newtheorem{theorem}[equation]{Theorem}
\newtheorem{prop}[equation]{Proposition}
\newtheorem{defn}[equation]{Definition}
\newtheorem{lemma}[equation]{Lemma}
\newtheorem{remark}[equation]{Remark}
\newtheorem{cor}[equation]{Corollary}

\newtheorem{notn}[equation]{Notation}

\newtheorem{ex}[equation]{Example}

\begin{document}

\title{The K-theory of left pointed derivators}
\author{Ian Coley}
\address{No longer academically affiliated, but please feel free to contact me}
\urladdr{\href{http://iancoley.org}{iancoley.org}}
\email{\href{mailto:derivators@iancoley.org}{derivators@iancoley.org}}
\begin{abstract}
We build on work of Muro-Raptis in \cite{MurRap17} and Cisinski-Neeman in \cite{CisNee08} to prove that the additivity of derivator K-theory holds for a large class of derivators that we call \emph{left pointed derivators}, which includes all triangulated derivators. The proof methodology is an adaptation of the combinatorial methods of Grayson in \cite{Gra11}. As a corollary, we prove that derivator K-theory is an infinite loop space. Finally, we speculate on the role of derivator K-theory as a trace from the algebraic K-theory of a stable $\infty$-category \`a la \cite{BluGepTab13}.
\end{abstract}
\keywords{Derivators, algebraic K-theory}
\maketitle

\tableofcontents

\section{Introduction}
Algebraic K-theory is a general tool for understanding complicated mathematical objects arising in homotopy theory, algebraic geometry, differential topology, representation theory, and other fields. Since Quillen's seminal work \cite{Qui73}, the field of algebraic K-theory has enjoyed incredible popularity and expansion beyond abelian or exact categories. Waldhausen in \cite{Wal85} set the tone for how K-theory would be constructed for more and more general objects. A common philosophy is that, if we expand the class of objects on which K-theory is defined, we should make sure that our new definitions agree with the old. Waldhausen made sure this was the case when he defined his K-theory of categories with cofibrations and weak equivalences.

However, the development of triangulated categories raised the question about whether K-theoretic information could be recovered using Waldhausen's construction alone or whether new techniques were needed. Defined first by Verdier in his doctoral thesis \cite{Ver96} in 1963, triangulated categories are invaluable in the study of homological algebra and homotopy theory. The bounded derived category associated to an exact or abelian category has a natural triangulated structure. The homotopy categories of both ordinary and $G$-equivariant spectra are similarly naturally triangulated. In algebraic geometry, the theory of motives is studied using triangulated categories, as the `abelian category of mixed motives' remains a conjecture.

Neeman in the 1990's published a series of papers on the K-theory of triangulated categories starting with \cite{Nee97a} and \cite{Nee97b}. There are a number of interesting properties of his construction, particularly the `theorem of the heart' which gives a homotopy equivalence between the K-theory of a triangulated category with a (non-degenerate) t-structure and the K-theory of the heart of that t-structure, see \cite[Strong~Theorem~I.7.1]{Nee99}. In the years following Neeman's publications, various authors proved that a satisfactory functorial construction on triangulated categories would never be possible, in the following sense. Starting from an exact category, we can take its K-theory via Quillen's or Waldhausen's definition, or pass to its bounded derived category and take its triangulated K-theory. If triangulated K-theory were to extend Quillen's K-theory, these two constructions would give the same K-groups. In other words, if $K_Q$ is Quillen's K-theory and $\cS$ is the category of spaces, can we find a functor $K_\Delta$ making the diagram below commute?
\begin{equation*}
\vcenter{\xymatrix@C=1em{
\textbf{ExCat}\ar[dr]_-{D^b(-)}\ar[rr]^-{K_Q}&&\cS\\
&\textbf{TriCat}\ar[ur]_-{K_\Delta}
}}
\end{equation*}

While not strictly answering the question above, Schlichting in \cite{Sch02} gives a general argument showing that we should not expect the triangulated category to retain all K-theoretic information. Specifically, he constructs two Waldhausen categories $\cW_1$ and $\cW_2$, arising from abelian categories of modules over a commutative ring, such that their homotopy categories (specifically, these are stable module categories) are equivalent but their Waldhausen $K_4$ differs. However, on the triangulated side, each homotopy category $\Ho(\cW_i)$ appears as a Verdier localization of the same category with equivalent localising subcategories. This leads to a contradiction between two desirable properties for K-theory: agreement and localization. The above diagram can only commute if $K_\Delta$ does not satisfy localization.

Schlichting's result points to the need for a richer structure for homotopy theory than triangulated categories alone. In some senses, the best replacement for triangulated categories, especially through the lens of algebraic K-theory, are \emph{stable $\infty$-categories}, \ie the triangulated analogue of higher categories. Recent work of Blumberg-Gepner-Tabuada in \cite{BluGepTab13} proves that algebraic K-theory is the universal additive invariant of a stable (small) $\infty$-category. This extends the origins of K-theory precisely; the Grothendieck group of a commutative monoid is the universal abelian group such that any additive invariant on the monoid must factor through it. Though the theory of $\infty$-categories is ideal for universal constructions, it is often difficult to know concretely what has been constructed. In the example of algebraic K-theory, there is an $\infty$-category of `non-commutative motives' in which algebraic K-theory is found -- but the rest of the category is quite mysterious (for now -- there is much active work on this topic).

There are lower-categorical tools that work well and are more constructive than $\infty$-categorical techniques. An early tool in studying triangulated categories, developed by Quillen in \cite{Qui67} before his work on algebraic K-theory, is that of \emph{model categories}. A model category is the data of a category we wish to treat homotopically and extra information allowing us to pass from the rigid structure to the homotopy category. A solution to the non-functoriality of the cone construction in triangulated categories can be solved in such a framework. For any model category $\cM$, the category of arrows $\Ar\cM$ inherits a compatible (Reedy) model category structure. We can define the cone of a morphism before passing to the homotopy categories, \ie $\Ho(\Ar\cM)\to\Ho(\cM)$ rather than $\Ar\Ho(\cM)\to\Ho(\cM)$. If we knew only the category $\Ho(\cM)$, this first approach will not be possible, so in this sense we have given ourselves more to work with.

Let us interpret this in the triangulated setting. Let $\cA$ be an abelian category. Then the arrow category $\Ar\cA$ is still abelian, so we can take its bounded derived category $D^b(\Ar\cA)$, which we can think of as homotopy classes of maps of chain complexes. Then we can define the cone construction as an exact functor of abelian categories $C^b(\Ar\cA)\to C^b(\cA)$ \emph{before} we invert quasi-isomorphisms. We still have a functor upon passing to the derived category, and so have a functorial cone construction with a new domain. However, there is a forgetful functor $D^b(\Ar\cA)\to\Ar D^b(\cA)$ which takes a homotopy class of a map to a map of homotopy classes. While $D^b(\Ar\cA)$ is triangulated, $\Ar D^b(\cA)$ is not, but this functor can be shown to be full and essentially surjective (and almost never faithful). We have constructed a cone \emph{functor} because we had access to $\cA$ itself and not just the triangulated category $D^b(\cA)$ and thus were able to build an auxiliary diagram category $\Ar\cA$ to fill in the gaps in information.

This is our slogan: we would like to study not only a triangulated category, but a whole system of triangulated diagram categories. An equivalence of homotopy categories as in \cite{Sch02} does not necessarily give rise to an equivalence of systems, and thus we are able to better distinguish distinct homotopy theories. Grothendieck in \cite{Gro90} coined the term \emph{derivator} for a system of derived categories, and this is the framework in which we will work to address questions about the K-theory of triangulated categories.

The theory of derivators was developed initially (under different names) by Heller in \cite{Hel88},  Grothendieck in \cite{Gro90}, and (in the triangulated setting) Franke in \cite{Fra96}. In brief, a derivator represents an abstract bicomplete homotopy theory; we attach the adjective \emph{triangulated} to a derivator when it represents a stable (bicomplete) homotopy theory. The fundamental proof techniques used in the theory of derivators and the diagrammatic flavor which is unique to this field have been well-articulated by Moritz Rahn (n\'e Groth) in \cite{Gro13}.

The K-theory of triangulated derivators was defined by Maltsiniotis in \cite{Mal07} and Garkusha in \cite{Gar06} and \cite{Gar05}, and revisited by Muro and Raptis in \cite{MurRap17}. Muro-Raptis proved that the definition of K-theory still makes sense for derivators which are not triangulated, and form a class which we call \emph{left pointed derivators}. We develop in \cite{Col20b} a more robust theory of such \emph{half derivators}, \ie ones representing homotopy theories that may not be bicomplete, but still admit many limits or colimits, in order to answer questions about K-theory in the broadest generality. This is one advantage over the approach of \cite{BluGepTab13}: we are not restricted to stable phenomena.

Cisinski and Neeman proved that the K-theory of triangulated derivators satisfies a form of additivity in \cite{CisNee08}, but their proof involves adapting Neeman's theory of regions developed in his triangulated K-theory papers and does not admit an obvious analogy in the non-triangulated situation. We prove the following broader theorem.\ \\

\noindent\textbf{Main Result. }(Theorems~\ref{thm:additivities} and \ref{thm:mainadditivity})

\noindent Let $\D$ and $\E$ be left pointed derivators. Then the following are equivalent and true:
\begin{enumerate}
\item The map
\begin{equation*}
\xymatrix@C=5em{
\Dcof\ar[r]^-{(0,0)^\ast\times(1,1)^\ast} &\D\times\D
}
\end{equation*}
induces a homotopy equivalence on derivator K-theory, where $\Dcof$ is the left pointed derivator of cofiber sequences in $\D$.

\item If $\Xi\colon\D\to\Ecof$ is a cofibration morphism of derivators, then there exists a homotopy between the target of the cofibration morphism and the source plus quotient. Specifically,
\begin{equation*}
\K(T)\simeq \K(S\vee Q)
\end{equation*}
where $S\vee Q\colon\D\to\E$ is the (non-strict) coproduct of the morphisms $S$ and $Q$, defined in Equation~\ref{eq:coprod} below. As we explain in Remark~\ref{rk:hspace}, the homotopy class of $\K(S\vee Q)$ is the sum of the homotopy classes of $\K(S)$ and $\K(Q)$ in the abelian group $[\K(\D),\K(\E)]$ of pointed homotopy classes of maps, so this is indeed a statement of the form `target $=$ source $+$ quotient'.
\end{enumerate}

The first statement is the form of additivity proven by Cisinski-Neeman and conjectured by Maltsiniotis in \cite[Conjecture~3]{Mal07} in the triangulated setting. The second statement is the one proven in Theorem~\ref{thm:mainadditivity} and uses techniques of Grayson in \cite{Gra11} that have a more diagrammatic flavor appropriate for general derivator theory. We obtain as a corollary the delooping of the K-theory space $\K(\D)$, and so conclude that $\K(\D)$ is an infinite loop space for a general left pointed derivator~$\D$, which was not known in any cases before. This answers the questions of Muro-Raptis posed in \cite[\S6.3]{MurRap17}.

We first recall the necessary results from \cite{Col20b} that establish the domain of derivator K-theory. We then give the construction of derivator K-theory and present previously-known results. We conclude with the main new additivity theorem of derivator K-theory and the important consequences thereof.

This research was conducted with the support of the UCLA Derivators Seminar: Paul Balmer, Kevin Carlson (formerly Arlin), Ioannis Lagkas, and John Zhang. The author would like to thank George Raptis for helpful comments which added to the $\infty$-categorical results in the last section. Large parts of this research were conducted at the Universidad de Sevilla as a guest of Fernando Muro, who helped declutter many of the proofs in this paper. Thanks to Chuck Weibel at Rutgers for helping push this paper through the revision process during the author's postdoctoral position. Finally, thanks are due to the anonymous referee, whose comments improved the flow and results of this paper tremendously.

As a final note, this represents the author's (probable) last publication in research mathematics, as he has transitioned out of academia to less categorical pursuits in the working world. Starting a new career, long referee cycles with multiple bounces, and the ever-present complicating factors of COVID-19 prevented this paper's formal publication in a peer-reviewed journal. Nonetheless, the author has done his best to address all the final referee's comments despite missing the deadline on properly resubmitting the paper for final review. Any errors contained within are therefore strictly the author's own business, though if any derivator-enthusiasts still persist, he would be thrilled to make any post-final revisions that are recommended.

\section{Preliminaries}\label{sec:preliminaries}

We recommend to the reader who is interested in the theory of derivators \emph{per se} \cite{Col20b}, which was written with the purpose of constructing precisely the domain of derivator K-theory. Much of that paper is also devoted to the history and intuition behind derivators, which we will more or less omit from the current work, instead focusing on the key definitions and vocabulary to set up derivator K-theory.

We provide references principally to \cite{Gro13}, which brought together disparate results in the derivator literature into one place. Definitions will be left unattributed, as the theory of derivators was simultaneously initiated by Grothendieck, Franke, and Heller as we mentioned above. Key propositions will be cited to \cite{Gro13}, although they can also be found in \cite{Col20b} (sometimes with new proofs).

A \emph{prederivator} is a strict 2-functor $\D\colon\Cat\op\to\CAT$, where the domain is the 2-category of small categories with only the 1-morphisms reversed and the codomain the `2-category' of not-necessarily-small categories. For a morphism $u\colon J\to K$ in $\Cat$ we denote by~$u^\ast$ the functor $\D(u)\colon \D(K)\to~\D(J)$ in $\CAT$, and for $\alpha\colon u\Rightarrow v$ in $\Cat$ we denote by~$\alpha^\ast$ the natural transformation $\D(\alpha)\colon u^\ast\Rightarrow v^\ast$. Composition is respected strictly, so that $(v u)^\ast = u^\ast v^\ast$ and $(\alpha\odot\beta)^\ast=\alpha^\ast\odot\beta^\ast$ (here $\odot$ is the pasting of natural transformations). Identities are also preserved, so that\linebreak $(\id_J)^\ast=\id_{\D(J)}$ and $(\id_u)^\ast=\id_{u^\ast}$.

A \emph{derivator} is a prederivator that models a system of diagram categories which is homotopically bicomplete. We give the definition in two parts.

Let $\D$ be a prederivator, $K$ a small category, and $k\in K$ be any object. Recall that we have a functor that classifies the object $k$ which we denote $k\colon e\to K$, where $e$ is the category with one object and one (identity) morphism. Then for any~$X\in \D(K)$, we have an object $k^\ast X\in \D(e)$. Suppose that $f\colon k_1\to k_2$ is a map in $K$. Then we can consider it as a natural transformation between object classifying functors $f\colon k_1\Rightarrow k_2$ and after applying $\D$ we obtain natural transformation $f^\ast\colon k_1^\ast \Rightarrow k_2^\ast$ and thus a map $f^\ast X\colon k_1^\ast X\to k_2^\ast X$ in $\D(e)$. Repeating this process for all objects and maps in $K$, we obtain a functor
\begin{equation*}
\dia_K\colon \D(K)\to \Fun(K,\D(e))
\end{equation*}
which sends $X\in\D(K)$ to the functor which assembles all the above data. We call this an \emph{underlying diagram functor}, and its existence implies that the prederivator $\D$ should be modelling $K$-shaped diagrams in $\D(e)$, which we call the \emph{underlying category} or the \emph{base} of the prederivator. We will refer to the categories $\D(K)$ as \emph{coherent} diagrams, as opposed to the \emph{incoherent} diagrams $\Fun(K,\D(e))$.

\begin{defn}
A \emph{semiderivator} is a prederivator $\D$ satisfying the following two axioms:
\begin{enumerate}
\item[(Der1)] Coproducts are sent to products. Explicitly, consider any set $\{K_a\}_{a\in A}$ of small categories, and let $i_b\colon K_b\to \displaystyle\coprod_{a\in A} K_a$ be the inclusion for any $b\in A$. Pulling back along this inclusion gives a functor
\begin{equation*}
i_b^\ast\colon\D\left(\coprod_{a\in A}K_a\right)\to\D(K_b)
\end{equation*}
which induces a map to the product
\begin{equation*}
(i_b^\ast)_{b\in A}\colon\D\left(\coprod_{a\in A}K_a\right)\to \prod_{b\in A}\D(K_b)
\end{equation*}
We require this map to be an equivalence of categories for any\linebreak collection $\{K_a\}_{a\in A}$. In particular, if $A=\varnothing$, we obtain an equivalence $\D(\varnothing)\cong e$.

\item[(Der2)] Isomorphisms are detected pointwise. That is, for any $K\in \Cat$, the underlying diagram functor $\dia_K$ is conservative. More specifically, a map $f\colon X\to Y$ is an isomorphism in $\D(K)$ if and only if the map $k^\ast f\colon k^\ast X\to k^\ast Y$ is an isomorphism for all $k\in K$.
\end{enumerate}
\end{defn}

These two axioms comprise the `system of diagram categories' part of the definition. This also furnishes us with a large class of examples.

\begin{ex}\label{ex:semiderivators}
For the following examples, let $u\colon J\to K$ be a functor in $\Cat$.
\begin{enumerate}
\item Let $\cC\in\CAT$ be any category. Define the prederivator $\D_\cC$ via the Yoneda embedding:
\begin{equation*}
K\mapsto \Fun(K,\cC),\quad u\colon J\to K\,\mapsto u^\ast\colon \Fun(K,\cC)\to \Fun(J,\cC)
\end{equation*}
with the action on natural transformations obvious. Here the map $u^\ast$ is precomposition with $u$, and for this reason we usually refer to $u^\ast$ as a \emph{pullback functor} even when $\D$ is not this specific prederivator. This is literally a system of diagram categories, so $\D_\cC$ is a semiderivator.

\item Let $\cA$ be a Grothendieck abelian category, \eg the category of $R$-modules for any commutative ring $R$. We can then form the (unbounded) derived category $D(\cA)$ without any set-theoretic issues. For any small category $K$, the category $\Fun(K,\cA)$ is still a Grothendieck abelian category, so we may take its derived category as well. The assignment $\D_{\cA}(K)=D(\Fun(K,\cA))$ defines a semiderivator $\D_\cA$ with~$u^\ast$ induced by precomposition.

\item Let $\cM$ be a combinatorial model category, \eg the category of simplicial sets with either the standard or Joyal model structure. The first model structure makes \textbf{sSet} a model for topological spaces and the second makes it a model for $(\infty,1)$-categories. In any case, for any small category $K$ the category $\Fun(K,\cM)$ is still a combinatorial model category (using either the injective or projective model structure) where the weak equivalences are defined pointwise. Then define the semiderivator $\D_\cM$ by $\D_\cM(K)=\Ho(\Fun(K,\cM))$ with~$u^\ast$ again induced by precomposition.
\end{enumerate}
\end{ex}

For the next two axioms, we need the following notation.

\begin{defn}
Let $u\colon J\to K$ be any functor, and let $k\in K$ be any object. We define the \emph{comma category} $(u/k)$ as follows: its objects are pairs $j\in J$ with a map~$f\colon u(j)\to k$, and a map $(j,f)\to (j',f')$ in the comma category is a\linebreak map $g\colon j\to j'$ in $J$ making the obvious diagram commute:
\begin{equation}\label{dia:commamorphism}
\vcenter{\xymatrix@C=1em{
u(j)\ar[rr]^-{u(g)}\ar[dr]_(0.4){f}&&u(j')\ar[dl]^(0.4){f'}\\
&k
}}
\end{equation}
\end{defn}
For any category $K\in\Cat$, we write $\pi_K$ for the unique functor $K\to e$.

\begin{defn}\label{defn:leftderivator}
A semiderivator $\D$ is a \emph{left derivator} if it satisfies the following two axioms:
\begin{enumerate}
\item[(Der3L)] The semiderivator $\D$ is (homotopically) cocomplete. Specifically, for every functor $u\colon J\to K$, the pullback $u^\ast$ admits a left adjoint, which we denote $u_!\colon\D(J)\to\D(K)$ and call the \emph{(homotopy) left Kan extension along $u$}.

\item[(Der4L)] Left Kan extensions can be computed pointwise. Let $u\colon J\to K$ and $k\in~K$. Then we have the following lax pullback square in $\Cat$:
\begin{equation*}
\vcenter{\xymatrix{
(u/k)\ar[r]^-{\pr}\ar[d]_-{\pi}&J\ar[d]^-u\ar@{}[dl]|\swtrans\ar@{}[dl]<-1.25ex>|\alpha\\
e\ar[r]_k&K
}}
\end{equation*}
where we let $\pi=\pi_{(u/k)}$ for brevity. Applying the semiderivator $\D$ to this square, we obtain the following square in $\CAT$, remembering that functors are reversed and natural transformations are not:
\begin{equation*}
\xymatrix{
\D((u/k))&\D(J)\ar@{}[dl]|{\swtrans}\ar@{}[dl]<-1.25ex>|{\alpha^\ast}\ar[l]_-{\pr^\ast}\\
\D(e)\ar[u]^-{\pi^\ast}&\D(K)\ar[l]^-{k^\ast}\ar[u]_-{u^\ast}
}
\end{equation*}

By Der3L, both vertical functors admit left adjoints, so we may construct the left mate of $\alpha^\ast$ as the pasting of the below diagram, which we denote by~$\alpha_!$ (rather than the `official' notation $(\alpha^\ast)_!$):
\begin{equation*}
\xymatrix{
\D(e)&\ar[l]_-{\pi_!}\D((u/k))\ar@{}[dl]|{\swtrans}&\D(J)\ar@{}[dl]|{\swtrans}\ar@{}[dl]<-1.25ex>|{\alpha^\ast}\ar[l]_-{\pr^\ast}&\ar@{}[dl]|{\swtrans}\\
&\D(e)\ar@(l,d)[ul]^-=\ar[u]^-{\pi^\ast}&\D(K)\ar[l]^-{k^\ast}\ar[u]_-{u^\ast}&\D(J)\ar[l]^-{u_!}\ar@(u,r)[ul]_-=
}
\end{equation*}
In total we have the natural transformation $\alpha_!\colon \pi_!\pr^\ast\Rightarrow k^\ast u_!$. We require this map to be a natural isomorphism. 
\end{enumerate}
\end{defn}

As a particular example, the left Kan extension along the projection $\pi_K\colon K\to e$ can be thought of taking the colimit of a (coherent) diagram $X\in\D(K)$. A general left Kan extension can be computed pointwise as certain colimits.

A semiderivator $\D$ is a \emph{right derivator} if it satisfies the analogous axioms Der3R and Der4R, which together say that every functor $u^\ast$ admits a right adjoint $u_\ast$ satisfying a pointwise computation formula. A \emph{derivator} is just a left and right derivator.

In Example~\ref{ex:semiderivators}, (2) and (3) are always derivators. In case~(1), the functors $u_!$ and $u_\ast$ are Kan extensions in the usual sense, and these exist if and only if $\cC$ admits all limits and colimits, see (e.g.) \cite[Theorem~3.7.2]{Bor94a}.

\begin{remark}
Historically in derivator literature, a derivator which admits all colimits and in which left Kan extensions may be computed pointwise is called a right derivator. This is due to an analogy with right exact functors (\ie those that admit \emph{left} derived functors) that we feel does not justify the confusing terminology.
\end{remark}

There is a relative construction that we need to introduce at this point. Suppose that $\D$ is a prederivator and $I\in \Cat$ is a category. Then we can define another prederivator $\D^I$ by $\D^I(K)=\D(I\times K)$; for $u\colon J\to K$, $\D^I(u)=\D(\id_I\times u)$; and similar for natural transformations. If $\D$ is a (left/right/full) derivator, so is $\D^I$. This is often called a \emph{shifted derivator}.
There is a `fifth axiom' for derivators that is not needed in all contexts, ours included, but we mention it for completeness.
\begin{defn}
A prederivator is \emph{strong} if for any finite free category $I$ and for any category $K\in\Cat$, the partial underlying diagram functor
\begin{equation*}
\dia_{K,I}\colon\D(K\times I)\to\Fun(I,\D(K))
\end{equation*}
is full and essentially surjective. In some derivator literature this axiom is called Der5.
\end{defn}
The functor $\dia_{K,e}$ is exactly $\dia_K$ defined above. The functor $\dia_{K,I}$ can be identified with the functor $\dia_K$ for the shifted derivator $\D^I$; it leaves the $I$-dimension of the coherent diagrams intact but makes the $K$-dimension incoherent.

The content of this axiom is that any incoherent diagram of a finite free shape is liftable to a coherent one; further, any map of these incoherent diagrams lifts to a map between the coherent ones. While we should not expect to lift incoherent diagrams of any shape, for finite free categories we ask for no obstructions to homotopy coherence.

\begin{remark}
The strongness axiom usually only asks for the case $I=[1]$ (and sometimes all finite ordinals $[n]$, see Notation~\ref{notn:ordinals}), but all known examples either satisfy this `strong strongness' version of the axiom or fail for $I=[1]$. Derivators failing the case $I=[1]$ are constructed in \cite[Remark~5.4]{Lag17}, so any version of strongness is a non-extraneous axiom. Nonetheless all derivators arising from some sort of model satisfy Der5 as above (see Theorem~\ref{thm:waldhausenderivator} below).
\end{remark}



We have one more adjective to attach to our derivators. From axiom Der3, we know that $\D(e)$ admits (coherent) limits and colimits. In particular, consider the left Kan extension along the empty projection $\pi_{\varnothing}\colon \varnothing\to e$, which yields $\pi_{\varnothing,!}\colon\D(\varnothing)\to\D(e)$. Since $\D(\varnothing)\cong e$ by Der1, this functor picks out an essentially unique object which is the colimit over the empty set, namely an initial object. By taking the right Kan extension instead, $\D(e)$ also admits a final object.

\begin{defn}
A derivator $\D$ is \emph{pointed} if its underlying category $\D(e)$ is pointed, \ie the unique morphism from an initial to a final object is an isomorphism (and thus the map from \emph{any} initial to \emph{any} final object is an isomorphism). We will write $0\in\D(e)$ for any (essentially unique) zero object.
\end{defn}
We will not think of $0\in \D(e)$ as a fixed or distinguished zero object. If $k\in K$ and $X\in \D(K)$, we write $k^\ast X=0$ to mean that $k^\ast X$ is \emph{some} zero object in $\D(e)$. In a later proof, we will need to distinguish between different zero objects in the same coherent diagram in order to prove a certain construction defines a functor; in this case we will write $0_i$ for various zero objects.

This definition is easy to check but does not tell the whole story. One immediate consequence is that each category $\D(J)$ has a zero object, given by $0_J:=\pi_J^\ast(0)$ where~$\pi_J\colon J\to e$ is the projection. The more interesting corollary requires some definitions first:

\begin{defn}
Let $u\colon J\to K$ be a fully faithful functor that is injective on objects.
\begin{enumerate}
\item The functor $u$ is a \emph{sieve} if for any morphism $k\to u(j)$ in $K$, $k$ lies in the image of $u$.
\item The functor $u$ is a \emph{cosieve} if for any morphism $u(j)\to k$ in $K$, $k$ lies in the image of $u$.
\end{enumerate}
\end{defn}

The following proposition is true for general (left or right) derivators, but becomes relevant for us when applied to a pointed derivator. Nonetheless, we state it in its full generality.

\begin{prop}[Proposition~1.23,~\cite{Gro13}]\label{prop:extensionbyzero}
Let $\D$ be a derivator, and let\linebreak $u\colon J\to K$ be a sieve (resp. cosieve). Then $u_\ast\colon\D(J)\to\D(K)$ (resp. $u_!$) is fully faithful, with essential image $X\in \D(K)$ such that $k^\ast X$ is a final object (resp. initial object) for all $k\in K\setminus u(J)$.
\end{prop}

In the case that $\D$ is a pointed derivator, we call these adjoints \emph{extension by zero functors}, as final and initial objects coincide. These will be used often in the main proofs of this paper.

\begin{prop}[Corollaries~3.5~and~3.8, \cite{Gro13}]\label{prop:stronglypointed}
A derivator is pointed if and only if extension by zero morphisms admit exceptional adjoints. Specifically, every right extension by zero $u_\ast$ along a sieve admits a right adjoint $u^!$ and every left extension by zero $u_!$ along a cosieve admits a left adjoint $u^?$.
\end{prop}

The existence of exceptional adjoints is crucial for functoriality properties on derivator K-theory, which we now address.

\begin{defn}
Let $\D,\E\colon\Cat\op\to\CAT$ be prederivators. A \emph{morphism of prederivators} $\Phi\colon\D\to\E$ is a pseudonatural transformation of the associated 2-functors. This consists of the following data: for each $K\in\Cat$ we have a\linebreak functor $\Phi_K\colon\D(K)\to~\E(K)$ and for every $u\colon J\to K$ we have a natural isomorphism $\gamma_u^\Phi\colon u^\ast \Phi_K\Rightarrow \Phi_J u^\ast$
\begin{equation*}
\xymatrix{
\D(K)\ar[r]^{\Phi_K}\ar[d]_-{u^\ast}&\E(K)\ar[d]^-{u^\ast}\ar@{}[dl]|{\swtrans}\ar@{}@<-1.75ex>[dl]|{\gamma_u^\Phi}\\
\D(J)\ar[r]_{\Phi_J}&\E(J)
}
\end{equation*}
where we have slightly abused notation by writing $u^\ast$ for both $\D(u)$ and $\E(u)$. These are subject to certain coherence conditions which we leave to \cite{Col20b}, \cite{Gro13}, or \cite[\S7]{Bor94a}.
\end{defn}

We refer to the collection $\{\gamma_u^\Phi\}$ as the \emph{structure isomorphisms} of $\Phi$. When defining morphisms of derivators, instead of writing $\Phi_K(X)$ for $X\in \D(K)$ for all $K\in\Cat$, we will usually write $\Phi X$ for $X\in \D$. Our constructions will be not heavily dependent on specific $K$.

A morphism of derivators is just a morphism of the underlying prederivators. We say that a morphism of (pre)derivators is an \emph{equivalence} if each functor $\Phi_K$ is an equivalence of categories. There is a subclass of morphisms that need singling out.

\begin{defn}
A morphism of (pre)derivators $\Phi\colon\D\to\E$ is called \emph{strict} if for every $u\colon J\to K$, the structure isomorphism $\gamma_u^\Phi\colon u^\ast \Phi_K\Rightarrow\Phi_Ju^\ast$ is the identity.
\end{defn}

In 2-categorical language, a strict morphism $\Phi$ is a \emph{strict} natural transformation of 2-functors, not just \emph{pseudo}natural. The model of derivator K-theory we use in this paper will require strict morphisms, and we will be able to obtain strict morphisms (up to equivalence) whenever we need. Whenever we start with a functor on models, \eg an exact functor of Grothendieck abelian categories $F\colon \cA\to\cB$, the morphism on derivators $F\colon \D_\cA\to\D_\cB$ will be strict; most morphisms we use in this paper do not arise this way.

The main class of morphisms of derivators that we study involve shifted derivators. Suppose that $u\colon J\to K$ is a functor and $\D$ is any prederivator. Then we obtain a morphism of prederivators $u^\ast\colon\D^K\to\D^J$ which is actually strict, as the coherence data $\gamma^{u^\ast}$ arise from the strict 2-functoriality of $\D\colon\Cat\op\to\CAT$. Moreover, if $\D$ is a left or right derivator, we obtain morphisms $u_!,u_\ast\colon \D^J\to \D^K$, but these are \emph{not} strict. This is related to the fact that (co)limits are essentially unique, which allows for the construction of the structure isomorphisms $\{\gamma^{u_!}\}$, but not actually unique. These morphisms enjoy other properties which we will describe now.

\begin{defn}
Let $\D,\E$ be left derivators and $u\colon J\to K$ in $\Cat$. We say that a morphism $\Phi\colon\D\to\E$ \emph{preserves left Kan extensions along $u$} if the left mate of $(\gamma^\Phi_u)\inv$ is a natural isomorphism. Specifically, we have the pasting
\begin{equation*}
\xymatrix{
\D(J)\ar@(d,l)[dr]_-=\ar[r]^-{u_!}&\D(K)\ar@{}[dl]|{\netrans}\ar[r]^{\Phi_K}\ar[d]_-{u^\ast}&\E(K)\ar[d]^-{u^\ast}\ar@{}[dl]|{\netrans}\ar@{}@<-1.75ex>[dl]|{(\gamma_u^\Phi)\inv}\ar@(r,u)[dr]^-=&\ar@{}[dl]|{\netrans}\\
{}&\D(J)\ar[r]_{\Phi_J}&\E(J)\ar[r]_{u_!}&\E(K)
}
\end{equation*}
giving us a natural transformation $(\gamma_u^\Phi)\inv_!\colon u_!\Phi_J\Rightarrow\Phi_K u_!$ which we demand is an isomorphism, where again we slightly abuse notation by writing $u_!$ for the left adjoint to both $\D(u)$ and $\E(u)$. If the morphism $\Phi$ preserves left Kan extensions along\linebreak all $u\colon J\to K$ in $\Cat$, we say that $\Phi$ is \emph{cocontinuous}.

There is an analogous notion of \emph{continuous morphism} that we will not spell out (as we will not need it below).
\end{defn}

Cocontinuous morphisms of derivators can appear in the same way as colimit-preserving functors in category theory: via adjunctions.

\begin{defn}
Given two morphisms of (pre)derivators $\Phi,\Psi\colon\D\to \E$, a natural transformation $\rho\colon\Phi\Rightarrow\Psi$ is given by a \emph{modification} of pseudonatural transformations. This is the data of a natural transformation $\rho_K\colon \Phi_K\Rightarrow\Psi_K$ for every $K\in\Cat$ satisfying coherence conditions that we do not record here. If $\rho\colon \Phi\Rightarrow\Psi$ is a modification such that each $\rho_K$ is a natural isomorphism, we say that $\rho$ is an \emph{isomodification}.
\end{defn}

\begin{defn}
Let $\Phi\colon\D\to\E$ and $\Psi\colon \E\to\D$ be two morphisms of (pre)derivators. We say that \emph{$\Phi$ is left adjoint to $\Psi$} (equivalently, \emph{$\Psi$ is right adjoint to $\Phi$}) if there exist two modifications $\eta\colon \id_\D\Rightarrow\Psi\Phi$ and $\varepsilon\colon \Phi\Psi\Rightarrow\id_\E$ satisfying the usual triangle identities.
\end{defn}

In particular, an adjunction of morphisms of derivators $(\Phi,\Psi)$ gives rise to an adjunction of functors $(\Phi_K,\Psi_K)$ for each $K\in\Cat$. However, this condition is not sufficient. A morphism of derivators $\Phi\colon\D\to\E$ may admit a right adjoint\linebreak to $\Phi_K\colon\D(K)\to\E(K)$ for all $K\in\Cat$, but part of the data of a right adjoint morphism of derivators is the structure isomorphisms $\{\gamma^\Psi_u\}$, which we have no way of recovering in this general situation.

\begin{lemma}[Proposition~2.9, \cite{Gro13}]
Let $\Phi\colon\D\to\E$ be a morphism of left derivators such that each $\Phi_K$ admits a right adjoint $\Psi_K$. Then the collection of functors $\{\Psi_K\}$ assembles to a morphism of derivators $\Psi\colon\E\to\D$ which is right adjoint to $\Phi$ if and only if $\Phi$ is cocontinuous.
\end{lemma}

The morphism $\Phi$ being cocontinuous allows us to construct the structure isomorphisms $\gamma^\Psi_u$ and `glue together' the various $\Psi_K$. We are not claiming anything like an adjoint functor theorem for general derivators, so this lemma does not admit a converse.

\begin{ex}\label{ex:adjoints}
There are three classes of examples that give us all the cocontinuous morphisms for this paper.
\begin{enumerate}
\item Let $u\colon J\to K$ be a functor in $\Cat$. If $u$ admits a categorical right adjoint $v\colon K\to J$, then $u^\ast\colon\D(K)\to\D(J)$ is right adjoint to $v^\ast\colon\D(J)\to\D(K)$ because (strict) 2-functors send adjunctions to adjunctions, though in our case which is left and which is right swaps. We can upgrade this to, for any prederivator $\D$, a (cocontinuous) left adjoint morphism $v^\ast\colon \D^J\to\D^K$ which preserves any left Kan extensions that $\D^K$ happens to have.

\item If $\D$ is a left derivator, then the left adjoint functor $u_!\colon \D(J)\to\D(K)$ lifts to a left adjoint morphism of derivators $u_!\colon \D^J\to\D^K$ with right adjoint $u^\ast$. Similarly, if~$\D$ is a right derivator, $u_\ast\colon \D^J\to\D^K$ is a right adjoint morphism of derivators with left adjoint $u^\ast$. In fact, for any prederivator $\D$, the morphism $u^\ast\colon\D^K\to\D^J$ preserves all left and right Kan extensions that $\D$ happens to have by \cite[Proposition~2.5]{Gro13} for categorical reasons.

\item Finally, suppose $\D$ is a pointed derivator and $u\colon J\to K$ is a sieve. Then the right extension by zero $u_\ast\colon\D(J)\to\D(K)$ admits an exceptional right adjoint by Proposition~\ref{prop:stronglypointed}, so $u_\ast\colon \D^J\to\D^K$ is a left adjoint and hence cocontinuous.
\end{enumerate}
We will refer to the functor $u^\ast\colon\D(K)\to\D(J)$ and the morphism of shifted derivators $u^\ast\colon \D^K\to \D^J$ by $u^\ast$ (similarly $u_!,u_\ast$) as a slight abuse of notation.
\end{ex}
\section{Left pointed derivators}

Having set up the basic vocabulary of the theory of derivators, we can begin to examine what we actually need for K-theory.

To motivate the following definition, we recall the definition of $K_0$ of an abelian category~$\cA$. It is constructed as the free abelian group on (isomorphism classes of) objects $A\in \cA$, written $[A]\in K_0(\cA)$, under the relation that if $0\to A\to B\to C\to 0$ is a short exact sequence, we have $[B]=[A]+[C]$. A short exact sequence is equivalently a cocartesian square
\begin{equation*}
\vcenter{\xymatrix@R=1em@C=1em{
A\ar[r]\ar[d]&B\ar[d]\\
0\ar[r]&C
}}
\end{equation*}
under the assumption that $A\to B$ is a monomorphism. Thus if we are to construct even $K_0$ for a derivator, it needs to admit a notion of (coherent) cocartesian squares and a zero object.

\begin{notn}\label{notn:ordinals}
For $n\in \N$, let $[n]$ denote the totally ordered set with $n+1$ elements:
\begin{equation*}
0\to 1\to\cdots\to n-1\to n
\end{equation*}
Each of these are finite free categories.
\end{notn}
\begin{notn}\label{notn:square}
Let $\square$ be the category $[1]\times[1]$, with labelling
\begin{equation*}
\xymatrix@C=1em@R=1em{
(0,0)\ar[r]\ar[d]&(1,0)\ar[d]\\
(0,1)\ar[r]&(1,1)
}
\end{equation*}
Let $i_{\ul}\colon\ul\to\square$ be the full subcategory of $\square$ lacking the element $(1,1)$. 
\end{notn}

\begin{defn}\label{defn:cocartsquare}
Let $\D$ be a left derivator and $X\in\D(\square)$. We say that $X$ is \emph{cocartesian} (\ie a pushout square) if $X$ is in the essential image of $i_{\ul,!}\colon\D(\ul)\to\D(\square)$. Otherwise put, $X$ is cocartesian if the counit $i_{\ul,!}i^\ast_\ul X\to X$ of the $(i_{\ul,!},i^\ast_\ul)$ adjunction is an isomorphism.
\end{defn}

Cocartesian squares in the sense of derivators satisfy properties analogous to our expectations from 1-category theory. Recall that there are exactly three fully-faithful functors $d_i\colon [1]\to[2]$ (which we call coface maps in the theory of simplicial sets).
\begin{prop}[Proposition~3.13,~\cite{Gro13}]\label{prop:cocart2outof3}
Let $\D$ be a left derivator. Given a coherent diagram $X\in \D([2]\times[1])$ of the form
\begin{equation*}
\vcenter{\xymatrix@R=1em@C=1em{
A\ar[r]\ar[d]&B\ar[r]\ar[d]&C\ar[d]\\
D\ar[r]&E\ar[r]&F
}}
\end{equation*}
we may obtain three coherent squares
\begin{equation*}
d_2^\ast X=\vcenter{\xymatrix@R=1em@C=1em{
A\ar[r]\ar[d]&B\ar[d]\\
D\ar[r]&E
}},\quad
d_1^\ast X=\vcenter{\xymatrix@R=1em@C=1em{
A\ar[r]\ar[d]&C\ar[d]\\
D\ar[r]&F
}},\quad
d_0^\ast X=\vcenter{\xymatrix@R=1em@C=1em{
B\ar[r]\ar[d]&C\ar[d]\\
E\ar[r]&F
}}
\end{equation*}
where the maps in $d_1^\ast X$ are the composites. Assume that the lefthand square $d_2^\ast X$ is cocartesian. Then the righthand square $d_0^\ast X$ is cocartesian if and only if the outside square $d_1^\ast X$ is cocartesian. Otherwise put: cocartesian squares compose and satisfy the usual form of 2-out-of-3.
\end{prop}

This is where Muro and Raptis obtained their domain for derivator K-theory: they considered left derivators which admit a zero object. However, we would like to be able to \emph{construct} pushouts appropriate for computing $K_0$ as above. This means coherently making cocartesian squares starting from an element in $\D(\ul)$ of the form
\begin{equation}\label{dia:cone-ell}
\vcenter{\xymatrix@R=1em@C=1em{
a\ar[r]\ar[d]&b\\
0
}}
\end{equation}
In order to construct this $\ul$-shaped diagram starting from a coherent\linebreak arrow $(a\to b)\in \D([1])$, we need more than the structure of a left derivator. 

\begin{defn}\label{defn:leftpointedder}
A prederivator $\D\colon\Dia\op\to\CAT$ is a \emph{left pointed derivator} if it is a left derivator, $\D(e)$ is pointed, and for every sieve $u\colon J\to K$, $u^\ast$ admits a right adjoint $u_\ast$ satisfying Der4R.
\end{defn}

These definitions recover everything asked by Muro-Raptis in \cite[\S2.2]{MurRap17}.

\begin{prop}[Theorem~9.8,~\cite{Col20b}]
If $\D$ is a left pointed derivator, then $\D(K)$ is pointed for every $K\in\Dia$ and $u^\ast$ is a pointed functor for every $u\colon J\to K$. Moreover, every $u_!$ is a pointed functor and left Kan extensions along cosieves admit exceptional left adjoints.
\end{prop}

The requirement that left pointed derivators admit right Kan extensions along sieves gives us the ability to construct coherent cocartesian squares. Indeed, the inclusion $i_{[1]}\colon[1]\to \ul$ is a sieve, so by Proposition~\ref{prop:extensionbyzero} we can compute that
\begin{equation}\label{dia:coneconstruction}
(a\overset f \to b)\overset{i_{[1],\ast}}\longrightarrow
\vcenter{\xymatrix@R=1em@C=1em{
a\ar[r]^-f\ar[d]&b\\
0
}}
\overset{i_{\ul,!}}\longrightarrow
\vcenter{\xymatrix@R=1em@C=1em{
a\ar[r]^f\ar[d]&b\ar[d]\\
0\ar[r]&C(f)
}}
\end{equation}
where we have named the object at the $(1,1)$ position the \emph{cone} of the (coherent) morphism $f\in \D([1])$. This means that there is a morphism of derivators $\D^{[1]}\to\D^\square$ realising the above diagram.

\begin{remark}
An alternative approach to the cone construction in Diagram~\ref{dia:coneconstruction} might be the following: from any $(a\to b)\in\D([1])$, we take its underlying diagram and add on a zero object as in Diagram~\ref{dia:cone-ell} in $\Fun(\ul,\D(e))$. Then because the category $\ul$ is finite free, we can lift the incoherent diagram back to $\D(\ul)$ and take its pushout with $i_{\ul,!}$. Unfortunately, the lift along $\dia_\ul$ is not functorial in any sort of way. We will need to refine the cone construction as a morphism of derivators $\D^{[1]}\to \D$ and this is impossible with `lift an incoherent diagram' as an intermediate step.
\end{remark}

The key example of a left pointed derivator, and indeed the motivation of the abstract definition, is the following, drawn from Corollary~2.24, Proposition~3.4, and Lemma~4.3 in \cite{Cis10}.

\begin{theorem}\label{thm:waldhausenderivator}
Let $\cW$ be a saturated Waldhausen category satisfying the cylinder axiom. Then the associated prederivator $\D_\cW\colon K\mapsto \Ho(\Fun(K,\cW))$ defined on $\Dirf$ is a (strong) left pointed derivator. Moreover, an exact functor of Waldhausen categories induces a cocontinuous morphism of the corresponding derivators. In particular, these morphisms preserve cocartesian squares and the zero object.
\end{theorem}

Recall that derivators need not be defined on all of $\Cat$, but on sub-2-categories $\Dia\subset\Cat$ satisfying some closure properties. One key example is $\Dirf$, which consists of all \emph{finite direct categories}, \ie categories whose nerve has only finitely many nondegenerate simplices. These are also called \emph{homotopy finite categories} by \cite{Arl20} and \cite{GroPonShu14b}. General Waldhausen categories will not admit arbitrary colimits and will admit no (non-empty) limits whatsoever.

For the purposes of K-theory, $\Dirf$ is an ideal domain for our left pointed derivators. Homotopical cocompleteness for all of $\Cat$ means the existence of infinite coproducts. This allows for a derivator version of the usual Eilenberg swindle on K-theory, \eg \cite[V.1.9.1]{Wei13}. Since this trick requires additivity, we will prove it below as a corollary of the main theorem. More importantly, the examples of interest are not defined on domains larger than $\Dirf$.

\begin{defn}\label{defn:derk}
Let $\Der_K$ be the 2-category with objects left pointed derivators on $\Dirf$, morphisms cocontinuous morphisms of derivators, and 2-morphisms isomodifications.
\end{defn}

\section{Derivator K-theory}

It is helpful at this point to recall Waldhausen's K-theory for a category with cofibrations and weak equivalences from~\cite[\S1.3]{Wal85}. To such a category $\cW$ we assign a simplicial object in Waldhausen categories $\mathbf{S}_\bullet\cW$, where $\mathbf{S}_n\cW$ is a subcategory of certain functors from the arrow category $\Ar[n]$ to $\cW$. Taking the wide subcategory of $\mathbf{S}_n\cW$ with only maps the weak equivalences $w\mathbf{S}_\bullet\cW$, we obtain again a simplicial Waldhausen category. Then we define K-theory as follows:
\begin{equation*}
\K(\cW):=\Omega|N_\bullet w\mathbf{S}_\bullet\cW|,
\end{equation*}
the loop space of the (diagonal) geometric realisation of the bisimplicial set given by the nerve.

In~\cite{MurRap17}, Muro and Raptis improved upon a construction of Garkusha in~\cite{Gar05} which generalises Waldhausen's $\mathbf{S}_\bullet$ construction. First, we can restate the $\mathbf{S}_\bullet$ construction in the language of derivators. To help with notation, for any ordinal category $[n]$, let the elements of its arrow category $\Ar[n]$ be written $(i, j)$ for $i\to j$. 

\begin{defn}
Let $\D$ be a left pointed derivator. Define $\mathbf{S}_n\D$ to be the full subcategory of $\D(\Ar[n])$ of objects $X$ such that:
\begin{enumerate}
\item For every $0\leq i \leq n$, $(i, i)^\ast X\in\D(e)$ is a zero object.
\item For every fully faithful inclusion $\iota\colon \square\to\Ar[n]$, the object $\iota^\ast X\in\D(\square)$ is cocartesian.
\end{enumerate}
In (2), it suffices to check only the inclusions such that $\iota(0,1)=(i, i)$ by~\cite[Proposition~3.13]{Gro13}. We then define \emph{derivator K-theory} by
\begin{equation*}
\K(\D)=\Omega|N_\bullet i\Sdot\D|
\end{equation*}
where $i\mathbf{S}_n\D\subset \mathbf{S}_n\D$ is the wide subcategory consisting only of isomorphisms, in analogy with $w\mathbf{S}_n\cW$, and the geometric realization is taken diagonally.
\end{defn}

To give a few examples, first we have that $\mathbf{S}_0\D\subset\D(\Ar[0])=\D(e)$ is trivial: it has only one object (up to isomorphism) and property (1) above requires it to be a zero object of $\D(e)$. The category $\mathbf{S}_1\D \subset\D(\Ar[1])$ is slightly more interesting: it is a staircase with one potentially nonzero object:
\begin{equation*}
\vcenter{\xymatrix@R=1em@C=1em{
0\ar[r]&a\ar[d]\\
&0
}}
\end{equation*}
There are no fully faithful inclusions $\square\to\Ar[1]$, so there is nothing else to require. We can see that $\mathbf{S}_1\D\cong\D(e)$ as a category, an observation we will use later. The first interesting category is $\mathbf{S}_2\D\subset\D(\Ar[2])$, whose objects have the form
\begin{equation*}
\vcenter{\xymatrix@R=1em@C=1em{
0\ar[r]&a\ar[r]^-f\ar[d]&b\ar[d]^-g\\
&0\ar[r]&c\ar[d]\\
&&0
}}
\end{equation*}
with the requirement that the square is cocartesian.

This is where we see $K_0(\D)$ being encoded: the zero simplices of $K(\D)=\Omega|N_\bullet i\Sdot\D|$ come from $\mathbf{S}_1\D$ (because $\Omega$ gives a dimension shift) and these zero simplices are identified in $\pi_0 K(\D)$ due to the existence of a path, \ie an element of $\mathbf{S}_2\D$, relating them. Thus three objects appearing a cocartesian square in $\D(\square)$ leads to a relation on the homotopy classes of zero simplices in $\pi_0 K(\D)$. A rigorous argument for triangulated derivators can be found in \cite[\S6]{Mal07}.

We said above that in order for a morphism $\Phi\colon\D\to\E$ to induce a map on K-theory, it needs to preserve cocartesian squares and the zero object. However, there is another problem. If we have a cocontinuous morphism $\Phi\colon\D\to\E$ that is not \emph{strict}, then $\Phi$ may not induce a map of simplicial sets $\Sdot\D\to\Sdot\E$. If we take (for example) the face map $d_i\colon[n]\to[n+1]$, then we obtain a diagram
\begin{equation*}
\xymatrix@C=4em{
\mathbf{S}_{n+1}\D\ar[d]_{d_i^\ast}\ar[r]^{\Phi_{\Ar[n+1]}}&\mathbf{S}_{n+1}\E\ar[d]^{d_i^\ast}\ar@{}[dl]|\swtrans\ar@{}[dl]<-1.75ex>|{\gamma_{d_i}^\Phi}\\
\mathbf{S}_{n}\D\ar[r]_{\Phi_{\Ar[n]}}&\mathbf{S}_n\E
}
\end{equation*}
But this diagram commutes only up to natural isomorphism, so $\Phi$ does not give us an honest natural transformation of the bisimplicial sets $N_\bullet i\mathbf{S}_\bullet\D\to N_\bullet i\mathbf{S}_\bullet\E$. However, we have the following proposition to aid us.
\begin{prop}[Proposition~10.14,~\cite{CisNee08}]\label{prop:strictification} Let $\Phi\colon\D\to\E$ be a morphism of prederivators. Then there exists a prederivator $\widetilde{\D}$, a strict equivalence of derivators $\Pi_\Phi\colon\widetilde{\D}\to\D$, and a strict morphism $\widetilde{\Phi}\colon\widetilde{\D}\to\E$ such that the following diagram commutes:
\begin{equation*}
\xymatrix@R=1em{
&\widetilde{\D}\ar[dl]_{\Pi_\Phi}^(0.4){\sim}\ar[dr]^{\widetilde \Phi}&\\
\D\ar[rr]_\Phi&&\E
}
\end{equation*}
\end{prop}
We name the equivalence $\Pi_\Phi$ because it is some sort of projection, though we will not need the precise formula. If $\Phi$ is cocontinuous, $\widetilde\Phi$ is also cocontinuous because it is the composition of cocontinuous morphisms. The following corollary is found at~\cite[Corollary~10.19]{CisNee08} or~\cite[Remark~5.1.4]{MurRap17}. 
\begin{cor}\label{cor:strictification}
Any cocontinuous morphism of left pointed derivators $\Phi\colon\D\to \E$ gives rise to a map on derivator K-theory $K(\Phi)\colon K(\D)\to K(\E)$ in $\Ho\cS$, the homotopy category of spaces. Moreover, this association is functorial, in the sense that we have a 1-functor $\Der_\text{K}\to \Ho\cS$ after inverting isomodifications in $\Der_\text{K}$ to obtain a 1-category.
\end{cor}
An immediate consequence is that equivalent left pointed derivators have homotopy equivalent derivator K-theories.

It will be essential below to have this functoriality available \emph{before} passing to the loop space, and to have a fixed notation for it, so we record it now.

\begin{notn}\label{notn:realisation}
For a left pointed derivator $\D$ we write
\begin{equation*}
\|\D\|:=\big|N_\bullet i\Sdot\D\big|
\end{equation*}
for the (diagonal) geometric realisation of the nerve of the simplicial category $i\Sdot\D$, so that $\K(\D)=\Omega\|\D\|$ by definition. If $\Phi\colon\D\to\E$ is a \emph{strict} cocontinuous morphism, then $\Sdot\Phi$ is a map of simplicial categories and we write $\|\Phi\|\colon\|\D\|\to\|\E\|$ for the induced map of spaces.
\end{notn}

\begin{lemma}\label{lemma:functoriality}
The assignment $\D\mapsto\|\D\|$, $\Phi\mapsto\|\Phi\|$ extends to a functor
\begin{equation*}
\|-\|\colon\Der_\text{\emph{K}}\big[\text{\emph{isomodifications}}\inv\big]\longrightarrow\Ho\cS
\end{equation*}
defined on all cocontinuous morphisms of left pointed derivators. Concretely, for cocontinuous $\Phi\colon\D\to\E$ choose strictification data $\Pi_\Phi\colon\widetilde\D\to\D$ and $\widetilde\Phi\colon\widetilde\D\to\E$ as in Proposition~\ref{prop:strictification}; then $\|\Pi_\Phi\|$ is a homotopy equivalence and
\begin{equation}\label{eq:normdef}
\|\Phi\|:=\|\widetilde\Phi\|\circ\|\Pi_\Phi\|\inv\colon\|\D\|\longrightarrow\|\E\|,
\end{equation}
independently of the choice of strictification. Moreover $\|\Psi\Phi\|=\|\Psi\|\circ\|\Phi\|$, and $\|\Phi\|=\|\Phi'\|$ whenever $\Phi\cong\Phi'$ as morphisms of derivators.
\end{lemma}
\begin{proof}
For the first claim, $\Pi_\Phi$ is a strict equivalence of derivators, so each functor $(\Pi_\Phi)_{\Ar[n]}$ is an equivalence of categories; being an equivalence it preserves and reflects zero objects and cocartesian squares, so it restricts to an equivalence $\Sn\widetilde\D\to\Sn\D$ for every $n$. Hence $N_\bullet i\Sdot\Pi_\Phi$ is a levelwise weak equivalence of bisimplicial sets and therefore a weak equivalence on diagonal realisations. The remaining assertions are~\cite[Corollary~10.19]{CisNee08}. They may also be read off from~\cite[Proposition~5.1.3 and Remark~5.1.4]{MurRap17}, where exactly this functor -- before looping, in the notation of \emph{ibid.} the space $K(\D)=|N_\bullet i\Sdot\D|$ -- is constructed as a simplicial functor on left pointed derivators and strict morphisms, and is shown to descend to the localisation at isomodifications.
\end{proof}
Applying $\Omega$ recovers Corollary~\ref{cor:strictification}, since $\K(\Phi)=\Omega\|\Phi\|$.

We stress what Lemma~\ref{lemma:functoriality} does \emph{not} say. If $\Phi$ is not strict then there is no map of simplicial categories $\Sdot\D\to\Sdot\E$, and hence no map of simplicial sets $N_\bullet i\Sdot\D\to N_\bullet i\Sdot\E$; the symbol $\|\Phi\|$ denotes a morphism of $\Ho\cS$ obtained from a zig-zag by inverting a weak equivalence. Every statement below involving a non-strict morphism is to be read in $\Ho\cS$ in this way. The practical consequence is that we may never evaluate such a map on objects; instead we compose morphisms of derivators \emph{first}, in $\Der_\text{K}$, and apply $\|-\|$ only at the end. This is the discipline followed in Section~\ref{sec:additivity} below.

\begin{notn}\label{notn:brackets}
Let $X,Y$ be two pointed spaces. We write $[X,Y]$ for the set of pointed homotopy classes of maps $X\to Y$, and $[f]$ for the class of a map $f$. A morphism of $\Ho\cS$, such as $\|\Phi\|$ or $\K(\Phi)$, already is such a class, so for these we use $g$ and $[g]$ interchangeably, writing the brackets where an argument passes between maps and their classes.
\end{notn}

There are some existing results that are worth collecting. Maltsiniotis in~\cite[\S6]{Mal07} proved that, if $\D$ is a triangulated derivator, then $K_0(\D(e))$ of the underlying triangulated category is equivalent to $K_0(\D)$. He also established a comparison map from Quillen's K-theory of an exact category $\cE$ to the derivator K-theory of $D^b(\cE)$, which was subsequently generalized to Waldhausen categories by Garkusha in~\cite{Gar05}.

Specifically, let $\cW$ be a saturated Waldhausen category satisfying the cylinder axiom as in Theorem~\ref{thm:waldhausenderivator} so that it gives rise to a left pointed derivator $\D_\cW$ (hereafter we will leave these assumptions on $\cW$ implicit). Then we have an obvious map
\begin{equation*}
\Fun(\Ar[n],\cW)\to\D_\cW(\Ar[n])=\Ho(\Fun(\Ar[n],\cW))
\end{equation*}
sending a diagram to its homotopy class. This map restricts to $\mathbf{S}_n\cW\to\mathbf{S}_n\D_\cW$, sends weak equivalences to isomorphisms, and behaves well with respect to the simplicial structures, so we obtain a map of spaces
\begin{equation}\label{eq:comparisonmap}
\mu\colon \K(\cW)\to \K(\D_\cW)
\end{equation}

Maltsiniotis' proof is easily rewritten to imply that $\mu_0:=\pi_0\mu$ is an isomorphism, and Muro in~\cite{Mur08} proved that $\mu_1$ is an isomorphism as well. Muro's techniques are a bit ad hoc, but recent work of Raptis \cite[Theorem~5.5]{Rap19} proves that the comparison map $\mu$ is  {2-connected}, recovering Muro's result on $\mu_1$ and proving that $\mu_2$ is surjective. Maltsiniotis conjectured that $\mu$ should be a weak homotopy equivalence in the case of $\cW=C^b(\cE)$, the category of bounded chain complexes on an exact category, a special case of when $\D_\cW$ is triangulated. The same question can be asked for any triangulated derivator, and indeed for more general flavours of derivator altogether.

Unfortunately, the conjecture fails almost totally. Muro and Raptis together in~\cite{MurRap11} show that $\mu$ is not always an equivalence for triangulated derivators, giving an example arising from perfect complexes of dg-modules. In that same work, the authors use computations of Schlichting in~\cite{Sch02} (which uses stable module categories) to prove that any K-theory of derivators invariant under derivator equivalence cannot both satisfy agreement and send Verdier localizations of triangulated derivators to homotopy fibrations in K-theory. Such a localization theorem was also conjectured by Maltsiniotis. Raptis in \cite{Rap19} further conjectures that $\mu$ should not be more than 2-connected in great generality.

One positive result is a theorem of Cisinski and Neeman~\cite{CisNee08} that derivator K-theory of triangulated derivators satisfies an additivity theorem, to be made more explicit shortly. Muro and Raptis in their second paper on derivator K-theory~\cite{MurRap17} asked whether additivity holds in the more general context of left pointed derivators. The positive answer to this question occupies the next section.

\section{Additivity}\label{sec:additivity}

Rather than approach the problem as Cisinski and Neeman did in~\cite{CisNee08} using Neeman's method of regions, we will adapt a method of Grayson from \cite{Gra11}. Throughout, let $\D$ be a left pointed derivator defined on $\Dia=\Dirf$. We begin by setting up the notation for the main theorem. We adapt this first definition from~\cite[Definition~11.7]{CisNee08}.

\begin{defn}\label{defn:cofderivator}
Let $\D$ be a left pointed derivator. We define the \emph{category of cofiber sequences} $D_\text{cof}$ in $\D$ to be the full subcategory of $\D(\square)$ whose objects are $X\in \D(\square)$ such that $X$ is cocartesian and $(0,1)^\ast X$ is a zero object.
\end{defn}

\begin{lemma}
There is a left pointed derivator $\Dcof$ and a fully-faithful morphism of derivators $\Dcof\to \D^\square$ whose base $\Dcof(e)$ is the category $D_\text{cof}$ defined above.
\end{lemma}
\begin{proof}
We begin by defining $\Dcof$ on objects in $\Dirf$. For a diagram category $K$, we let $\Dcof(K)$ be the full subcategory of $\D(\square\times K)$ consisting of those $X\in \D(\square\times K)$ such that, under the identification $\D(\square\times K)=\D^K(\square)$, $X$ is a cocartesian square in $\D^K(\square)$ such that $(0,1)^\ast X=0\in \D^K(e)=\D(K)$. This mimics Definition~\ref{defn:cofderivator} as closely as possible, but is perhaps not the best characterization of this category.

Recall from Definition~\ref{defn:cocartsquare} that $X\in \D^K(\square)$ is cocartesian if and only if the counit $i_{\ul,!} i_{\ul}^\ast X\to X$ is an isomorphism, where $i_\ul\colon \ul\to\square$ is the inclusion of the top-right corner. Under the identification of $\D^K(\square)=\D(\square\times K)$, this counit is
\begin{equation*}
\varepsilon_X\colon (i_\ul\times \id_K)_!(i_\ul\times \id_K)^\ast X\to X.
\end{equation*}
By instead viewing this as a map in $\D(\square\times K)=\D^\square(K)$, we may apply Der2 to the left pointed derivator $\D^\square$ and conclude that the above map is an isomorphism if and only if for every $k\in K$, the map
\begin{equation*}
k^\ast \varepsilon_X\colon k^\ast (i_\ul\times \id_K)_!(i_\ul\times \id_K)^\ast X\to k^\ast X
\end{equation*}
is an isomorphism in $\D^\square(e)=\D(\square)$. But because $k^\ast$ commutes up to isomorphism with left Kan extensions and commutes strictly with pullbacks, there is a canonical isomorphism
\begin{equation*}
k^\ast (i_\ul\times \id_K)_!(i_\ul\times \id_K)^\ast \overset{\cong}{\longrightarrow}i_{\ul,!}i_{\ul}^\ast k^\ast X.
\end{equation*}
By the naturality of all morphisms involved, we can identify $k^\ast \varepsilon_X$ with the counit of the $(i_{\ul,!},i_{\ul}^\ast)$ adjunction on $\D(\square)$ applied to $k^\ast X$.

Putting all this together, we conclude that $X\in \D^K(\square)$ is cocartesian if and only if $k^\ast X\in \D(\square)$ is cocartesian for all $k\in K$. By a similar (and easier) check, the condition that $(0,1)^\ast X=0\in \D(K)$ is equivalent to $(0,1)^\ast k^\ast X=0\in \D(e)$ for all $k\in K$.

Having established this technical detail, we need to define the action of $\Dcof$ on functors in $\Dirf$. For any $u\colon J\to K$, $u^\ast\colon\D^\square(K)\to\D^\square(J)$  is cocontinuous by \cite[Proposition~2.5]{Gro13}. Therefore $u^\ast$ preserves cocartesian squares and the zero object, so restricts to $u^\ast\colon\Dcof(K)\to\Dcof(J)$. There are no additional assumptions for the action of $\Dcof$ on natural transformations in $\Dirf$, so this defines $\Dcof$ as a prederivator.

We can prove that $\Dcof$ is actually a left pointed derivator by establishing an equivalence of derivators $\D^{[1]}\to \Dcof$. This equivalence is given by the composition of Diagram~\ref{dia:coneconstruction}:
\begin{equation*}
\xymatrix{
\D^{[1]}\ar[r]^-{i_{[1],\ast}}&\D^{\ul}\ar[r]^-{i_{\ul,!}}&\Dcof\subset\D^\square
}
\end{equation*}

By definition the image of this composite consists of cocartesian squares with a zero object in the $(0,1)$ position. Since $i_{[1]}$ and $i_\ul$ are fully faithful, their left and right Kan extensions are fully faithful by \cite[Proposition~7.18]{Col20b}), hence the above composite induces an equivalence onto its image, which is precisely $\Dcof\subset\D^\square$.
\end{proof}

We will refer to $\Dcof$ as the \emph{(left pointed) derivator of cofiber sequences} in $\D$.

\begin{remark}\label{rk:rightexactmorphism}
The association $\D\mapsto \Dcof$ defines an endomorphism of $\Der_\text{K}$. Given a morphism of derivators $\Phi\colon \D\to\E$, we can consider the `shifted morphism'\linebreak $\Phi^\square\colon \D^\square\to\E^\square$. This is really the same pseudonatural transformation, but its component functors all have the form $\Phi_{K\times \square}$. For $\Phi^\square\colon\D^\square\to\E^\square$ to restrict to a morphism $\Dcof\to\Ecof$, it would have to send cocartesian squares to cocartesian squares, and it would have to send zero objects of $\D$ to zero objects of $\E$. As $\Der_\text{K}$ only contains cocontinuous morphisms of derivators, this property holds for any $\Phi$. Isomodifications need no special properties to pass to the derivators of cofiber sequences.

In the constructions below, we will construct morphisms $\Phi\colon \Dcof\to(\Dcof)^K$ for various categories $K\in\Dirf$, but often we will have to define these morphisms first as $\Phi\colon \D\to\D^K$. We may then extend $\Phi$ to a morphism $\Dcof\to(\Dcof)^K$ as long as $\Phi$ is cocontinuous.
\end{remark}

\begin{defn}
Let $\D,\E$ be left pointed derivators. We define a \emph{cofibration morphism of derivators} to be a strict cocontinuous morphism $\Xi\colon\D\to\Ecof$. To $\Xi$ we associate three strict cocontinuous morphisms $\D\to\E$
\begin{equation*}
S:=(0,0)^\ast\Xi\qquad T:=(1,0)^\ast\Xi\qquad Q:=(1,1)^\ast\Xi
\end{equation*}
and two strict cocontinuous morphisms $\alpha,\beta\colon \D\to \E^{[1]}$ given by restricting to the top and right arrows of the coherent square, respectively. Incoherently, we have
\begin{equation*}
a\in\D\mapsto
\vcenter{\xymatrix{
S(a)\ar[r]^{\alpha_{a}}\ar[d]&T(a)\ar[d]^{\beta_{a}}\\
0\ar[r]&Q(a)
}}\in\Ecof
\end{equation*}
\end{defn}
The morphisms $S,T,Q$ are named for \emph{source}, \emph{target}, and \emph{quotient}, respectively. This is a coherent version of a cofibration sequence of exact morphisms of Waldhausen categories in~\cite[p.~331]{Wal85}. We prove a similar theorem to~\cite[Proposition~1.3.2]{Wal85}.

\begin{theorem}\label{thm:additivities}
Let $\D$ and $\E$ be left pointed derivators. The following are equivalent:
\begin{enumerate}
\item The map
\begin{equation*}
\xymatrix@C=5em{
\Dcof\ar[r]^-{(0,0)^\ast\times(1,1)^\ast} &\D\times\D
}
\end{equation*}
induces a homotopy equivalence on derivator K-theory.

\item If $\Xi\colon\D\to\Ecof$ is a cofibration morphism of derivators, then there exists a homotopy
\begin{equation*}
K(T)\simeq K(S \vee Q)
\end{equation*}
where $\vee$ is the coproduct of morphisms of derivators.
\end{enumerate}
\end{theorem}
The first statement is the statement of additivity \`a la Garkusha, Maltsiniotis, and Cisinski-Neeman, first found in~\cite[Conjecture~3]{Mal07} and similar to~\cite[Proposition~1.3.2(2)]{Wal85}. The latter is a reinterpretation of~\cite[Proposition~1.3.2(4)]{Wal85}. Our proof follows the strategy set out by Waldhausen.

To expand a little on (1), it will be helpful to use that $\Sdot\D\times\Sdot\D\cong\Sdot\D^{e\sqcup e}$. First, we know that
\begin{equation*}
\D(\Ar[n])\times\D(\Ar[n])\cong\D(\Ar[n]\sqcup\Ar[n])\cong\D^{e\sqcup e}(\Ar[n])
\end{equation*}

Second, for a diagram $X\in\D^{e\sqcup e}(\Ar[n])$, the condition of being in $\mathbf{S}_n\D^{e\sqcup e}$ coincides with each projection to $\D(\Ar[n])$ being in $\mathbf{S}_n\D$. This makes it easier for us to define maps into $\Sdot\D\times\Sdot\D$; they can arise from (strict) cocontinuous morphisms into $\D^{e\sqcup e}$.

For (2), a word is needed to clarify the precise meaning of the notation $\vee$. For any left pointed derivator $\D$, we can define a non-strict morphism $\vee\colon \D\times\D\to\D$ that endows $\D$ with a coherent coproduct.

An object in $\D^{e\sqcup e}(e)$ is just a pair of objects $(a,c)$ with $a,c\in \D(e)$ with no additional data, and this identification (levelwise) gives an isomorphism between the product $\D\times\D$ and the shifted derivator $\D^{e\sqcup e}$. Recall from Notation~\ref{notn:square} the functor $i_\ul\colon \ul\to\square$. Further, let $i\colon e\sqcup e\to\ul$ be the inclusion into $(1,0)$ and $(0,1)$. Note that $i$ is a cosieve, so that $i_!$ is an extension by zero morphism.
 
The morphism $\vee\colon \D\times\D\to\D$ is defined by
\begin{equation}\label{eq:coprod}
\xymatrix{
\D\times\D\cong \D^{e\sqcup e}\ar[r]^-{i_!}&\D^{\ul}\ar[r]^-{i_{\ul,!}}&\D^{\square}\ar[r]^-{(1,1)^\ast}&\D
}
\end{equation}
All of these maps are cocontinuous, though not all are strict. The morphism of derivators $S\vee Q$ is the composition
\begin{equation*}
	\xymatrix{
		\D\ar[r]^-{(S,Q)}&\E\times\E\ar[r]^-\vee&\E
	}
\end{equation*}
which must then be strictified as in Lemma~\ref{lemma:functoriality} to obtain the map $\K(S\vee Q)=\Omega\|S\vee Q\|$ on K-theory.

\begin{remark}\label{rk:hspace}
The notation $\K(S\vee Q)$ in Theorem~\ref{thm:additivities}(2) means the map induced on K-theory by the \emph{single} morphism of derivators $S\vee Q\colon\D\to\E$ just constructed; it is not a coproduct of spaces, nor a coproduct of maps of spaces. Waldhausen writes $\K(S)\vee\K(Q)$ in the corresponding statement~\cite[Proposition~1.3.2(4)]{Wal85} and explains parenthetically what is meant; since our $\vee$ already has an unambiguous meaning inside a derivator, we prefer to leave the coproduct where it is actually performed, namely in $\E$. This remark records why the two readings agree, and in particular why one is entitled to add and cancel K-theoretic maps in the way that additivity arguments always do.

Throughout we use the notation $\|\E\|=|N_\bullet i\Sdot\E|$ of Notation~\ref{notn:realisation}, so that\newline$\K(\E)=\Omega\|\E\|$ by definition.

Since $\K(\E)$ is by construction a loop space, the set $[\K(\D),\K(\E)]$ of pointed homotopy classes of maps is a group under pointwise concatenation of loops,
\begin{equation*}
(f\cdot g)(x)(t)=\begin{cases} f(x)(2t)&t\leq 1/2\\ g(x)(2t-1)&t\geq 1/2,\end{cases}
\end{equation*}
with unit the constant map at the basepoint. No additivity is needed for this; it is immediate from the definition $\K(\E)=\Omega\|\E\|$.

The coproduct morphism $\vee=(1,1)^\ast i_{\ul,!}i_!\colon\E^{e\sqcup e}\to\E$ of Equation~\ref{eq:coprod} is cocontinuous but not strict, so by Lemma~\ref{lemma:functoriality} it induces a morphism $\|\vee\|$ in $\Ho\cS$. Combining this with the levelwise identification $\Sn\E^{e\sqcup e}\cong\Sn\E\times\Sn\E$ recorded after Theorem~\ref{thm:additivities}, we obtain a based map
\begin{equation*}
m\colon \|\E\|\times\|\E\|\cong\big\|\E^{e\sqcup e}\big\|\overset{\|\vee\|}{\longrightarrow}\|\E\|
\end{equation*}
in $\Ho\cS$. This is the point that the notation obscures: the multiplication induced by $\vee$ does not merely exist on $\K(\E)$, it exists already on the delooping $\|\E\|$, because $\vee$ acts on the whole simplicial category $i\Sdot\E$ and not just on its loop space. Looping gives
\begin{equation*}
\Omega m\colon\K(\E)\times\K(\E)\longrightarrow\K(\E),
\end{equation*}
and the canonical isomorphisms $0\vee a\cong a\cong a\vee 0$ in $\E$ show that the basepoint is a two-sided homotopy unit for $m$, hence for $\Omega m$. Thus $\Omega m$ makes $[\K(\D),\K(\E)]$ into a monoid with the same unit as the loop structure. (We avoid the letter $\mu$ here, which already denotes the comparison map of Equation~\ref{eq:comparisonmap}.)

Choose an actual based map representing $m$ and let $f,g,h,k\colon\K(\D)\to\K(\E)$ be based maps. Because $\Omega m$ acts in the $\|\E\|$-coordinate while concatenation acts in the loop coordinate, we may compute, for $x\in\K(\D)$ and $t\in[0,1]$,
\begin{equation*}
\Omega m\big(f\cdot h,\;g\cdot k\big)(x)(t)=m\Big((f\cdot h)(x)(t),\,(g\cdot k)(x)(t)\Big)=\Big(\Omega m(f,g)\cdot\Omega m(h,k)\Big)(x)(t),
\end{equation*}
both sides being $m\big(f(x)(2t),g(x)(2t)\big)$ for $t\leq 1/2$ and $m\big(h(x)(2t-1),k(x)(2t-1)\big)$ for $t\geq 1/2$. So the interchange law holds on the nose, not merely up to homotopy. Two unital binary operations on a set satisfying interchange coincide and are commutative and associative, by the Eckmann--Hilton argument~\cite{EckHil62}. Therefore the operation induced by $\Omega m$ and concatenation agree on $[\K(\D),\K(\E)]$, and this common operation makes $[\K(\D),\K(\E)]$ an \emph{abelian} group. We write it additively.

As a consequence, the morphism $S\vee Q$ is by definition $\vee\,(S,Q)$, so under the identification above $[\K(S\vee Q)]=\Omega m\big([\K(S)],[\K(Q)]\big)$, whence
\begin{equation*}
[\K(S\vee Q)]=[\K(S)]+[\K(Q)]\in[\K(\D),\K(\E)].
\end{equation*}
This is the precise sense in which additivity says `target $=$ source $+$ quotient', and it is what justifies the cancellation performed in the Eilenberg swindle of Proposition~\ref{prop:derivatorswindle}, where we subtract $[\K(\bigvee)]$ from both sides of an identity in $[\K(\D),\K(\D)]$. Below we always write $\K(S\vee Q)$, reserving $\vee$ for coproducts internal to a derivator.

The corresponding fact for Waldhausen K-theory is what licenses the parenthetical identification in~\cite[Proposition~1.3.2(4)]{Wal85}.
\end{remark}

\begin{proof}[Proof of Theorem~\ref{thm:additivities}]\ \\
\noindent$(2) \implies (1)$

The map $\rho:=(0,0)^\ast\times(1,1)^\ast$ admits a section $\sigma\colon \D^{e\sqcup e}\to\Dcof$ up to isomorphism, which gives a section up to homotopy on K-theory. The functor $\sigma$ on $\D^{e\sqcup e}(e)$ is roughly (but not precisely)
\begin{equation*}
(a,c)\mapsto
\vcenter{
\xymatrix@C=1em@R=1em{
a\ar[r]\ar[d]&a\vee c\ar[d]\\ 0\ar[r]& c
}}
\end{equation*}
where $a\vee c$ is the (coherent) coproduct in $\D(e)$ that we defined above. We will define $\sigma$ as a composite of morphisms of derivators coming from diagram functors, but we will need some diagram notation first.

Consider the category $[1]\times [2]$, with labelling
\begin{equation*}
\xymatrix@C=1em@R=1em{
(0,0)\ar[r]\ar[d]&(1,0)\ar[d]\\
(0,1)\ar[r]\ar[d]&(1,1)\ar[d]\\
(0,2)\ar[r]&(1,2)
}
\end{equation*}
Let $J$ be the full subcategory of $[1]\times[2]$ without the element $(1,2)$. Let $i_\square\colon \square\to J$ and $j\colon J\to [1]\times[2]$ be the obvious inclusions, and let $r\colon \square\to [1]\times[2]$ be the inclusion of the bottom square. Note that $i_\ul$, $i_\square$, and $j$ are sieves.
 
At the level of the derivators, the section $\sigma$ is given by
\begin{equation}
\xymatrix{
\D^{e\sqcup e}\ar[r]^-{i_!}&\D^{\ul}\ar[r]^-{i_{\ul,!}}&\D^{\square}\ar[r]^-{i_{\square,\ast}}&\D^{J}\ar[r]^-{j_!}&\D^{[1]\times[2]}\ar[r]^-{r^\ast}&\D^{\square}\supset\Dcof
}
\end{equation}
The first part of this composition is the same as the definition of $\vee$ in Equation~\ref{eq:coprod}. All of these maps are cocontinuous, though not all are strict. But once we strictify the map, $\sigma$ gives rise to a well-defined map in $\Ho\cS$.

For $(a,c)\in\D^{e\sqcup e}$, $\sigma(a,c)$ is explicitly
\begin{equation*}
(a,c)\mapsto
\vcenter{\xymatrix@C=1em@R=1em{0\ar[r] \ar[d] &c\\a}}\mapsto
\vcenter{\xymatrix@C=1em@R=1em{0\ar[r] \ar[d] &c\ar[d]\\a\ar[r]& a\vee c}}\mapsto
\vcenter{\xymatrix@C=1em@R=1em{0\ar[r] \ar[d] &c\ar[d]\\a\ar[r]\ar[d]& a\vee c\\0}}\mapsto
\vcenter{\xymatrix@C=1em@R=1em{0\ar[r] \ar[d] &c\ar[d]\\a\ar[r]\ar[d]& a\vee c\ar[d]\\0\ar[r]&c'}}\mapsto
\vcenter{\xymatrix@C=1em@R=1em{a\ar[r]\ar[d]& a\vee c\ar[d]\\0\ar[r]&c'}}
\end{equation*}
By construction, the image of this composite lands in $\Dcof\subset\D^\square$, as the square $r^\ast j_! X$ is easily shown to be cocartesian using Proposition~\ref{prop:cocart2outof3} for any $X\in\D^J$, and $(1,0)^\ast r^\ast j_! X=0$ as long as $X$ is in the image of $i_{\square,\ast}$ (as is our case). Note further that the composite map $c\to a\vee c\to c'$ is an isomorphism. Because the outside square
\begin{equation*}
\vcenter{\xymatrix@R=1em@C=1em{
0\ar[r]\ar[d]&c\ar[d]\\
0\ar[r]&c'
}}
\end{equation*}
is cocartesian and the lefthand map is an isomorphism, the righthand map is also an isomorphism by \cite[Proposition~3.12(2)]{Gro13}.

All this data together gives us the section $\sigma\colon\D^{e\sqcup e}\to\Dcof$. Hence we obtain an isomodification $\id_{\D^{e\sqcup e}}\Rightarrow\rho\sigma$, as the canonical isomorphism $c\to c'$ gives rise to an isomorphism $(a,c)\to\rho\sigma(a,c)=(a,c')$ natural in $(a,c)\in\D^{e\sqcup e}$. On K-theory (after strictifying the non-strict morphisms and passing to $\Ho\cS$), this gives an equality of homotopy classes $K(\rho\sigma)= K(\id_{\D^{e\sqcup e}})$. Therefore it suffices to construct a homotopy in the reverse direction, \ie $K(\sigma\rho)\simeq K(\id_{\Dcof})$.

To that end, we use our assumption. We will construct a cofibration morphism of derivators $\Xi\colon \Dcof\to (\Dcof)_\text{cof}$ such that $S\vee Q\cong\sigma\rho$ and $T=\id_{\Dcof}$. As a reminder, by $S\vee Q$ we mean the coproduct of the two morphisms $S,Q$ given by the following composition:
\begin{equation}\label{eq:coprodmap}
\vcenter{\xymatrix{
\Dcof\ar[r]^-{S\times Q}& \Dcof\times \Dcof\cong \Dcof^{e\sqcup e}\ar[r]^-{\vee}&\Dcof
}}
\end{equation}
where $\vee$ is defined in Equation~\ref{eq:coprod}. For an object $X\in \Dcof$, this is explicitly
\begin{equation*}
X\mapsto (S(X),Q(X))\mapsto
\vcenter{\xymatrix@C=1em@R=1em{0\ar[r] \ar[d] &Q(X)\\S(X)}}\mapsto
\vcenter{\xymatrix@C=1em@R=1em{0\ar[r] \ar[d] &Q(X)\ar[d]\\S(X)\ar[r]& S(X)\vee Q(X)}}\mapsto
S(X)\vee Q(X)
\end{equation*}

Our morphism $\Xi\colon\Dcof\to(\Dcof)_\text{cof}$ will have the form
\begin{equation}\label{dia:cofibreseq}
\vcenter{\xymatrix@C=1em@R=1em{
a\ar[r]^{f}\ar[d]&b\ar[d]^{g}\\
0\ar[r]&c
}}
\mapsto
\vcenter{\xymatrix@R=1em@C=1em{
a\ar[r]^{=}\ar[d]&a\ar[d]\ar@{}[d]<1ex>_{}="A0"&&a\ar[d]\ar@{}[d]<-1ex>_{}="A1"\ar[r]^f&b\ar[d]^g\\
0\ar[r]\ar@{}[r]<-1ex>_{}="B0"&0&&0\ar[r]\ar@{}[r]<-1ex>_{}="C0"&c\\
&&&&\\
0\ar[r]\ar@{}[r]<1ex>^{}="B1"\ar[d]&0\ar[d]\ar@{}[d]<1ex>^{}="D0"&&0\ar[r]\ar@{}[r]<1ex>^{}="C1"\ar[d]\ar@{}[d]<-1ex>_{}="D1"&c\ar[d]^=\\
0\ar[r]&0&&0\ar[r]&c 
\ar "A0";"A1" ^{\id_a,f} \ar "B0";"B1" \ar "C0";"C1" ^{g,\id_c} \ar "D0";"D1"
}}
\end{equation}
where all squares commute.

We can accomplish this as the pullback along a single functor in $\Dirf$. In the diagram above of shape $\square\times\square$, let the first $\square$ denote the outer square coordinates and the second the inner coordinates. For example, the entry $c$ in the top right is at $(1,0,1,1)$ We now define $\xi\colon\square\times\square\to\square$ by
\begin{equation*}
\xi(a_1,b_1,a_2,b_2)=\begin{cases}
(0,0)&(a_1,b_1,a_2,b_2)=(0,0,0,0),(0,0,1,0),(1,0,0,0)\\
(1,0)&(a_1,b_1,a_2,b_2)=(1,0,1,0)\\
(1,1)&(a_1,b_1,a_2,b_2)=(1,0,1,1),(1,1,1,0),(1,1,1,1)\\
(0,1)&\text{otherwise}
\end{cases}
\end{equation*}
In plain language, we make this definition so that $\xi$ behaves on objects as sketched in Diagram~\ref{dia:cofibreseq}, and it is also the appropriate functor for the maps. For example, consider the map $(1,0,0,0)\to(1,0,1,0)$ in $\square\times\square$. Diagram~\ref{dia:cofibreseq} says that we want
\begin{equation*}
(1,0,0,0)^\ast\xi^\ast X\to (1,0,1,0)^\ast\xi^\ast X = a\overset{f}{\longrightarrow} b
\end{equation*}
We see that $\xi(1,0,0,0)=(0,0)$ and $\xi(1,0,1,0)=(1,0)$, and thus the natural transformation $(\xi(1,0,0,0))^\ast\Rightarrow(\xi(1,0,1,0))^\ast$ induces the map $f\colon a\to b$ when applied to $X\in\Dcof$. Checking that we also have $g$ and identities where required can be done similarly.

We define a strict cocontinuous morphism of derivators $\Xi:=\xi^\ast\colon\D^\square\to\D^{\square\times\square}$. By construction, assuming we restrict our domain to $\Dcof$ (as illustrated), the image will be a global cocartesian square in $\Dcof(\square)$ with zero in the bottom-left corner, so we obtain the required (strict cocontinuous) morphism $\Xi\colon\Dcof\to(\Dcof)_\text{cof}$.

Our assumption gives us that $K(T)\simeq K(S\vee Q)$, and clearly $T=\id_{\Dcof}$. We can obtain a zig-zag of isomodifications connecting $\sigma\rho$ to $S\vee Q$ by the following process. Because $i_\ul^\ast \sigma\rho$ and $i_\ul^\ast(S\vee Q)$ are equal on the nose, for every $X\in \Dcof$ we have a diagram which commutes up to isomorphism
\begin{equation*}
\vcenter{\xymatrix{
i_{\ul,!}i_\ul^\ast\sigma\rho(X)\ar[r]^-\cong\ar[d]_-{\id_{\ul,!}(\id)}& \sigma\rho(X)\\
i_{\ul,!}i_\ul^\ast(S\vee Q)(X)\ar[r]_-\cong& (S\vee Q)(X)
}}
\end{equation*}
where the horizontal maps are the counit of the $(i_{\ul,!},i_\ul^\ast)$ adjunction. This zig-zag gives rise to an equality of homotopy classes of maps $K(\sigma\rho)$ and $K(S\vee Q)$ as required, which proves additivity in the historical sense for derivator K-theory.
\ \\

\noindent $(1) \implies (2)$

First, consider the two maps $(1,0)^\ast\colon\Ecof\to \E$ and $\overline{\rho}\colon\Ecof\to \E^{e\sqcup e}\to\E$, where $\overline\rho$ is defined to be the composite cocontinuous morphism $(1,1)^\ast i_{\ul,!}i_!\rho$ which computes the coproduct of $(0,0)^\ast X$ and $(1,1)^\ast X$ for any $X\in\Ecof$ (see Equation~\ref{eq:coprod}).

We claim that these maps induce homotopic maps on K-theory. If we precompose with the map $\sigma\colon\E^{e\sqcup e}\to\Ecof$, it is immediate that $(1,0)^\ast \sigma$ and $\overline{\rho}\sigma$ are (canonically) isomorphic, as they both compute the coproduct of $(a,c)\in\E^{e\sqcup e}$. Thus if $\sigma$ is a homotopy equivalence on K-theory, $(1,0)^\ast$ and $\overline\rho$ are still homotopic. But by our assumption~(1),~$\sigma$ is a section of the homotopy equivalence $\rho$, so it too is a homotopy equivalence.

Statement (2) then follows immediately by precomposing these two homotopic maps by any cofibration morphism $\Xi\colon\D\to\Ecof$, which yields
\begin{equation*}
K(\overline{\rho}\, \Xi)\cong K(S\vee Q)\simeq K(T)=K((1,0)^\ast \Xi).
\end{equation*}
\end{proof}

This new reformulation of the additivity theorem produces a proof that differs greatly from~\cite{CisNee08} and~\cite{Gar05}, the latter of which includes a gap which seems irreparable, see~\cite[\S6.3]{MurRap17}. We will now prove Theorem~\ref{thm:additivities}(2) in the remainder of this section, following a strategy given by Grayson in~\cite{Gra11}. Grayson's paper proves the (classical) additivity theorem for Waldhausen K-theory using an explicit combinatorial homotopy. We are able to take this diagram-flavored argument and make it homotopy-coherent enough for application to derivator K-theory.

Let $Y$ be a simplicial set. Without any loss of generality, we may\linebreak extend $Y\colon\Delta^\text{op}\to\mathbf{Sets}$ to $Y\colon\Ord^\text{op}\to\mathbf{Sets}$, where $\Ord$ is the category of\linebreak nonempty finite totally ordered sets with order-preserving maps. For $A\in\Ord$, we let $Y(A):=Y([n])$, where~$[n]$ is the unique element of $\Delta\subset\Ord$ isomorphic to~$A$. We do this in order to introduce a binary operation $\ast$ on $\Ord$, which otherwise would cause us problems. We let $A\ast B$ be concatenation, that is,
\begin{equation*}
A\ast B:=(\{0\}\times A)\cup(\{1\}\times B)\subset [1]\times(A\cup B)
\end{equation*}
with lexicographical ordering. This results in each element of $A$ being smaller than each element of $B$, but within $A$ and $B$ the ordering does not change.

\begin{defn}
Let $Y$ be a simplicial set (on $\Ord$). The \emph{two-fold edge-wise subdivision} $\operatorname{sub}_2\!Y$ of $Y$ is the simplicial set defined by $\operatorname{sub}_2\!Y(A):=Y(A\ast A)$.
\end{defn}
There is a natural homeomorphism $|\operatorname{sub}_2\!Y|\to|Y|$ defined in~\cite[\S4]{Gra89} whose construction we do not recall here. The important thing to note is that we do not change the homotopy type (or even homeomorphism type) of our simplicial set by subdividing.

Now we can begin to bring derivators back into the conversation. Let $\Phi,\Psi\colon\D\to\E$ be two strict cocontinuous morphisms of derivators. These induce morphisms of simplicial categories $\Sdot \Phi,\Sdot \Psi\colon\Sdot\D\to\Sdot\E$. We want an operation $\nabla_{\Phi,\Psi}$ which applies $\Phi$ to the first half of a subdivided simplex, $\Psi$ to the second, and takes the coproduct of the results. We first describe it on objects, and then say precisely what kind of map it is -- for it is \emph{not} a map of simplicial categories.

The totally ordered set $A\ast A$ has two full subcategories $i_0,i_1\colon A\to A\ast A$, given by~$a\mapsto(0,a),(1,a)$ respectively. These extend to functors on the arrow categories\linebreak $\Ar(A)\to\Ar(A\ast A)$, and give $i_0^\ast,i_1^\ast\colon\D(\Ar(A\ast A))\to\D(\Ar(A))$. Since restriction morphisms are strict and cocontinuous, $i_0^\ast$ and $i_1^\ast$ define morphisms of simplicial categories $\SD{A\ast A}\D\to\SD{A}\D$, where we adopt the notation $\SD{A}\D=\Sn\D$ for the unique~$[n]$ such that $A\cong [n]$.

Consider an object $X\in\operatorname{sub}_2\!\SD A \D=\SD {A\ast A} \D$. Then let
\begin{equation*}
\nabla_{\Phi,\Psi}(X):=\Phi(i_0^\ast(X))\vee \Psi(i_1^\ast(X))
\end{equation*}
This indeed lands in $\SD A \E$ as the coproduct of any two cocartesian squares is again cocartesian. We have, in essence, doubled $\Sdot\D$ and applied $\Phi$ to the first half and $\Psi$ to the second, then taken the coproduct of the results.

The formula above does \emph{not} define a map of simplicial categories, and since the symbol $\nabla_{\Phi,\Psi}$ will be used repeatedly we say once and for all what it does denote.

The pair $(\Phi i_0^\ast,\Psi i_1^\ast)$ \emph{is} an honest map of simplicial categories
\begin{equation}\label{eq:pairmap}
(\Phi i_0^\ast,\Psi i_1^\ast)\colon \operatorname{sub}_2\!\Sdot\D\longrightarrow \Sdot\E^{e\sqcup e},
\end{equation}
because both components are strict and the identification $\Sdot\E\times\Sdot\E\cong\Sdot\E^{e\sqcup e}$ recorded after Theorem~\ref{thm:additivities} is levelwise an isomorphism of categories. What fails to be strict is only the coproduct $\vee=(1,1)^\ast i_{\ul,!}i_!\colon \E^{e\sqcup e}\to\E$ of Equation~\ref{eq:coprod}, which therefore induces no map of simplicial categories $\Sdot\E^{e\sqcup e}\to\Sdot\E$ whatsoever. It does induce a morphism $\|\vee\|\colon\|\E^{e\sqcup e}\|\to\|\E\|$ in $\Ho\cS$, by Lemma~\ref{lemma:functoriality}, and we define
\begin{equation}\label{eq:strictifiedcoproduct}
\nabla_{\Phi,\Psi}:=\|\vee\|\circ\big|N_\bullet i\,(\Phi i_0^\ast,\Psi i_1^\ast)\big|\colon\;\big|\operatorname{sub}_2\!N_\bullet i\Sdot\D\big|\longrightarrow \|\E\|
\end{equation}
\emph{as a morphism in $\Ho\cS$}. Unwinding Equation~\ref{eq:normdef}, this is the zig-zag of maps of simplicial categories
\begin{equation*}
\vcenter{\xymatrix@C=2.5em{
\operatorname{sub}_2\!\Sdot\D\ar[r]^-{(\Phi i_0^\ast,\Psi i_1^\ast)}&\Sdot\E^{e\sqcup e}&\Sdot\widetilde{\E^{e\sqcup e}}\ar[l]_-{\Sdot\Pi_\vee}\ar[r]^-{\Sdot\widetilde\vee}&\Sdot\E
}}
\end{equation*}
with the wrong-way arrow inverted after realisation, where $\Pi_\vee$ and $\widetilde\vee$ are strictification data for $\vee$.

This is the only meaning we ever attach to the symbol $\nabla_{\Phi,\Psi}$. In particular, an assertion that some construction `agrees with $\nabla_{\Phi,\Psi}$' is always an assertion about morphisms in $\Ho\cS$, and the objectwise formula $X\mapsto \Phi(i_0^\ast X)\vee\Psi(i_1^\ast X)$ describes it only up to canonical isomorphism, which is all that survives to $\Ho\cS$. The same convention will govern the morphism $\Theta$ of Proposition~\ref{prop:homotopy1} below.

We record one property of this definition for later use: $\nabla_{\Phi,\Psi}$ is unchanged if $\Phi$ or $\Psi$ is replaced by an isomorphic strict morphism. Indeed, an isomodification between strict morphisms commutes with every pullback, hence with the face and degeneracy maps of $\Sdot$, so it induces a natural isomorphism between the corresponding maps of simplicial categories, and naturally isomorphic functors induce homotopic maps on nerves.

Now, we need to construct a cylinder object for $\Sdot\D$ that will allow us to use $\nabla_{\Phi,\Psi}$ as a replacement for $\Phi\vee \Psi$. We need three definitions to get us there.

\begin{defn}[Definition~2, \cite{Gra11}]
For $A,B\in\Ord$, let $A\ltimes B$ be the set $A\times B$ with lexicographic ordering. That is, $(a,b)\leq (a',b')$ if and only if $a<a'$ or $a=a'$ and $b\leq b'$. Note that the map $A\ltimes B\to A$ is order-preserving, hence a morphism in $\Ord$, but the other `projection' $A\ltimes B\to B$ is generally not.
\end{defn}

\begin{defn}[Definition~3, \cite{Gra11}]
Given two maps $\phi\colon A\to C$ and $s\colon B\to C$ in $\Ord$, define $\phi^{-1}(s)\in\Ord$ to be the subset of $B\ltimes A$ given by $\{(b,a):s(b)=\phi(a)$\}.
\end{defn}

\begin{defn}[Definition~4, \cite{Gra11}]
Let $s\colon [2]\to[1]$ be the morphism defined by $s(0)=0$ and $s(1)=s(2)=1$. For any simplicial set $Y$, define a new simplicial set $IY$ by $IY(A):=\{(\phi,y):\phi\colon A\to[1],y\in Y(\phi^{-1}(s))\}$. The definition of $IY$ on morphisms in $\Ord$ extends by naturality.
\end{defn}

\begin{remark}\label{rk:usefulcylinder}
To see how $IY$ is a useful object, notice that
\begin{equation*}
\phi^{-1}(s)=\phi^{-1}(0)\ast\phi^{-1}(1)\ast\phi^{-1}(1)
\end{equation*}
In particular, if we let $\phi=0$ be the constant map, then $\phi^{-1}(s)=A$. The constant map $\phi=1$ gives $\phi^{-1}(s)=A\ast A$. The simplicial subset of $IY$ at $\phi=0$ is isomorphic to~$Y$, and the simplicial subset of $IY$ at $\phi=1$ is isomorphic to $\operatorname{sub}_2\!Y$. Any other morphism $\phi\colon A\to [1]$ gives a totally ordered set $\phi^{-1}(s)$ interpolating between these two endpoints.
\end{remark}

\begin{lemma}[Lemma~6, \cite{Gra11}]\label{lemma:cylinder}
There is a homeomorphism $|IY|\to |\Delta^1|\times|Y|$ that restricts to homeomorphisms $|Y|\to \{0\}\times|Y|$ and $|\operatorname{sub}_2 Y|\to \{1\}\times Y$ on the endpoints $\phi=0,1$ of $IY$.
\end{lemma}

We do not include the proof because nothing is changed in the context of derivators. This shows that $IY$ is indeed a cylinder object for $Y$, so we may prove the main proposition.

\begin{prop}\label{prop:homotopy1}
Let $\Xi\colon\D\to\Ecof$ be a cofibration morphism of derivators, and let~$S,T,Q\colon\D\to\E$ be the corresponding morphisms of derivators. Then there is a morphism
\begin{equation*}
\Theta\colon \big|I N_\bullet i\Sdot\D\big|\longrightarrow\|\E\|
\end{equation*}
in $\Ho\cS$ whose restriction to the end $\phi=0$ of the cylinder is $\|T\|$ and whose restriction to the end $\phi=1$ is $\nabla_{Q,S}$.
\end{prop}

Two comments before the proof. First, the construction of $IY$ from a simplicial set $Y$ works equally well for any sort of simplicial object, including simplicial categories. Second, exactly as for $\nabla_{Q,S}$, the morphism $\Theta$ is \emph{not} a map of simplicial categories, and the proposition does not claim that it is: it is a morphism of $\Ho\cS$ assembled from Lemma~\ref{lemma:functoriality}, and that is all we record.

It is worth saying at once why a morphism of $\Ho\cS$ suffices, since this is the only use we make of the cylinder. By Lemma~\ref{lemma:cylinder} there is a homeomorphism $|I N_\bullet i \Sdot\D|\cong|\Delta^1|\times \|\D\|$ under which the end $\phi=0$ becomes the section $\iota_0$ at $0$, and the end $\phi=1$ becomes the section $\iota_1$ at $1$ after the subdivision homeomorphism $\big|\operatorname{sub}_2\!N_\bullet i\Sdot\D\big|\cong\|\D\|$. As $\iota_0$ and $\iota_1$ are both homotopy inverse to the projection $|\Delta^1|\times\|\D\|\to\|\D\|$, they are homotopic to one another; hence for \emph{any} morphism $\Theta$ of $\Ho\cS$ as in the proposition,
\begin{equation*}
\big[\|T\|\big]=[\Theta\iota_0]=[\Theta\iota_1]=[\nabla_{Q,S}]
\end{equation*}
in $\big[\|\D\|,\|\E\|\big]$. We never need $\Theta$ to be an honest homotopy between honest maps of simplicial sets, only this equality of homotopy classes, which is exactly what additivity asserts.

\begin{remark}\label{rk:strongproblem}
In the construction of $\Theta$, it will be convenient to introduce the following terminology.

When $\D$ is a semiderivator, we should think of an object $f\in \D([1])$ as a coherent arrow-shaped diagram (\ie a coherent map) in $\D(e)$. We could then say that $f$ is a \emph{coherent isomorphism} if its underlying diagram is an isomorphism. Similar to the condition of a square being cocartesian, this can be checked by showing some canonical map is an isomorphism: $f\in \D([1])$ is a coherent isomorphism if and only if the counit of the $(0_!,0^\ast)$ adjunction is an isomorphism by \cite[Proposition~3.12]{Gro13}. The left Kan extension $0_!$ actually exists in any prederivator, as the functor $0_!\colon \D(e)\to \D([1])$ is canonically isomorphic to $\pi_{[1]}^\ast$, where $\pi_{[1]}\colon [1]\to e$ is the projection. 

If $\D$ is a pointed semiderivator, then in each $\Hom_{\D(e)}(a,b)$ there is a unique map that factors as $a\to 0\to b$ (for a given choice of $0\in\D(e)$). For a pointed semiderivator $\D$, we can look at the full subcategory $\mathcal Z\subset \D([2])$, thinking of $[2]$ as coherent compositions of two maps, on the objects $F$ whose middle term $1^\ast F$ is zero.  There is a composition morphism $c^\ast\colon \D([2])\to \D([1])$ induced by $c\colon [1]\to[2]$ with $c(0)=0$ and $c(1)=2$, so we define the \emph{coherent zero maps} in $\D([1])$ as the essential image of~$c^\ast(\mathcal Z)$.

Alternatively, we could say that $f\in \D([1])$ is a coherent zero map if $\dia_{[1]}(f)$ factors in $\D(e)$ as $a\to 0 \to b$. If $\D$ is strong, then this definition is equivalent to the above. We may lift the composition $a\to 0\to b$ in $\Fun([2],\D(e))$ to a coherent map $F\in \D([2])$ such that $\dia_{[2]} F\cong (a\to 0\to b)$. Again using the functor $c\colon[1]\to[2]$ from above, we have a commutative diagram
\begin{equation*}
\vcenter{\xymatrix{
\D([2])\ar[r]^-{\dia_{[2]}}\ar[d]_-{c^\ast}&\Fun([2],\D(e))\ar[d]^-{c^\ast}\\
\D([1])\ar[r]_-{\dia_{[1]}}&\Fun([1],\D(e))
}}
\end{equation*}
so that the representative $F\in\mathcal Z$ satisfies $c^\ast F\cong f$.

If $\D$ is not strong, then we do not know that $\dia_{[2]}$ is essentially surjective, so there might not be $F\in\D([2])$ with underlying diagram $a\to 0 \to b$ so that we could identify~$c^\ast F\cong f$. Because we are working with derivators which may not be strong (although all our examples are strong), it is prudent to take `essential image of $\mathcal Z$' as the definition.
\end{remark}

\begin{proof}[Proof of Proposition~\ref{prop:homotopy1}]
We will construct $\Theta$ in two steps.

First, we define a morphism $P\colon \D^{[1]}\to\E$. For a coherent morphism\linebreak $(f\colon a\to b)\in\D^{[1]}$, we may apply $\Xi^{[1]}$ to obtain an object in $\Ecof^{[1]}\subset \E^{[1]\times\square}$. Specifically, its underlying diagram takes the form (where we do not label every arrow)
\begin{equation*}
\vcenter{\xymatrix@R=1em@C=1em{
S(a)\ar[dd]_{S(f)}\ar[dr]\ar[rr]^{\alpha_{a}}&&T(a)\ar[dr]^{\beta_{a}}\ar'[d][dd]& \\
&0\ar[rr]\ar[dd]&{}&Q(a)\ar[dd]^{Q(f)} \\
S(b)\ar[dr]\ar'[r][rr]&&T(b)\ar[dr]& \\
&0\ar[rr]&&Q(b)
}}
\end{equation*}

We may consider the functor $\ul\to [1]\times\square$ by the inclusion into the upper-left corner of the back face of the cube given above. Restriction along this functor gives the coherent diagram in $\E^\ul$
\begin{equation}\label{dia:Ppushout}
\vcenter{\xymatrix{
S(a)\ar[r]^-{\alpha_a}\ar[d]_-{S(f)}&T(a)\\
S(b)
}}
\end{equation}
We may then apply $(1,1)^\ast i_{\ul,!}$ to first take the pushout of Diagram~\ref{dia:Ppushout} and then restrict to the new object. The composition is a cocontinuous morphism which we denote $P\colon \D^{[1]}\to\E$.

We now record two special cases. Each is naturally stated as an isomorphism between values of $P$ at particular objects, but we state and prove them as isomodifications between morphisms of derivators, since that is the only form in which they can be fed into Lemma~\ref{lemma:functoriality}.

In the first case: let $\pi_{[1]}\colon[1]\to e$ be the projection, so that the strict morphism $\pi_{[1]}^\ast\colon\D\to\D^{[1]}$ sends an object to the coherent identity map on it; recall from Remark~\ref{rk:strongproblem} that $\pi_{[1]}^\ast\cong 0_!$. We claim there is an isomodification
\begin{equation}\label{eq:Pidentity}
P\,\pi_{[1]}^\ast\;\cong\;T\colon\ \D\longrightarrow\E .
\end{equation}
Informally this is the statement that $P(f)\cong T(a)$ when $f$ is a coherent isomorphism, since pushing out along an isomorphism yields an isomorphism; here is the argument in the form we need. Let $\tau\colon[1]\to\square$ be the inclusion of the top row, $\tau(0)=(0,0)$ and $\tau(1)=(1,0)$, and let $t\colon\ul\to[1]$ be the functor with $t(0,0)=t(0,1)=0$ and $t(1,0)=1$. Writing $k\colon\ul\to[1]\times\square$ for the inclusion used above and $p\colon[1]\times\square\to\square$ for the projection, strictness of $\Xi$ gives $\Xi^{[1]}\pi_{[1]}^\ast=p^\ast\Xi$, and one checks directly that $p\,k=\tau\,t$. Hence
\begin{equation*}
k^\ast\Xi^{[1]}\pi_{[1]}^\ast=(p\,k)^\ast\Xi=t^\ast\tau^\ast\Xi,
\end{equation*}
which is to say that Diagram~\ref{dia:Ppushout} evaluated at $f=\pi_{[1]}^\ast a$ is $t^\ast$ of the coherent arrow $\alpha_a\colon S(a)\to T(a)$, as one expects: its two legs are $\alpha_a$ and an identity.

Now $t$ is a right adjoint: its left adjoint is $\ell\colon[1]\to\ul$ with $\ell(0)=(0,0)$ and $\ell(1)=(1,0)$, as one checks directly, these being posets. Likewise $1\colon e\to[1]$ and $(1,1)\colon e\to\square$, being inclusions of terminal objects, are right adjoint to the projections $\pi_{[1]}$ and $\pi_\square$. By Example~\ref{ex:adjoints}(1), pulling back reverses these adjunctions, so $t^\ast$, $1^\ast$ and $(1,1)^\ast$ are left adjoint to $\ell^\ast$, $\pi_{[1]}^\ast$ and $\pi_\square^\ast$ respectively. Since $i_{\ul,!}$ is left adjoint to $i_\ul^\ast$, the composite $(1,1)^\ast i_{\ul,!}\,t^\ast$ is left adjoint to
\begin{equation*}
\ell^\ast i_\ul^\ast\pi_\square^\ast=(\pi_\square i_\ul\ell)^\ast=\pi_{[1]}^\ast,
\end{equation*}
and so is $1^\ast$. By uniqueness of adjoints there is an isomodification $(1,1)^\ast i_{\ul,!}\,t^\ast\cong 1^\ast$, and therefore
\begin{equation*}
P\,\pi_{[1]}^\ast=(1,1)^\ast i_{\ul,!}\,t^\ast\tau^\ast\Xi\;\cong\;1^\ast\tau^\ast\Xi=(1,0)^\ast\Xi=T.
\end{equation*}
This proves Equation~\ref{eq:Pidentity}.

In the second case: suppose $f$ is a coherent zero map; we claim that $P(f)\cong Q(a)\vee S(b)$, naturally. Here `naturally' needs a precise meaning, which we now supply, because the coherent zero maps do not themselves form a subprederivator of $\D^{[1]}$ in any evident way. Let $\mathcal Z\subset \D^{[2]}$ be the full subprederivator of Remark~\ref{rk:strongproblem}, so that $\mathcal Z(K)\subset\D^{[2]}(K)$ consists of those $F$ with $1^\ast F=0\in\D(K)$; this is a subprederivator because every $u^\ast$ is a pointed functor. The coherent zero maps are by definition the essential image of $c^\ast\colon\mathcal Z\to\D^{[1]}$. What we will produce is an isomodification
\begin{equation}\label{eq:Pnaturality}
P\,c^\ast\big|_{\mathcal Z}\;\cong\;\big(Q\,0^\ast\big)\vee\big(S\,2^\ast\big)
\end{equation}
of cocontinuous morphisms $\mathcal Z\to\E$, where $0^\ast,2^\ast\colon\mathcal Z\to\D$ pick out the two outer terms; naturality in $f$ is then naturality in a chosen $F\in\mathcal Z$ with $c^\ast F\cong f$. This is exactly the form of naturality needed at the end of the proof, where the relevant $F$ is produced coherently and not chosen by hand.

Before proceeding we check that $\mathcal Z$ is itself an object of $\Der_\text{K}$, so that Lemma~\ref{lemma:functoriality} may be applied to morphisms out of it. The argument of Remark~\ref{rk:subderivator} applies verbatim. Axioms Der1 and Der2 hold because $\mathcal Z(K)$ is a full subcategory of $\D^{[2]}(K)$ cut out by a levelwise condition, and $\mathcal Z(e)$ is pointed because the zero object of $\D^{[2]}$ has zero middle term. For the Kan extensions, let $u\colon J\to K$ and $F\in\mathcal Z(J)$. The morphism $1^\ast\colon\D^{[2]}\to\D$ preserves all left and right Kan extensions that $\D$ possesses, by~\cite[Proposition~2.5]{Gro13}, since it is a pullback in a diagram direction unrelated to $u$; hence $1^\ast u_!F\cong u_!\,1^\ast F=u_!0=0$ as $u_!$ is pointed, and likewise $1^\ast u_\ast F\cong u_\ast 1^\ast F=0$ for a sieve $u$, as $u_\ast$ is a right adjoint and so preserves the zero object. Therefore $u_!$ and $u_\ast$ restrict to $\mathcal Z$, and $\mathcal Z$ is a left pointed derivator on $\Dirf$. The morphisms $c^\ast,0^\ast,2^\ast$ out of $\mathcal Z$ are pullbacks, hence strict and cocontinuous.

By identifying $f\cong c^\ast F$ (where $c\colon [1]\to [2]$ is the composition map), we also identify $\Xi^{[1]}(f)\cong c^\ast \Xi^{[2]}(F)$ and by restricting to the same upper corner of the cube we conclude
\begin{equation}\label{dia:pushouts}
\vcenter{\xymatrix@C=1em@R=1em{
S(a)\ar[r]\ar[d]&T(a)\\
0\ar[d]&\\
S(b)&
}}
\overset{c^\ast}\mapsto
\vcenter{\xymatrix@C=1em@R=1em{
S(a)\ar[r]\ar[dd]_-{S(f)}&T(a)\\
\phantom{0}\\
S(b)&
}}
\end{equation}
where the leftmost composition is the underlying diagram of $S(F)$. Denote by $I$ the diagram shape of the lefthand side.

Because the morphism $c^\ast$ is cocontinuous, the pushout $P(f)$ of the righthand side of Diagram~\ref{dia:pushouts}, \ie the colimit of the diagram $\ul$, is isomorphic to the colimit of shape $I$ on the lefthand side. Rather than compute the lefthand colimit directly, we can take the left Kan extension of $I\to [1]\times[2]$ and restrict to the object at $(1,2)$ to obtain an isomorphic result.

To illustrate this completely, we take the left Kan extension in two steps: first the left Kan extension along $I\to J$, where $J\subset [1]\times[2]$ is the full subcategory lacking $(1,2)$, and then extend along $J\to [1]\times[2]$. We first obtain an object isomorphic to~$Q(a)$ and second compute the pushout of $Q(a)$ with $S(b)$:
\begin{equation*}
\vcenter{\xymatrix@C=1em@R=1em{
S(a)\ar[r]\ar[d]&T(a)\\
0\ar[d]&\\
S(b)&
}}
\quad\mapsto\quad
\vcenter{\xymatrix@C=1em@R=1em{
S(a)\ar[r]\ar[d]&T(a)\ar[d]\\
0\ar[d]\ar[r]&Q(a)\\
S(b)&
}}
\quad\mapsto\quad
\vcenter{\xymatrix@C=1em@R=1em{
S(a)\ar[r]\ar[d]&T(a)\ar[d]\\
0\ar[d]\ar[r]&Q(a)\ar[d]\\
S(b)\ar[r]&S(b)\vee Q(a)
}}
\end{equation*}
Thus we conclude $P(f)\cong S(b)\vee Q(a)$.

Two remarks complete the proof of Equation~\ref{eq:Pnaturality}. First, every step of the computation -- the restriction along $c$, the two left Kan extensions, and the final pullback -- is a morphism of derivators applied to $F$, and the isomorphisms between the steps are mates or counits of adjunctions; hence the composite isomorphism is an isomodification of morphisms $\mathcal Z\to\E$, and not merely a family of isomorphisms chosen objectwise. Second, the coproduct is symmetric up to a canonical isomodification: interchanging the two copies of $e$ in $e\sqcup e$ and composing with the reflection of $\square$ in its diagonal, which fixes $(0,0)$ and $(1,1)$ and exchanges $(1,0)$ and $(0,1)$, gives $\vee\cong\vee\circ\text{swap}$. Hence the computed $S(b)\vee Q(a)$ agrees canonically with $Q(a)\vee S(b)$, and Equation~\ref{eq:Pnaturality} follows in the stated form.

Now, we construct the second step of $\Theta$. Recall the definition of $s\colon [2]\to [1]$. We define two sections of $s$: first,  $d\colon[1]\to[2]$ with $d(0)=0$, $d(1)=1$; second, $e\colon [1]\to[2]$ with $e(0)=0$, $e(1)=2$. For any $a\in A$, $(d\phi(a),a)$ and $(e\phi(a),a)$ are both elements of $\phi^{-1}(s)$ as $se=sd=\id_{[1]}$. This gives two inclusions from $A$ into $\phi^{-1}(s)$, to which we also name $d$ and $e$ (as the other functors will not be returning). There is also a unique natural transformation $\zeta\colon d\Rightarrow e$ as $d\phi(a) \leq e\phi(a)$ for any~$a\in A$.

We may first extend $d,e:\Ar(A)\to\Ar(\phi^{-1}(s))$ by naturality, and we still have the natural transformation $\zeta\colon d\Rightarrow e$. We can better view $\zeta$ as a\linebreak functor $\zeta\colon \Ar(A)\times[1]\to\Ar(\phi^{-1}(s))$, where (in particular) the original data of the natural transformation is retained by the maps in the $[1]$-dimension
\begin{equation*}
\zeta\Big(((a\to b),0)\to((a\to b),1)\Big)=\,d(a\to b)\xrightarrow{\zeta_{(a\to b)}} e(a\to b).
\end{equation*}

Therefore we obtain a strict cocontinuous morphism
\begin{equation*}
\zeta^\ast\colon \D(\Ar(\phi^{-1}(s)))\to\D(\Ar(A)\times[1])
\end{equation*}
As we vary $\phi\colon A\to[1]$, we obtain a map with domain $I\SD A \D$, and because each $\zeta^\ast$ is cocontinuous, we may restrict our codomain to $\SD A \D^{[1]}\subset \D(\Ar(A)\times[1])=~\D^{[1]}(\Ar(A))$. This process tells us how to extend the various $\zeta^\ast$ to a map\linebreak $Z\colon I\Sdot\D\to\Sdot\D^{[1]}$ (read: capital $\zeta$).

We now need to check that $Z$ is a map of simplicial categories. Suppose we have a map $\psi\colon B\to A$ in $\Ord$. Then we need to check that the following square commutes:
\begin{equation}\label{dia:simplicialsquare}
\vcenter{\xymatrix@R=2em@C=2em{
I\SD A \D\ar[r]^-{\psi_1^\ast}\ar[d]_-{Z_A}&I\SD B\D\ar[d]^-{Z_B}\\
\SD A \D^{[1]}\ar[r]_-{\psi_2^\ast}&\SD B\D^{[1]}
}}
\end{equation}
where we use $\psi_1^\ast,\psi_2^\ast$ to denote the maps induced by $\psi$ on the simplicial categories $I\Sdot\D$ and $\Sdot\D^{[1]}$.

Let $(\phi\colon A\to [1],X\in\SD {\phi^{-1}(s)}\D)\in I\SD A\D$. Then taking the upper composition of Diagram~\ref{dia:simplicialsquare} we first get the object
\begin{equation*}
\psi_1^\ast(\phi,X)=\big(\phi\psi, (\psi')^\ast X\big)\in I\SD B\D,\quad (\psi')^\ast X\in\SD {(\phi\psi)^{-1}(s)}\D
\end{equation*}
where $\psi'$ is the natural map $(\phi\psi)^{-1}(s)\to\phi^{-1}(s)$ induced by $\psi$, $(n,b)\mapsto (n,\psi(b))$. If we apply $Z_B$ to $(\phi\psi,(\psi')^\ast X)$, we obtain the coherent object in $\Sdot\D^{[1]}(B)$
\begin{equation*}
\xymatrix@C=2em{
d_B^\ast(\psi')^\ast X\ar[r]^{\zeta_B^\ast}&e_B^\ast(\psi')^\ast X
}
\end{equation*}
where $d_B$ and $e_B$ are the specific instances of $d,e$ in this case.

If we traverse Diagram~\ref{dia:simplicialsquare} using the lower composition, we first take $Z_A(\phi,X)$ to obtain $\zeta_A^\ast\colon d_A^\ast X\to e_A^\ast X$.  Then applying $\psi_2^\ast$ we obtain
\begin{equation*}
\xymatrix@C=3em{
\psi^\ast d_A^\ast X\ar[r]^{\psi^\ast\zeta_A^\ast}& \psi^\ast e_A^\ast X.
}
\end{equation*}
But by strict 2-functoriality, these two coherent maps are equal on the nose, as they are just pullbacks induced by maps in $\Dirf$. Therefore $Z$ is indeed a morphism of simplicial categories.

The morphism $P$ is a composite of a pullback, a left Kan extension and a further pullback, so it is cocontinuous but \emph{not} strict. Consequently $\Sdot P$ is \emph{not defined}, and we cannot postcompose $Z$ with it; this is precisely the difficulty flagged for the coproduct at Equation~\ref{eq:strictifiedcoproduct}. Lemma~\ref{lemma:functoriality} resolves it: $P$ induces a morphism
\begin{equation*}
\|P\|\colon\big\|\D^{[1]}\big\|\longrightarrow\|\E\|
\end{equation*}
in $\Ho\cS$, namely the zig-zag
\begin{equation}\label{eq:thetazigzag}
\vcenter{\xymatrix@C=3em{
I\Sdot\D\ar[r]^-{Z}&\Sdot\D^{[1]}&\Sdot\widetilde{\D^{[1]}}\ar[l]_-{\Sdot\Pi_P}\ar[r]^-{\Sdot\widetilde P}&\Sdot\E
}}
\end{equation}
with $\Sdot\Pi_P$ inverted after realisation, where $\Pi_P\colon\widetilde{\D^{[1]}}\to\D^{[1]}$ and $\widetilde P\colon\widetilde{\D^{[1]}}\to\E$ are strictification data for $P$ as in Proposition~\ref{prop:strictification}. We define
\begin{equation*}
\Theta:=\|P\|\circ\big|N_\bullet i Z\big|\colon\;\big|IN_\bullet i\Sdot\D\big|\longrightarrow\|\E\|,
\end{equation*}
a morphism of $\Ho\cS$. One and the same $\|P\|$ is used along the whole cylinder, which is what makes the two ends comparable.

It remains to identify the two ends, and this is where the discipline announced after Lemma~\ref{lemma:functoriality} is needed. We cannot evaluate $\Theta$ on objects: $\|P\|$ is not induced by any map of simplicial sets, so an object-by-object computation would be meaningless. What we do instead is observe that the restriction of $Z$ to each end of the cylinder is itself induced by a \emph{strict} morphism of derivators; we may then form the composite with $P$ inside $\Der_\text{K}$, replace it there by an isomodification-equivalent morphism, and only afterwards apply $\|-\|$.

If $\phi=0$ is the zero map, then $\phi^{-1}(s)=A$ by Remark~\ref{rk:usefulcylinder}, and $\zeta\colon d\Rightarrow e$ is the identity natural transformation since $d\phi(a)=e\phi(a)$ always. Thus $Z_A(0,X)$ is the constant diagram in the $[1]$-direction on $X$; that is, the restriction of $Z$ to this end is the map of simplicial categories
\begin{equation*}
\Sdot\pi_{[1]}^\ast\colon\Sdot\D\longrightarrow\Sdot\D^{[1]}
\end{equation*}
induced by the strict morphism $\pi_{[1]}^\ast$. Applying Lemma~\ref{lemma:functoriality} and then Equation~\ref{eq:Pidentity},
\begin{equation*}
\Theta\iota_0=\|P\|\circ\big\|\pi_{[1]}^\ast\big\|=\big\|P\,\pi_{[1]}^\ast\big\|=\|T\|,
\end{equation*}
which is the first end. Note that at no point did we evaluate $\|P\|$ anywhere: the composition $P\,\pi_{[1]}^\ast$ and the isomodification to $T$ both take place in $\Der_\text{K}$.

If $\phi=1$, then $\phi^{-1}(s)=A\ast A$. In this case, the functors $d,e\colon A\to A\ast A$ are the inclusion of the first and second component, respectively, and these extend to $d,e\colon \Ar(A)\to \Ar(A\ast A)$. Suppose that $A\cong [n]$ to simplify notation for the next observation. We denote the elements of $[n]\ast[n]$ by $i_0$ in the first copy and $i_1$ in the second copy. Then for any $(i\to j)\in\Ar([n])$, we have that
\begin{equation*}
\zeta_{(i\to j)}\colon d(i\to j)=(i_0\to j_0)\to (i_1\to j_1)=e(i\to j)
\end{equation*}
always factors (uniquely) through $(n_0\to n_0)$. Specializing further to the case $[n]=[2]$ and $(i\to j)=(1\to 2)$ for illustration,
\begin{equation*}
\vcenter{\xymatrix@R=1em@C=1em{
(0_0\to 0_0)\ar[r]&(0_0\to 1_0)\ar[r]\ar[d]&(0_0\to 2_0)\ar[r]\ar[d]&(0_0\to 0_1)\ar[r]\ar[d]&(0_0\to 1_1)\ar[r]\ar[d]&(0_0\to 2_1)\ar[d]\\
&(1_0\to 1_0)\ar[r]&*+[F=]{(1_0\to 2_0)}\ar[r]\ar[d]|\bullet&(1_0\to 0_1)\ar[r]\ar[d]&(1_0\to 1_1)\ar[r]\ar[d]&(1_0\to 2_1)\ar[d]\\
&&*+[F--]{(2_0\to 2_0)}\ar[r]|\bullet&(2_0\to 0_1)\ar[r]|\bullet\ar[d]&(2_0\to 1_1)\ar[r]|\bullet\ar[d]&(2_0\to 2_1)\ar[d]|\bullet\\
&&&(0_1\to 0_1)\ar[r]&(0_1\to 1_1)\ar[r]\ar[d]&(0_1\to 2_1)\ar[d]|\bullet\\
&&&&(1_1\to 1_1)\ar[r]&*+[F=]{(1_1\to 2_1)}\ar[d]\\
&&&&&(2_1\to 2_1)
}}
\end{equation*}
The transformation $\zeta_{i\to j}$ has domain and codomain in boxes above. Because all maps from the domain to the codomain are equal, we see that the marked sequence of arrows gives the appropriate factorization via $(2_0\to 2_0)$ in the dashed box.

Translating this fact to diagram notation, let $\delta\colon \Ar(A)\to \Ar(A\ast A)$ be the constant functor with value $(n_0\to n_0)$, \ie the identity map of the terminal object in the first copy of $A$. Then the natural transformation $\zeta\colon d\Rightarrow e$ factors as $d\Rightarrow \delta\Rightarrow e$ by the observation above. We can rewrite this composite of natural transformations as a functor $\upsilon\colon\Ar(A)\times[2]\to \Ar(A\ast A)$. If we again let $c\colon [1]\to [2]$ be the functor inducing coherent composition, we see that $\upsilon\circ (\id_{\Ar(A)}\times c)=\zeta$.

The preceding argument proves the following. As $A$ varies the functors $\upsilon$ are natural in $A$, and each is a functor in $\Dirf$, so the pullbacks $\upsilon^\ast$ assemble into a map of simplicial categories exactly as the $\zeta^\ast$ assembled into $Z$; the verification is the one carried out for Diagram~\ref{dia:simplicialsquare}, with $\zeta$ replaced by $\upsilon$, and again the two composites agree on the nose by strict $2$-functoriality. Furthermore this map lands in $\Sdot\mathcal Z$. Indeed, for $X\in\SD{A\ast A}\D$ we have
\begin{equation*}
0^\ast\upsilon^\ast X=d^\ast X=i_0^\ast X,\qquad 1^\ast\upsilon^\ast X=\delta^\ast X=(n_0\to n_0)^\ast X=0,\qquad 2^\ast\upsilon^\ast X=e^\ast X=i_1^\ast X,
\end{equation*}
the middle vanishing because $(n_0\to n_0)$ is a diagonal object of $\Ar(A\ast A)$ and vanishing on the diagonal is one of the conditions for $X$ to lie in $\SD{A\ast A}\D$. The outer two lie in $\SD A\D$ because $i_0^\ast$ and $i_1^\ast$ define maps of simplicial categories, as observed when $\nabla$ was constructed, and the middle one lies there trivially; since the conditions defining $\Sdot\mathcal Z$ inside $\mathcal Z(\Ar(A))=\D^{[2]}(\Ar(A))$ may be checked after restricting along $0,1,2\in[2]$, we conclude $\upsilon^\ast X\in\SD A\mathcal Z$. We therefore have a map of simplicial categories
\begin{equation*}
\upsilon^\ast\colon\operatorname{sub}_2\!\Sdot\D\longrightarrow\Sdot\mathcal Z .
\end{equation*}

Because $\upsilon\circ(\id_{\Ar(A)}\times c)=\zeta$, the restriction of $Z$ to the end $\phi=1$ factors as
\begin{equation*}
\vcenter{\xymatrix@C=3em{
\operatorname{sub}_2\!\Sdot\D\ar[r]^-{\upsilon^\ast}&\Sdot\mathcal Z\ar[r]^-{\Sdot c^\ast}&\Sdot\D^{[1]}
}}
\end{equation*}
and the second map is induced by the strict morphism $c^\ast$. Applying Lemma~\ref{lemma:functoriality} and then the isomodification of Equation~\ref{eq:Pnaturality} -- which is available precisely because $\upsilon^\ast X$ is a canonical lift of $Z_A(1,X)$ to $\mathcal Z$, rather than a lift chosen object by object -- we obtain
\begin{equation*}
\Theta\iota_1=\|P\|\circ\|c^\ast\|\circ\big|N_\bullet i\,\upsilon^\ast\big|=\big\|P\,c^\ast\big\|\circ\big|N_\bullet i\,\upsilon^\ast\big|=\big\|(Q\,0^\ast)\vee(S\,2^\ast)\big\|\circ\big|N_\bullet i\,\upsilon^\ast\big|.
\end{equation*}
Finally $(Q\,0^\ast)\vee(S\,2^\ast)=\vee\circ(Q\,0^\ast,S\,2^\ast)$, and $(Q\,0^\ast,S\,2^\ast)\colon\mathcal Z\to\E^{e\sqcup e}$ is strict, being a pair of composites of strict morphisms. So Lemma~\ref{lemma:functoriality} gives
\begin{equation*}
\Theta\iota_1=\|\vee\|\circ\big|N_\bullet i\,\Sdot(Q\,0^\ast,S\,2^\ast)\big|\circ\big|N_\bullet i\,\upsilon^\ast\big|=\|\vee\|\circ\big|N_\bullet i\,(Q\,i_0^\ast,S\,i_1^\ast)\big|=\nabla_{Q,S},
\end{equation*}
the middle equality because $0^\ast\upsilon^\ast=i_0^\ast$ and $2^\ast\upsilon^\ast=i_1^\ast$ on the nose, so that $\Sdot(Q\,0^\ast,S\,2^\ast)\circ\upsilon^\ast=(Q\,i_0^\ast,S\,i_1^\ast)$ as maps of simplicial categories, and the last by the definition of $\nabla_{Q,S}$ in Equation~\ref{eq:strictifiedcoproduct}. This completes the proof.
\end{proof}

Proving the additivity of derivator K-theory is now immediate.

\begin{theorem}\label{thm:mainadditivity}
Let $\Xi\colon \D\to\Ecof$ be a cofibration morphism of derivators, and let $S,T,Q$ be the corresponding morphisms of derivators. Then $S\vee Q$ and $T$ induce homotopic maps $\K(\D)\to \K(\E)$.
\end{theorem}
\begin{proof}
Proposition~\ref{prop:homotopy1}, together with the cylinder argument recorded before it, gives an equality of homotopy classes
\begin{equation}\label{eq:firstend}
\big[\|T\|\big]=\big[\nabla_{Q,S}\big]
\end{equation}
in $\big[\|\D\|,\|\E\|\big]$, where Lemma~\ref{lemma:cylinder} has been used to identify $\big|\operatorname{sub}_2\!N_\bullet i\Sdot\D\big|$ with $\|\D\|$. The right-hand side depends only on the pair of morphisms $(Q,S)$, and not on the cofibration morphism $\Xi$ that produced them. We exploit this by exhibiting a second cofibration morphism with the same source and quotient, but with target $S\vee Q$.

From the proof of Theorem~\ref{thm:additivities}, recall the morphism ${\sigma\colon \E^{e\sqcup e}\to\Ecof}$, which on objects is $(a,c)\mapsto\big(a\to a\vee c,\,0\to c'\big)$ with $c\to c'$ a canonical isomorphism. Precompose $\sigma$ with the morphism $(S,Q)\colon\D\to\E\times\E\cong \E^{e\sqcup e}$ (note that this is the \emph{pair} $(S,Q)$, not their coproduct) and set $\Xi':=\sigma\,(S,Q)\colon\D\to\Ecof$. Reading off the three corners,
\begin{equation*}
(0,0)^\ast\Xi'\cong S,\qquad (1,0)^\ast\Xi'\cong S\vee Q,\qquad(1,1)^\ast\Xi'\cong Q .
\end{equation*}
The morphism $\Xi'$ is cocontinuous but not strict, so it is not literally a cofibration morphism of derivators and Proposition~\ref{prop:homotopy1} does not apply to it directly. By Proposition~\ref{prop:strictification} choose a strict equivalence $\Pi\colon\widetilde\D\to\D$ and a strict cocontinuous morphism $\widetilde{\Xi'}\colon\widetilde\D\to\Ecof$ with $\Xi'\Pi=\widetilde{\Xi'}$. Then $\widetilde{\Xi'}$ \emph{is} a cofibration morphism of derivators. Its source, target and quotient morphisms $S'$, $T'$ and $Q'$ are strict, but they are only isomorphic, not equal, to $S\Pi$, $(S\vee Q)\Pi$ and $Q\Pi$. This does not matter, because both sides of the equation $[\|T'\|]=[\nabla_{Q',S'}]$ supplied by Proposition~\ref{prop:homotopy1} are unchanged when a morphism is replaced by an isomorphic one: the left side by Lemma~\ref{lemma:functoriality}, and the right side by the property of $\nabla$ recorded after Equation~\ref{eq:strictifiedcoproduct}. We may therefore work with $S\Pi$, $(S\vee Q)\Pi$ and $Q\Pi$ directly.

Proposition~\ref{prop:homotopy1} applied to $\widetilde{\Xi'}$ therefore gives
\begin{equation*}
\big[\big\|(S\vee Q)\Pi\big\|\big]=\big[\nabla_{Q\Pi,S\Pi}\big].
\end{equation*}
Because $\Pi$ is strict we have $\Pi\,i_j^\ast=i_j^\ast\,\Pi$ for $j=0,1$, whence $(Q\Pi\, i_0^\ast,S\Pi\, i_1^\ast)=(Q\,i_0^\ast,S\,i_1^\ast)\circ\operatorname{sub}_2\!\Sdot\Pi$ as maps of simplicial categories; postcomposing with $\|\vee\|$ gives
\begin{equation*}
\nabla_{Q\Pi,S\Pi}=\nabla_{Q,S}\circ\big|\operatorname{sub}_2\!N_\bullet i\Sdot\Pi\big| .
\end{equation*}
On the other side, $\big\|(S\vee Q)\Pi\big\|=\|S\vee Q\|\circ\|\Pi\|$ by Lemma~\ref{lemma:functoriality}. The subdivision homeomorphism of~\cite[\S4]{Gra89} is natural, so $\big|\operatorname{sub}_2\!N_\bullet i\Sdot\Pi\big|$ is identified with $\|\Pi\|$, which is a homotopy equivalence by Lemma~\ref{lemma:functoriality}. Cancelling it leaves
\begin{equation*}
\big[\|S\vee Q\|\big]=\big[\nabla_{Q,S}\big],
\end{equation*}
and comparing with Equation~\ref{eq:firstend} gives $\big[\|T\|\big]=\big[\|S\vee Q\|\big]$. Applying $\Omega$ yields $\K(T)\simeq \K(S\vee Q)$, as required.
\end{proof}
In particular, by Theorem~\ref{thm:additivities}, this proves additivity for non-stable derivators as it was conjectured by Muro-Raptis. By Remark~\ref{rk:hspace}, the conclusion may equivalently be read as $[\K(T)]=[\K(S)]+[\K(Q)]$ in $[\K(\D),\K(\E)]$.

\section{Further properties}

As additivity sits as the central property of any flavor of algebraic K-theory, we can now see what else we have obtained. First, as promised, we can perform the Eilenberg swindle for left pointed derivators with `large' domains following \cite[V.1.9.1]{Wei13}.

\begin{prop}\label{prop:derivatorswindle}
Suppose that $\D$ is a left pointed derivator defined on $\Dia\subset\Cat$ with $\Dirf\subset\Dia$, and such that the countable discrete set $\omega$ is in $\Dia$. Then $\K(\D)$ is contractible.
\end{prop}
Theorem~\ref{thm:mainadditivity} was stated for left pointed derivators on $\Dirf$, but its proof uses only diagram shapes lying in $\Dirf$, so it applies verbatim to a left pointed derivator defined on any $\Dia$ containing $\Dirf$; this is why we ask for that containment.
\begin{proof}
If we look at the functor $\pi_\omega\colon\omega\to e$, we see that the left Kan extension $\pi_{\omega,!}$ is the colimit of this discrete set, \ie the coproduct. From an object $X\in\D(e)$, we can obtain the countable coproduct of copies of $X$ via the functor $\pi_{\omega,!}\pi_\omega^\ast\colon\D(e)\to\D(e)$, and we can promote this to a cocontinuous morphism of derivators $\bigvee\colon\D\to\D$.

We now want to construct a cofibration morphism of derivators with this construction in mind. Let $s\colon \omega\to\omega$ denote the successor function, \ie $s(n)=n+1$ We build another category $\Gamma_s$ as the diagram underlying the function $s$. That is, the objects of $\Gamma_s$ are two disjoint copies of $\omega$, which we label 0 and 1, with an arrow~$n_0\to s(n)_1=(n+1)_1$. To sketch the initial segment of $\Gamma_s$:
\begin{equation*}
\vcenter{\xymatrix{
0_0\ar[dr]&1_0\ar[dr]&2_0\ar[dr]&3_0\ar[dr]&\cdots\\
0_1&1_1&2_1&3_1&4_1&\cdots
}}
\end{equation*}
This category admits a functor $p\colon\Gamma_s\to [1]$ which sends $n_i$ to $i\in[1]$. We claim that the cofibration morphism of derivators $\Xi\colon \D\to\Dcof$ induced by
\begin{equation*}
\vcenter{\xymatrix{
\D\ar[r]^-{\pi_{\Gamma_s}^\ast}&\D^{\Gamma_s}\ar[r]^-{p_!}&\D^{[1]}\ar[r]^-{i_{[1],\ast}}&\D^\ul\ar[r]^-{i_{\ul,!}}&\Dcof
}}
\end{equation*}
provides us with a null homotopy of the identity on $\D$, giving us our swindle.

One preliminary remark. The composite above is cocontinuous but not strict, since $p_!$, $i_{[1],\ast}$ and $i_{\ul,!}$ are Kan extensions; so $\Xi$ is not literally a cofibration morphism of derivators, and Theorem~\ref{thm:mainadditivity} does not apply to it as it stands. We proceed exactly as in the proof of that theorem: by Proposition~\ref{prop:strictification} choose a strict equivalence $\Pi\colon\widetilde\D\to\D$ and a strict cocontinuous $\widetilde\Xi\colon\widetilde\D\to\Dcof$ with $\Xi\Pi=\widetilde\Xi$, apply additivity to $\widetilde\Xi$, and use that $\|\Pi\|$ is a homotopy equivalence to transport the resulting identity back to $[\K(\D),\K(\D)]$. Since this replacement changes each of $S$, $T$, $Q$ only by precomposition with the equivalence $\Pi$, we suppress it below and compute with $\Xi$ itself.

We need to identify the source, target, and quotient morphisms associated to this cofibration morphism. Let $X\in\D$. To begin, $(0,0)^\ast\Xi$ is isomorphic to $0^\ast p_!\pi_{\Gamma_s}^\ast$, as the last two functors $i_{[1],\ast}$ and $i_{\ul,!}$ are fully faithful and do not change the underlying diagram of $[1]\subset\square$. Using Der4L, we can compute $0^\ast p_!\pi_{\Gamma_s}^\ast X$ via the comma category~$(p/0)$. Specifically,
\begin{equation*}
\vcenter{\xymatrix{
0^\ast p_!\pi_{\Gamma_s}^\ast X\cong \pi_{(p/0),!}\pr^\ast \pi_{\Gamma_s}^\ast X
}}
\end{equation*}
where $\pr\colon (p/0)\to \Gamma_s$ is the usual projection to the comma category as in Definition~\ref{defn:leftderivator}. The composition $\pr^\ast \pi_{\Gamma_s}^\ast$ is the pullback along $(p/0)\to \Gamma_s\to e$, which is the same thing as the pullback $\pi_{(p/0)}^\ast$ directly. To complete this computation, we need to identify the comma category.

The objects of the comma category are $n_i\in\Gamma_s$ along with a map $p(n_i)=i\to 0$. This forces $p(n_i)=0$ as there are no other maps in $[1]$. Therefore $(p/0)$ consists solely of the subcategory $\omega_0\subset\Gamma_s$. Therefore we conclude
\begin{equation*}
\vcenter{\xymatrix{
(0,0)^\ast \Xi X\cong 0^\ast p_!\pi_{\Gamma_s}^\ast X\cong \pi_{(p/0),!}\pr^\ast \pi_{\Gamma_s}^\ast X\cong \pi_{\omega_0,!}\pi_{\omega_0}^\ast X=\bigvee X
}}
\end{equation*}

The computation for $(1,0)^\ast \Xi X$ is similar: it is isomorphic to $1^\ast p_!\pi_{\Gamma_s}^\ast X$. We have an isomorphism by Der4L
\begin{equation*}
\vcenter{\xymatrix{
1^\ast p_!\pi_{\Gamma_s}^\ast X\cong \pi_{(p/1),!}\pr^\ast\pi_{\Gamma_s}^\ast X\cong \pi_{(p/1),!}\pi_{(p/1)}^\ast X
}}
\end{equation*}
We now need to identify $(p/1)$ as a category. As $1\in[1]$ is the terminal object, this comma category is equal to $\Gamma_s$ itself. In order to compute the colimit of shape $\Gamma_s$, we note that it admits a reflective subcategory. Consider $\omega_1\subset\Gamma_s$ as a subcategory. Then this inclusion admits a left adjoint $\ell\colon\Gamma_s\to \omega_1$ such that $\ell(n_1)=n_1$ and $\ell(m_0)=S(m_0)=(m+1)_1$. To double check that this is actually an adjunction, we check for $m_0,n_1\in\Gamma_s$
\begin{equation*}
\vcenter{\xymatrix{
\Hom_{\Gamma_s}(m_0,n_1)=\Hom_{\omega_1}((m+1)_1,n_1)
}}
\end{equation*}
The lefthand side is nonempty if and only if $m+1=n$, which is the same for the righthand side. In the case that we are looking at the maps between $m_0$ and $n_0$ or $m_1$ and $n_1$, the hom-set equality is straightforward as both sides are empty.

By \cite[Proposition~1.18]{Gro13}, right adjoint functors preserve colimits; that is, if~$r\colon \omega_1\to \Gamma_s$ is the inclusion, we have $\pi_{\omega_1,!}r^\ast Y\cong\pi_{\Gamma_s,!} Y$ for any $Y\in\D(\Gamma_s)$. Noting finally that the composition $r^\ast\pi_{\omega_1}^\ast=\pi_{(p/1)}^\ast$, we can now complete the chain of isomorphisms to finish the computation:
\begin{equation*}
\vcenter{\xymatrix{
(1,0)^\ast \Xi X\cong 1^\ast p_!\pi_{\Gamma_s}^\ast X\cong \pi_{(p/1),!}\pr^\ast\pi_{\Gamma_s}^\ast X\cong \pi_{(p/1),!}\pi_{(p/1)}^\ast X\cong \pi_{\omega_1,!}\pi_{\omega_1}^\ast X\cong \bigvee X
}}
\end{equation*}

The last step is to compute $(1,1)^\ast\Xi X$, which takes a little more work. In order to identify it, we modify the construction of the cofibration morphism $\Xi$ up to isomorphism. Instead of immediately applying $p_!\colon \D^{\Gamma_s}\to\D^{[1]}$, we note that $\Gamma_s$ is nearly~$\omega\times[1]$, if we think of this latter category as
\begin{equation*}
\vcenter{\xymatrix{
-1_0\ar[d]&0_0\ar[d]&1_0\ar[d]&2_0\ar[d]&3_0\ar[d]&\cdots\\
0_1&1_1&2_1&3_1&4_1&\cdots
}}
\end{equation*}
Let $i\colon \Gamma_s\to \omega\times[1]$ be the inclusion as indicated by the labelling in the diagram. The category $\omega\times[1]$ has the canonical projection to $[1]$ which we denote $p_{[1]}$. Then we have a factorisation $p=p_{[1]} i$, so we have an isomorphism $p_!\cong p_{[1],!}i_!$. The functor~$i$ is a cosieve, so the left Kan extension $i_!$ is extension by zero, and the morphism $p_{[1],!}$ is similar to $p_!$ in that it computes the infinite coproduct of the top row and the bottom row separately. We do not need to make this computation, but in essence we have
\begin{equation*}
\vcenter{\xymatrix{
0^\ast p_{[1],!}i_!\pi_{\Gamma_s}^\ast\cong 0\vee X\vee X\vee\cdots\cong\bigvee X,\quad 1^\ast p_{[1],!}i_!\pi_{\Gamma_s}^\ast\cong X\vee X\vee\cdots\cong \bigvee X
}}
\end{equation*}

This reorganising allows us to make explicit the $[1]$-dimension of $\Gamma_s$ and reverse the order of the infinite coproduct and the pushout. More specifically, the following diagram commutes up to isomorphism as all the morphisms involved are cocontinuous and involve distinct dimensions of the overall diagram:
\begin{equation*}
\vcenter{\xymatrix@C=5em{
\D^{\omega\times[1]}\ar[r]^-{\id_\omega^\ast\times i_{\ul,!}i_{[1],\ast}}\ar[d]_-{p_{[1],!}}&\Dcof^\omega\ar[d]^-{\coprod}\\
\D^{[1]}\ar[r]^-{i_{\ul,!}i_{[1],\ast}}&\Dcof\
}}
\end{equation*}
Therefore to compute $(1,1)^\ast \Xi X$, it suffices to compute the following:
\begin{equation*}
\vcenter{\xymatrix{
(1,1)^\ast i_{\ul,!}i_{[1],\ast}p_!\pi_{\Gamma_s}^\ast X\cong (1,1)^\ast i_{\ul,!}i_{[1],\ast}p_{[1],!}i_!\pi_{\Gamma_s}^\ast X\cong (1,1)^\ast\bigvee(\id_\omega^\ast\times i_{\ul,!}i_{[1],\ast})i_!\pi_{\Gamma_s}^\ast X
}}
\end{equation*}
We can commute the coproduct morphism with $(1,1)^\ast$, so the final challenge is to understand the $(1,1$) entry of each of the cocartesian squares in $\Dcof^\omega$ that we obtain before taking the infinite coproduct. We have already made the computation of~$i_!\pi_{\Gamma_s}^\ast X\in\D^{\omega\times[1]}$ and found that it is
\begin{equation*}
\vcenter{\xymatrix{
0\ar[d]&X\ar[d]&X\ar[d]&X\ar[d]&\cdots\\
X&X&X&X&\cdots
}}
\end{equation*}
where all but the first map are (coherent) identities. The pushouts are accomplished independently of each other, which means $(\id_\omega^\ast\times i_{\ul,!}i_{[1],\ast})i_!\pi_{\Gamma_s}^\ast X$ has the form (where we rotate all our arrows from vertical to horizontal before pushing out)
\begin{equation*}
\vcenter{\xymatrix@R=1em@C=1em{
0\ar[r]\ar[d]&X\ar[d]&X\ar[r]\ar[d]&X\ar[d]&X\ar[r]\ar[d]&X\ar[d]&\cdots\\
0\ar[r]&X&0\ar[r]&0&0\ar[r]&0&\cdots
}}
\end{equation*}
So we see that $(1,1)^\ast\bigvee$ applied to this diagram yields a coproduct of $X$ with countably many zero objects, so we conclude that $(1,1)^\ast\Xi X\cong X$.

By additivity we have $\K(T)\simeq \K(S\vee Q)$, which by Remark~\ref{rk:hspace} we may read as an identity in the group $[\K(\D),\K(\D)]$:
\begin{equation*}
[\K(T)]=[\K(S)]+[\K(Q)].
\end{equation*}
Substituting the three computations above yields $[\K(\bigvee)]=[\K(\bigvee)]+[\K(\id_\D)]$, and cancelling in the group $[\K(\D),\K(\D)]$ gives $[\K(\id_\D)]=0$. The only way for the identity to be null-homotopic is if $\K(\D)$ itself is contractible, completing the proof.
\end{proof}

We now proceed in the spirit of Waldhausen in~\cite{Wal85}, beginning first by establishing that the K-theory space $\K(\D)$ is actually an infinite loop space. We will prove that $\K(\D)$ is homotopy equivalent to the space $\Omega \K([\id_\D])$, where $\K([\id_\D])$ denotes a space constructed from the morphism $\id_\D$ rather than the map $\id_\D$ induces on K-theory, which we would continue to denote $\K(\id_\D)$. The notation is imperfect, but we hope to imply by $[-]$ that we are thinking of the morphism in a different way. At any rate, we must first perform the construction of relative K-theory in the derivator context before proving this assertion.

Before we begin, we give a technical `Detection Lemma' that will be needed in the following proof. The statement is \cite[Proposition~3.10]{Gro13} with only cosmetic changes so that its statement is consistent with our notation.
\begin{lemma}\label{lemma:detectionlemma}
Let $\iota\colon \square\to K$ be a fully faithful functor and let $u\colon J\to K$ be any functor. We may consider the full category $K\setminus\iota(1,1)$ obtained by removing the image of the bottom-right corner of the square, and then form the comma category
\begin{equation*}
\vcenter{\xymatrix{
(K\setminus\iota(1,1)/\iota(1,1))\ar[r]\ar[d]&K\setminus\iota(1,1)\ar[d]\ar@{}[dl]|\swtrans\\
e\ar[r]_-{\iota(1,1)}&K
}}
\end{equation*}
where the righthand vertical map is the inclusion of the subcategory. This comma category receives a functor from $\ul$ induced by $\iota$, which we denote $\overline\iota$.

Assume that $\overline\iota\colon\ul\to(K\setminus\iota(1,1)/\iota(1,1))$ admits a left adjoint and that $\iota(1,1)$ does not lie in the image of $u\colon J\to K$. Then for all $X\in \D(J)$, the square $\iota^\ast u_! X\in \D(\square)$ is cocartesian.
\end{lemma}

Two points of notation, since $K$ is not assumed skeletal. By $K\setminus\iota(1,1)$ we always mean the full subcategory of $K$ spanned by all objects \emph{other than the object $\iota(1,1)$ itself}, and not by all objects non-isomorphic to it; and `$\iota(1,1)$ does not lie in the image of $u$' refers to the strict image, \ie, there is no $j\in J$ with $u(j)=\iota(1,1)$. These are the readings under which the proof goes through, and they are the ones we verify below; in any case the distinction is immaterial in our applications, where $K=\Ar[n]$ is a poset, so that isomorphic objects are equal and the strict and essential images agree.

To begin the construction of relative K-theory, we note that the $\Sdot$-construction on derivators not only gives us a simplicial category, but in fact a simplicial left pointed derivator.
\begin{prop}\label{prop:simplicialderivator}
Let $j\colon [n-1]\to\Ar[n]$ denote the inclusion $i\mapsto (0,i+1)$. Then~$j^\ast\colon \Sn\D\to\D([n-1])$ is an equivalence of categories. Thus we can give $\Sn\D\subset\D^{\Ar[n]}$ the structure of a left pointed derivator via the equivalence with $\D^{[n-1]}$.
\end{prop}
\begin{proof}
We will construct the quasi-inverse directly.

First, consider the functor $i_0\colon [n-1]\to [n]$ defined by $i\mapsto i+1$. This map is a cosieve, so the morphism $i_{0,!}$ is extension by zero.

Second, consider the subcategory $D\subset \Ar[n]$ which contains the top row $(0,i)$ as well as all diagonal entries $(i,i)$. The inclusion $i_1\colon [n]\to D$ is a sieve, so $i_{1,\ast}$ is extension by zero.

The last step is to take the inclusion $i_2\colon D\to \Ar[n]$ and compute $i_{2,!}$. We claim that the image of this composition lies in $\Sn\D(e)\subset\D(\Ar[n])$. To see this, we will apply the Detection Lemma~\ref{lemma:detectionlemma} with $J=[n-1]$ and $K=\Ar[n]$; note that $\Ar[n]$ is a poset, so the two points of notation raised after that lemma do not arise here. Let $\iota_{i,j}\colon\square\to \Ar[n]$ be the square given by $(a,b)\mapsto (i+b,j+a)$ for $0\leq i<j\leq n-1$, where the `flip' of $a$ and $b$ comes in because of an unfortunate inconsistency in the notation between $\square$ and $\Ar[n]$. It suffices to prove each of these squares is cocartesian, as any other square will be a composite of such squares. If a larger square can be subdivided into cocartesian squares, then it too is cocartesian by \cite[Proposition~3.13(1)]{Gro13}.

Because $\Ar[n]$ is a poset, we can identify the comma category
\begin{equation*}
(\Ar[n]\setminus (i+1,j+1)/(i+1,j+1))
\end{equation*}
as the full subcategory of $\Ar[n]$ on $(p,q)$ admitting a map to $(i+1,j+1)$, \ie $p\leq i+1$ and $q\leq j+1$, excluding $(p,q)=(i+1,j+1)$ as it has been removed. Call the resulting category $B^{i,j}$. We now construct a left adjoint for $\iota_{i,j}$ directly. We define~$\ell\colon B^{i,j}\to \ul$ by
\begin{equation*}
\ell(p,q)=\begin{cases} (0,0)&p\leq i\text{ and }q\leq j\\ (0,1)&q=j+1\\(1,0)&p=i+1\end{cases}
\end{equation*}
Direct computation shows that $\Hom(\ell(p,q),(a,b))=\Hom((p,q),\iota_{i,j}(a,b))$ for any elements $(p,q)\in B^{i,j}$ and $(a,b)\in\ul$. We can also construct the unit and the counit directly; the counit is just the identity on $\ul$, and the unit can only be the unique map $(p,q)\to \iota(\ell(p,q))$ at each $(p,q)\in B^{i,j}$.

This proves that $\iota_{i,j}^\ast i_{2,!} X$ is a cocartesian square for any $X\in\D(D)$. Pasting these squares together shows that any square in $i_{2,!}X$ is cocartesian. Moreover, if we have~$X=i_{1,\ast}i_{0,!}Y$ for some $Y\in\D([n-1])$, then $(i,i)^\ast X = 0$ by construction. Therefore $X\in\Sn\D(e)$. Call this total morphism $\Phi\colon\D([n-1])\to\Sn\D$.

Because $\Phi$ is constructed as left and right Kan extensions of fully faithful functors, it too is fully faithful. Moreover, it is left adjoint to $j^\ast$, with the counit $\id_{\D([n-1])}=j^\ast \Phi$ the identity modification. Because the left adjoint morphism is fully faithful, the unit is an isomorphism. Because both the unit and counit are invertible modifications, this gives an equivalence of categories. 
\end{proof}

\begin{remark}\label{rk:subderivator}
We can give an alternative proof that $\Sn\D$ has the structure of a left pointed derivator. First, for $K\in\Dirf$, we define $\Sn\D(K)\subset \D^{\Ar[n]}(K)$ to be the full subcategory on objects $X$ such that $k^\ast X\in \Sn\D$ for any $k\in K$. This makes $\Sn\D$ a prederivator on $\Dirf$. Der1 and Der2 follow immediately from its definition as a certain levelwise subcategories of a derivator, and it is also (weakly) pointed because the 0 object of $\D(\Ar[n])$ is in $\Sn\D(e)$.

For the remainder of the axioms, it suffices to show that the left and right Kan extensions in $\D^{\Ar[n]}$ land in the appropriate subcategory. Let $X\in\Sn\D(J)$ and let~$u\colon J\to K$ be a functor. We only know for sure that $u_!X\in\D^{\Ar[n]}(K)$, so we need to check that for all $k\in K$, $(i,i)^\ast k^\ast u_! X = 0$ for all $i\in[n]$ and for all squares~$\iota\colon\square\to\Ar[n]$, $\iota^\ast k^\ast u_! X\in\D^\square$ is cocartesian.

For the first point, let us just examine $(i,i)^\ast u_! X\in\D(K)$. Because $(i,i)^\ast$ is cocontinuous, we have $(i,i)^\ast u_! X\cong u_! (i,i)^\ast X$. We know that $(i,i)^\ast X = 0\in\D(J)$ because $X\in\Sn\D(J)$, so we have $u_!(i,i)^\ast X=0\in\D(K)$ because $u_!$ is a pointed morphism. Therefore $k^\ast (i,i)^\ast u_!X=0\in\D(e)$, and these first two morphisms commute because they are pullback morphisms in unrelated diagrammatic directions, giving us $(i,i)^\ast k^\ast u_! X=0$ for any $i\in [n]$ and $k\in K$.

For the second point, we have
\begin{equation*}
\iota^\ast k^\ast u_! X\cong k^\ast \iota^\ast u_! X\cong k^\ast u_! \iota^\ast X
\end{equation*}
for reasons identical to the above. We know that $\iota^\ast X$ is a cocartesian square in~$\D(\square\times J)$, and $u_!$ preserves cocartesian squares. This implies that each $k^\ast u_! \iota^\ast X$ is cocartesian in $\D(\square)$, and following the chain of isomorphisms backwards finishes the proof.

There is no difference if we consider the right Kan extension sieve $u\colon J\to K$, as the extension by zero morphism $u_\ast\colon\D^{\Ar[n]}(J)\to\D^{\Ar[n]}(K)$ is still cocontinuous. That was the only fact we used in the case of left Kan extensions $u_!$,  and so we complete the argument.
\end{remark}

Because $\Sn\D$ has the structure of a left pointed derivator, it means that $\Sdot\D$ is actually a simplicial object in left pointed derivators. This means we can iterate the $\Sdot$ construction, and will do so shortly. But before that, we will define our relative K-theory construction. To do so, we need the following general simplicial constructions.

Consider the functor $S\colon\Delta\to\Delta$ from the simplicial category defined on objects by $[n]\mapsto [n+1]$. If $f\colon [m]\to[n]$ is any morphism in the domain $\Delta$, its image is $\overline f\colon [m+1]\to [n+1]$ defined by $\overline{f}(i)=f(i-1)+1$ and $\overline f(0)=0$. If $Y$ is any simplicial set, then we can define the \emph{path object} on $Y$, denoted $PY$, by precomposing $Y\colon\Delta\op\to\mathbf{Set}$ with $S\op$.

\begin{lemma}[Lemma~1.5.1,~\cite{Wal85}]
$PY$ is simplicially homotopy equivalent to the constant simplicial set on $Y_0$.
\end{lemma}
There is a projection map $PY\to Y$ induced by the 0-face map. Moreover, there is a functor $Y_1\to PY$ which is the inclusion of 0-simplices, as $(PY)_0=Y_1$. This gives a sequence $Y_1\to PY\to Y$ for any simplicial set $Y$.

Now suppose that $\Phi\colon\D\to\E$ is a strict cocontinuous morphism of left pointed derivators. We then define the simplicial category $\Sdot[\Phi]$ by the following 2-pullback in $\CAT$, sometimes called the iso-comma, which lies between the lax pullback and the strict pullback:
\begin{equation*}
\xymatrix{
\mathbf{S}_\bullet[\Phi]\ar[r]\ar[d]&P\Sdot\E\ar[d]^{d_0}\ar@{}[dl]|\netrans\ar@{}[dl]<-1.25ex>|(0.55)\cong\\
\mathbf{S}_\bullet\D\ar[r]_-{\Phi}&\mathbf{S}_\bullet\E
}
\end{equation*}
Specifically, at each $[n]\in\Delta^\text{op}$, we have a square
\begin{equation*}
\xymatrix{
\Sn[\Phi]\ar[r]\ar[d]&(P\Sdot\E)_n=\mathbf{S}_{n+1}\E\ar[d]^{d_0}\ar@{}[dl]|\netrans\ar@{}[dl]<-1.25ex>|(0.55)\cong\\
\Sn\D\ar[r]_-{\Phi}&\Sn\E
}
\end{equation*}
in which the top-left corner is explicitly the following: an object in $\Sn[\Phi]$ is a pair \linebreak $(A\in\Sn\D,B\in\mathbf{S}_{n+1}\E)$ along with an isomorphism $f_{A,B}\colon\Phi(A)\to d_0(B)\in\Sn\E$. Note that if $\Phi$ is not a strict morphism, then there are naturality problems with the face and degeneracy maps of $\Sdot[\Phi]$.

If $\Phi$ is not a strict morphism, then by Proposition~\ref{prop:strictification} there are two strict (cocontinuous) morphisms $\Pi_\Phi\colon\widetilde\D\to \D$ and $\widetilde\Phi\colon \widetilde \D\to\E$ such that $\Pi_\Phi$ is a weak equivalence and $\Phi \Pi_\Phi=\widetilde\Phi$. While $\Sdot\Phi$ will not be defined directly, we will still be able to make sense of its K-theory (up to homotopy equivalence) by taking a strict replacement $\Sdot[\widetilde\Phi]$. Therefore it is not an issue for us to assume the morphism $\Phi$ is strict for the rest of this argument.

In Waldhausen K-theory, $\Sdot[F]$ for $F\colon \cC\to\cD$ an exact morphism of Waldhausen categories is again a simplicial Waldhausen category. For us, it is not immediately clear that $\Sdot[\Phi]$ should be a simplicial left pointed derivator, so we prove that now.

\begin{prop}
The simplicial category $\Sdot[\Phi]$ underlies a simplicial object in left pointed derivators (which we give the same name).
\end{prop}
\begin{proof}
For $K\in\Dirf$, the category $\Sn[\Phi](K)$ will have objects a triple
\begin{equation*}
A\in\Sn\D(K),\;B\in\mathbf{S}_{n+1}\E(K),\;f_{A,B}\colon \Phi(A)\overset{\cong}\longrightarrow d_0(B),
\end{equation*}
which we will shorten to $(A,B,f_{A,B})$. For $u\colon J\to K$, we define $u^\ast (A,B,f_{A,B})$ to be the triple $(u^\ast A,u^\ast B, g_{A,B})$, where $g_{A,B}$ is the map filling in the commutative diagram of isomorphisms below:
\begin{equation*}
\vcenter{\xymatrix@C=3em{
u^\ast \Phi (A)\ar[d]_-{\gamma_u^\Phi}\ar[r]^-{u^\ast f_{A,B}}&u^\ast d_0(B)\ar[d]^-{\gamma_u^{d_0}}\\
\Phi(u^\ast A)\ar@{-->}[r]_-{g_{A,B}}&d_0(u^\ast B)
}}
\end{equation*}
The vertical isomorphisms are actually equalities because $\Phi$ is assumed to be strict (and $d_0$ is in any case), so $g_{A,B}=u^\ast f_{A,B}$. We include the full picture for analogy with what follows. This proves that $\Sn[\Phi]$ has the structure of a prederivator. Der1 and Der2 follow as $\Sn\D$ and $\mathbf{S}_{n+1}\E$ are both semiderivators (and the data of $f_{A,B}$ does not affect these axioms).

A fair question at this point is why we need the flexibility of an isomorphism $f_{A,B}$ if $\Phi$ is assumed to be strict. If we required $f_{A,B}$ to be the identity, then we know that $u^\ast (f_{A,B})$ would also be the identity by strict 2-functoriality. The issue arises for the left and right Kan extensions, which we address now. Suppose now that $(A,B,f_{A,B})\in\Sn[\Phi](J)$. Then we define $u_!(A,B,f_{A,B})$ to be $(u_!A,u_!B,h_{A,B})$, where $h_{A,B}$ is the map filling in a similar commutative diagram of isomorphisms:
\begin{equation*}
\vcenter{\xymatrix@C=3em{
u_! \Phi(A)\ar[d]_-{(\gamma_u^\Phi)_!}\ar[r]^-{u_! f_{A,B}}&u_! d_0(B)\ar[d]^-{(\gamma_u^{d_0})_!}\\
\Phi(u_! A)\ar@{-->}[r]_-{h_{A,B}}&d_0(u_!B)
}}
\end{equation*}
Because $\Phi$ and $d_0$ are both cocontinuous, the left mates of the structure isomorphisms $\gamma_u$ are isomorphisms (by definition). However, even though $\gamma_u^{d_0}$ may be the identity, there is no reason to believe that its mates are also identities; they are only guaranteed to be isomorphisms. Therefore we have
\begin{equation*}
h_{A,B}=(\gamma_u^{d_0})_!\circ u_!f_{A,B}\circ (\gamma_u^{\Phi})_!\inv
\end{equation*}
This explains the definition of $\Sdot[\Phi]$ as an iso-comma (simplicial) category instead of a strict pullback as in Waldhausen's original construction.

These left Kan extensions for $\Sn[\Phi]$ are natural and (moreover) are the only ones that make sense. The construction of left Kan extensions also generalises to right Kan extensions along sieves $u\colon J\to K$, as the right mates $(\gamma_u^{\Phi})_\ast$ and $(\gamma_u^{d_0})_\ast$ will also be isomorphisms by Der4R.


Finally, $\Sn[\Phi](e)$ is pointed by $(0,0,\cong)$, where the isomorphism is unique, which gives $\Sdot[\Phi]$ the structure of a simplicial left pointed derivator.
\end{proof}

We can now formulate the statement of relative derivator K-theory. Let $\Phi\colon\D\to\E$ be a strict cocontinuous morphism of left pointed derivators. First, there is an inclusion $\mathbf{S}_1\E\to P\Sdot\E$ as zero simplices, where we view $\mathbf{S}_1\E$ as a constant simplicial left pointed derivator. Composing this with the map $d_0\colon P\Sdot\E\to \Sdot\E$ we obtain a sequence
\begin{equation*}
\mathbf{S}_1\E\to P\Sdot\E\to\Sdot\E
\end{equation*}
the composition of which is constant. We can lift the map $\mathbf{S}_1\E\to P\Sdot\E$\linebreak to $\mathbf{S}_1\E\to\Sdot[\Phi]$ using the pullback defining $\Sdot[\Phi]$ and the constant (at zero)\linebreak map $\mathbf{S}_1\E\to\Sdot\D\to\Sdot\E$
\begin{equation*}
\vcenter{\xymatrix{
\mathbf{S}_1\E\ar@(r,u)[drr]^-{0\text{-simplices}}\ar@(d,l)[ddr]_-0\ar@{-->}[dr]&&\\
&\Sdot[\Phi]\ar[r]\ar[d]&P\Sdot\E\ar[d]^-{d_0}\ar@{}[dl]|\netrans\ar@{}[dl]<-1.25ex>|(0.55)\cong\\
&\Sdot\D\ar[r]_-\Phi&\Sdot\E
}}
\end{equation*}

Composing with the projection from $\Sdot[\Phi]$ to $\Sdot\D$ we obtain a sequence
\begin{equation*}
\mathbf{S}_1\E\to \Sdot[\Phi]\to\Sdot\D
\end{equation*}
the composition of which is again constant. Iterating the $\Sdot$ construction, we have the following theorem.

\begin{theorem}\label{thm:localization1}
Let $\Phi\colon \D\to\E$ be a strict cocontinuous morphism of derivators. The sequence
\begin{equation}\label{eq:localization}
i\Sdot\mathbf{S}_1\E\to i\Sdot\Sdot[\Phi]\to i\Sdot\Sdot\D
\end{equation}
is a homotopy fibration after (diagonally) geometrically realizing nerves.
\end{theorem}
We will proceed as in~\cite[Proposition~1.5.5]{Wal85}, which requires the following `Realization Lemma' from~\cite{Wal78}:
\begin{lemma}\label{lemma:realizationlemma}
Let $X_{\bullet\bullet}\to Y_{\bullet\bullet}\to Z_{\bullet\bullet}$ be a sequence of bisimplicial sets so that $X_{\bullet\bullet}\to Z_{\bullet\bullet}$ is constant. Suppose that $X_{\bullet n}\to Y_{\bullet n}\to Z_{\bullet n}$ is a homotopy fibration for every $n$. Suppose further that $Z_{\bullet n}$ is connected for every $n$. Then the original sequence is also a homotopy fibration.
\end{lemma}
\begin{proof}[Proof of Theorem~\ref{thm:localization1}]
First, let us show that we can apply the Realization Lemma~\ref{lemma:realizationlemma} to our situation. We have a sequence of bisimplicial categories, which we will turn into a sequence of trisimplicial sets by taking the nerve:
\begin{equation*}
N_\bullet i\Sdot\mathbf{S}_1\E\to N_\bullet i\Sdot\Sdot[\Phi]\to N_\bullet i\Sdot\Sdot\D.
\end{equation*}
However, let us treat this as a bisimplicial set by considering the first two simplicial directions as one diagonal direction, \ie if we let $\bullet$ and $\star$ denote the two different directions, we have
\begin{equation*}
N_\bullet i\Sdot\mathbf{S}_1\E\to N_\bullet i\Sdot\mathbf{S}_\star[\Phi]\to N_\bullet i\Sdot\mathbf{S}_\star\D
\end{equation*}
which is now a sequence of bisimplicial sets. The first term appears the same because~$\mathbf{S}_1\E$ was constant in the $\star$ direction. As geometric realization may be taken variable-by-variable or diagonally (in fact, these give homeomorphic spaces), it suffices to show that this second sequence of bisimplicial sets is a homotopy fibration.

The last thing to check is that $N_\bullet i\Sdot\Sn\D$ is connected for all $n$. But $N_0 i\mathbf{S}_0\Sn\D$ consists of the zero objects in $\Sn\D(e)$, all of which are isomorphic, hence there is a 1-simplex in $N_1 i\mathbf{S}_0\Sn\D$ which is this isomorphism. Applying a degeneracy map in the $\mathbf{S}_0$-direction will give us a 1-simplex in the diagonal simplicial set $N_1 i\mathbf{S}_1\Sn\D$ which connects these 0-simplices, showing that this simplicial set is indeed connected.

Therefore let us fix an $n$ and consider the sequence
\begin{equation*}
i\Sdot\mathbf{S}_1\E\to i\Sdot\Sn[\Phi]\to i\Sdot\Sn\D
\end{equation*}
of simplicial left pointed derivators. We will make our argument here and pass to the nerve and the corresponding diagonal simplicial sets as we outlined above.

We will show, as Waldhausen does, that this relative K-theory sequence is homotopic to the trivial homotopy fibration. We will do so using the additivity theorem.

We define a cofibration morphism of derivators 
\begin{equation*}
\Xi_n\colon \Sn[\Phi]\to(\Sn[\Phi])_\text{cof}
\end{equation*}
such that $(0,0)^\ast\Xi_n$ takes values in a copy of $\mathbf{S}_1\E$ inside $\Sn[\Phi]$, $(1,0)^\ast\Xi_n=\id_{\Sn[\Phi]}$, and~$(1,1)^\ast\Xi_n$ takes values in a copy of $\Sn\D$ inside $\Sn[\Phi]$.

The sketch for the construction of $\Xi_n$ is the following: for $(A,B,f_{A,B})\in \Sn[\Phi]$,
\begin{equation*}
\Xi_n(A,B,f_{A,B})=
\vcenter{\xymatrix{
(0,s_n\cdots s_1(0\to (0,1)^\ast B\to 0),\cong)\ar[r]\ar[d]&(A,B,f_{A,B})\ar[d]\\
(0,0,\cong)\ar[r]&(A,s_0d_0(B),f_{A,B})
}}
\end{equation*}
where $s_i$ are the degeneracy maps of the simplicial set $P\Sn\E$. The entry in the top left is a degenerate $n$-simplex in $P\Sn\E$ which comes from $(0\to(0,1)^\ast B\to 0)\in P\mathbf{S}_0\E$. The isomorphisms between zero objects on the lefthand side are more subtle than they appear, and we will address this below.

To begin with the $\Sn\D$ component of $\Sn[\Phi]$, we define a cofibration morphism\linebreak $\Sn\D\to(\Sn\D)_\text{cof}$. Consider the map $p_n\colon\Ar[n]\times\square\to\Ar[n]$ defined as follows: for\linebreak $(a,b)=(1,0)$ or $(1,1)$, we let $p_n(i,j,a,b)=(i,j)$. Further, we let $p_n(i,j,0,0)=(0,0)$ and $p_n(i,j,0,1)=(0,0)$ be constant. We illustrate the functor $p_2\colon \Ar[2]\times\square\to\Ar[2]$, using bold arrows for the $\square$ dimension of the diagram. We label the objects of the domain according to where they map in the codomain, which shows better how the pullback $p_2^\ast$ behaves:
\begin{equation*}
\vcenter{\xymatrix@C=1em@R=1em{
0_1\ar[r]&0_1\ar[r]\ar[d]&0_1\ar[d] && 0_1\ar[r]&a\ar[r]\ar[d]&b\ar[d]\\
& 0_1\ar[r]&0_1\ar[d] &\ar@{=>}[r]& &0_2\ar[r]&c\ar[d]\\
&\ar@{=>}[d]& 0_1 && &\ar@{=>}[d]&0_3\\
&&&&&&&&\\
0_1\ar[r]&0_1\ar[r]\ar[d]&0_1\ar[d] && 0_1\ar[r]&a\ar[r]\ar[d]&b\ar[d]\\
& 0_1\ar[r]&0_1\ar[d] &\ar@{=>}[r]& &0_2\ar[r]&c\ar[d]\\
&& 0_1 && &&0_3
}}\to\quad
\vcenter{\xymatrix@C=1em@R=1em{
0_1\ar[r]&a\ar[r]\ar[d]&b\ar[d]\\
&0_2\ar[r]&c\ar[d]\\
&&0_3
}}
\end{equation*}
The horizontal arrows in the $\square$ dimension are necessarily zero maps, and the vertical arrows are identity maps, as $p_n(-,-,1,-)\colon \Ar[n]\times[1]\to\Ar[n]$ defining the righthand vertical map is just the projection $\id_{\Ar[n]}\times\pi_{[1]}$ and $p_n(-,-,0,-)$ defining the lefthand vertical map is the constant functor $\Ar[n]\times[1]\to e$. This square is cocartesian, and establishes the construction for $\Sn\D$. Note that the definition of $p_n$ implicitly uses that $n\geq 1$, but for the case $n=0$, $\Ar[0]=e$ so $p_0\colon\square\to e$ is the constant functor~$\pi_{\square}$ by necessity.

For the $P\Sn\E=\mathbf{S}_{n+1}\E$ component of the derivator $\Sn[\Phi]$, we define a\linebreak map $q_n\colon\Ar[n+1]\times\square\to\Ar[n+1]$ computing what we want. First, we will\linebreak have ${q_n(i,j,1,0)=(i,j)}$, just as $p_n$ was defined. To deal with the other $(a,b)\in\square$, we start with $(a,b)=(0,0)$:
\begin{equation*}
q_n(i,j,0,0)=\begin{cases} (0,0)&(i,j)=(0,0)\\ (0,1)&i=0,1\text{ and }j\geq 1\\(1,1)&\text{otherwise} \end{cases}
\end{equation*}

We want $q_n(-,-,0,1)$ to be a constant functor (as $p_n(-,-,0,1)$ was) but in this case we let $q_n(i,j,0,1)=(1,1)$. Finally, for the case $(a,b)=(1,1)$,
\begin{equation*}
q_n(i,j,1,1)=\begin{cases} (1,1)&(i,j)=(0,0)\\ (1,j)&i=0\text{ and }j\geq 1\\(i,j)&\text{otherwise}\end{cases}
\end{equation*}
To illustrate $q_2$, we have the following picture, where we again label the zeroes:
\begin{equation*}
\vcenter{\xymatrix@C=1em@R=1em{
0'_0\ar[r]&\alpha\ar[r]^-=\ar[d]&\alpha\ar[d]\ar[r]^-=&\alpha\ar[d] && 0'_0\ar[r]&\alpha\ar[r]\ar[d]&\beta\ar[d]\ar[r]&\gamma\ar[d]\\
& 0'_1\ar[r]&0'_1\ar[d]\ar[r]&0'_1\ar[d] &\ar@{}[d]_{}="a0"&\ar@{}[d]_{}="a1" &0'_1\ar[r]&\delta\ar[d]\ar[r]&\varepsilon\ar[d]\\
&& 0'_1\ar[r]&0'_1\ar[d] &&  &&0'_2\ar[r]&\zeta\ar[d]\\
&\ar@{}[r]_{}="c0"&& 0'_1 && &\ar@{}[r]_{}="d0"&&0'_3\\
&\ar@{}[r]_{}="c1"&&&&&\ar@{}[r]_{}="d1"&&\\
0'_1\ar[r]&0'_1\ar[r]\ar[d]&0'_1\ar[d]\ar[r]&0'_1\ar[d] && 0'_1\ar[r]&0'_1\ar[r]\ar[d]&\delta\ar[d]^-=\ar[r]&\varepsilon\ar[d]^-=\\
& 0'_1\ar[r]&0'_1\ar[d]\ar[r]&0'_1\ar[d] &\ar@{}[d]_{}="b0"&\ar@{}[d]_{}="b1" &0'_1\ar[r]&\delta\ar[d]\ar[r]&\varepsilon\ar[d]\\
&& 0'_1\ar[r]&0'_1\ar[d] &&  &&0'_2\ar[r]&\zeta\ar[d]\\
&&& 0'_1 && &&&0'_3
\ar@{=>} "a0";"a1" \ar@{=>} "b0";"b1" \ar@{=>} "c0";"c1" \ar@{=>} "d0";"d1"
}}
\to
\vcenter{\xymatrix@C=1em@R=1em{
0'_0\ar[r]&\alpha\ar[r]\ar[d]&\beta\ar[r]\ar[d]&\gamma\ar[d]\\
&0'_1\ar[r]&\delta\ar[r]\ar[d]&\varepsilon\ar[d]\\
&&0'_2\ar[r]&\zeta\ar[d]\\
&&&0'_3
}}
\end{equation*}

To check that this defines the functor $\Xi_n\colon\Sn[\Phi]\to(\Sn[\Phi])_\text{cof}$, we need to construct the isomorphisms $\Phi(p_n^\ast A)\to d_0(q_n^\ast B)$ corresponding to $\Xi_n(A,B,f_{A,B})$, and we will demonstrate how do so at each corner of $\square$. The isomorphisms at $(1,0)$ and $(1,1)$ are just $f_{A,B}$, but the morphisms at $(0,0)$ and $(0,1)$ are between specific zero objects. In the above diagrams, $f_{A,B}$ defines an isomorphism $\Phi(0_i)\to 0'_i$ for each $i$. By construction, $(0,0)^\ast p_n^\ast A$ and $(0,1)^\ast p_n^\ast A$ are the constant diagram on the zero object~$0_1$, and $d_0((0,0)^\ast q_n^\ast B)$ and $d_0((0,1)^\ast q_n^\ast B)$ are the constant diagram on the zero object~$0'_1$. Therefore $(0,0)^\ast f_{A,B}$ is still the appropriate isomorphism of zero objects and no additional data is required. Indeed, if we were forced to compose isomorphisms of the form $\Phi(0_1)\to 0'_1\to 0'_2$ (or worse, $\Phi(0_1)\to 0'_1\leftarrow 0_0'$), we could not complete this construction coherently.

Thus we define
\begin{equation*}
\Xi_n(A,B,f_{A,B})=(p_n^\ast A,q_n^\ast B,p_n^\ast f_{A,B}).
\end{equation*}
The structure isomorphisms $\gamma_u^{\Xi_n}$ come directly from the structure isomorphisms $\gamma_u^{p_n}$ and~$\gamma_u^{q_n}$.

The source morphism $(0,0)^\ast \Xi_n$ has image a subcategory of $\Sn[\Phi]$ equivalent to $\mathbf{S}_1\E$, and the quotient morphism $(1,1)^\ast \Xi_n$ has image a subcategory equivalent to $\Sn\D$. In particular, every object $(A,s_0d_0(B),f_{A,B})$ is isomorphic to the object $(A,s_0\Phi(A),\id)$. The target morphism $(1,0)^\ast\Xi$ has image precisely $\Sn[\Phi]$. These morphisms are moreover essentially surjective.

Now, there are two projections $\pi_1\colon\Sn[\Phi]\to\mathbf{S}_1\E$ and $\pi_2\colon\Sn[\Phi]\to\Sn\D$. The first is defined by $\pi_1(A,B,f_{A,B})=(0\to (0,1)^\ast B\to 0)$ and the second defined by $\pi_2(A,B,f_{A,B})=A$. This gives a total projection
\begin{equation*}
\rho\colon\Sn[\Phi]\to \mathbf{S}_1\E\times \Sn\D.
\end{equation*}
This map has a section $\sigma\colon\mathbf{S}_1\E\times\Sn\D\to\Sn[\Phi]$ given by
\begin{equation*}
\sigma((0\to b\to 0),A)=(A,s_n\cdots s_1(0\to b \to 0)\vee s_0\Phi(A),\id).
\end{equation*}
For example, for $A=(a_1\to a_2\to a_3)\in\mathbf{S}_2\D$, the component of $\sigma(b,A)$ in $P\mathbf{S}_2\E$ is
\begin{equation*}
\vcenter{\xymatrix@R=1em@C=1em{
0\ar[r]&b\ar[r]\ar[d]&b\vee \Phi(a_1)\ar[r]\ar[d]&b\vee \Phi(a_2)\ar[d]\\
& 0\ar[r] & \Phi(a_1)\ar[r]\ar[d] & \Phi(a_2)\ar[d]\\
&& 0\ar[r]&\Phi(a_3)\ar[d]\\
&&& 0
}}
\end{equation*}

The actual construction of $\sigma$ relies on the additive structure of $\Sn[\Phi]$. We first define $\mathbf{S}_1\E\times\Sn\D\to\Sn[\Phi]\times\Sn[\Phi]$ by the two morphisms
\begin{equation*}
\vcenter{\xymatrix@C=1em{
\mathbf{S}_1\E\times\Sn\D\ar[r]&\mathbf{S}_1\E\ar[r]^-{S'}&\Sn[\Phi]
}}\text{ and }
\vcenter{\xymatrix@C=1em{
\mathbf{S}_1\E\times\Sn\D\ar[r]&\Sn\D\ar[r]^-{Q'}&\Sn[\Phi]
}}
\end{equation*}
where the first map in each case is the projection. The morphism $S'$ is defined by 
\begin{equation*}
S'(0\to b\to 0)=(0,s_n\cdots s_1(0\to b \to 0),\cong)
\end{equation*}
and the morphism $Q'$ is defined by
\begin{equation*}
Q'(A)=(A,s_0\Phi(A),\id).
\end{equation*}
We then compose $(S',Q')\colon \mathbf{S}_1\E\times\Sn\D\to\Sn[\Phi]\times\Sn[\Phi]$ with the coproduct map\linebreak $\Sn[\Phi]\times\Sn[\Phi]\cong\Sn[\Phi]^{e\sqcup e}\to\Sn[\Phi]$. The definition of this map and how to strictify it is contained in Equation~\ref{eq:strictifiedcoproduct} and Lemma~\ref{lemma:functoriality} above; in particular $\sigma$ is cocontinuous but not strict, so the statements about it below are statements in $\Ho\cS$.

By construction, $\rho\sigma\cong\id_{\mathbf{S}_1\E\times\Sn\D}$, and so obtain a homotopy after passing to K-theory. We now to show that the reverse composition $\sigma\rho$ is homotopic to the identity.

By applying the additivity theorem to $\Xi_n$ -- which is legitimate because $\Xi_n$ is given by the pullbacks $p_n^\ast$ and $q_n^\ast$ and is therefore strict -- the identity on $\Sn[\Phi]$ is homotopic to the sum of the inclusions of $\mathbf{S}_1\E$ and $\Sn\D$, and this is a morphism isomorphic to $\sigma\rho$. Therefore map of sequences
\begin{equation*}
\xymatrix{
i\Sdot\mathbf{S}_1\E\ar[r]\ar[d]^-=& i\Sdot\Sn[\Phi]\ar[r]\ar[d]^-{\sim}& i\Sdot\Sn\D\ar[d]^-=\\
i\Sdot\mathbf{S}_1\E\ar[r]& i\Sdot\mathbf{S}_1\E\times i\Sdot\Sn\D \ar[r]& i\Sdot\Sn\D
}
\end{equation*}
has all vertical maps equivalences. The bottom sequence is a trivial fibration, so the top sequence is also a homotopy fibration, completing the proof. 
\end{proof}

\begin{defn}
For $\Phi\colon\D\to\E$ a strict cocontinuous morphism of left pointed derivators, define
\begin{equation*}
\K([\Phi]):=\Omega^2|i\Sdot\Sdot[\Phi]|.
\end{equation*}
\end{defn}

We can now make good on our initial promise of this section.

\begin{cor}
The topological space $\K(\D)$ is an infinite loop space.
\end{cor}
\begin{proof}
If we take the case $\Phi=\id_\D\colon\D\to\D$, we can identify $\Sdot[\id_\D]$ with $P\Sdot\D$. There is certainly an equivalence of simplicial categories $\Sdot[\id_\D](K)\to P\Sdot\D(K)$ for each $K\in\Dirf$ as the pullback of the equivalence $\id_\D\colon\Sdot\D(K)\to\Sdot\D(K)$ is still an equivalence. But because the morphism $\Sdot[\id_\D]\to P\Sdot\D$ is defined globally and is levelwise an equivalence of categories, we get an equivalence of simplicial left pointed derivators.

Using this replacement, we have a fibration
\begin{equation*}
i\Sdot\mathbf{S}_1\D\to P(i\Sdot\Sdot\D)\to i\Sdot\Sdot\D,
\end{equation*}
where $P$ modifies the first simplicial direction. But now the middle term is contractible, giving a homotopy equivalence
\begin{equation*}
\xymatrix{
|i\Sdot\D|\cong |i\Sdot\mathbf{S}_1\D|\ar[r]^-\sim& \Omega |i\Sdot\Sdot\D|
}
\end{equation*}
We use here that the bisimplicial category $i\Sdot\mathbf{S}_1\D$ is homotopy equivalent to the simplicial category $i\Sdot\D$. The equivalence is given (using morphisms of derivators and passing to maps up to homotopy in the geometric realization) by the forgetful functor $\mathbf{S}_1\D\to~\D$ on the one hand and an iterated extension by zero morphism for the inverse.

Specifically, consider the cosieve $t\colon e\to[1]$ given by the inclusion of the target and the sieve $s\colon [1]\to\Ar[1]$ given by the inclusion into $(0,0)\to(0,1)$. Then the morphism $s_\ast t_!\colon \D\to \D(\Ar[1])$ is, for $a\in \D$,
\begin{equation*}
a\mapsto
\vcenter{\xymatrix@R=1em@C=1em{
0\ar[r]& a}}
\mapsto
\vcenter{\xymatrix@R=1em@C=1em{
0\ar[r]&a\ar[d]\\
&0
}}
\end{equation*}
This justifies the isomorphism above.

We now replace $\D$ by the simplicial left pointed derivator $\Sdot\D$ and repeat the process, obtaining
\begin{equation*}
\vcenter{\xymatrix{
|i\Sdot\Sdot\D|\cong|i\Sdot\mathbf{S}_1\Sdot\D|\ar[r]^-\sim&\Omega|i\Sdot\Sdot\Sdot\D|
}}
\end{equation*}
which implies that $|i\Sdot\D|$ is equivalent to $\Omega^2|i\Sdot\Sdot\Sdot\D|$.

By induction, we then conclude
\begin{equation*}
\xymatrix{
\K(\D)=\Omega|i\Sdot\D|\ar[r]^-\sim& \Omega^{(n)}|i\Sdot^{(n)}\D|
}
\end{equation*}
and so $\K(\D)$ is an infinite loop space. In particular, we can view $\K(\D)$ as a connective $\Omega$-spectrum.
\end{proof}

\begin{cor}
Let $\Phi\colon\D\to\E$ be a strict cocontinuous morphism of left pointed derivators. Then there is a homotopy fibration
\begin{equation*}
\K([\Phi])\to \K(\D)\to \K(\E)
\end{equation*}
\end{cor}
\begin{proof}
If we rotate the fibration sequence of Equation~\ref{eq:localization} to the left twice (and replace $\mathbf{S}_1\E$ by $\E$), we have a fibration sequence
\begin{equation*}
\Omega|i\Sdot\Sdot[\Phi]|\to\Omega|i\Sdot\Sdot\D|\to|i\Sdot\E|.
\end{equation*}
By the above corollary, $|i\Sdot\D|\to\Omega|i\Sdot\Sdot\D|$ is a homotopy equivalence. Replacing the middle term and applying $\Omega$ everywhere, the corollary follows.
\end{proof}

The next logical step is to use Theorem~\ref{thm:localization1} to prove a localization theorem in K-theory and answer (positive or negatively) Maltsiniotis' conjecture that Verdier quotients of triangulated derivators get sent to long exact sequences in K-theory. Unfortunately, the necessity of coherent diagrams in derivator K-theory obstructs \cite[Theorem~1.6.4]{Wal85} from proceeding verbatim.

This is not a problem \emph{per se}, as we do not expect derivator K-theory to merely be a translation of Waldhausen K-theory. Work of To\"en-Vezzosi in \cite[\S7]{ToeVez04} argues that general Waldhausen categories should not have equivalent derivator and Waldhausen K-theory. In particular, let $X$ be a space and let $\cW_X$ be the Waldhausen category of spaces admitting a retract to $X$. Then there is an action of $\operatorname{Aut}(X)$ on both $A(X)=\K(\cW_X)$ and $\K(\D_{\cW_X})$. But this action must be 1-truncated on the derivator side and is generally \emph{not} 1-truncated on the A-theory side. An explicit (algebraic) example is also worked out in \cite{MurRap11} as we noted earlier.

Nonetheless, we hold out hope that some version of a localization theorem should hold. We cannot have both agreement and localization, but we know that we do not have agreement -- just a 2-connected comparison map from Waldhausen to derivator K-theory as proven by Raptis in \cite{Rap19}. We hope an interested reader will prove some form of localization -- in particular, the original conjecture of Maltsiniotis concerning only triangulated derivators enhancing $D^b(\cE)$ for exact categories $\cE$ remains open.

\section{$\infty$-categorical coda}

We begin by looking at Raptis' construction of higher derivators in \cite{Rap19}.
\begin{defn}[Definition~4.1, \cite{Rap19}]
An $\infty$-prederivator is an enriched functor
\begin{equation*}
\D\colon\Dia\op\to\CAT_\infty
\end{equation*}
valued in the $\infty$-category of (not necessarily small) $\infty$-categories, where we enrich $\Dia\subset\Cat\subset\Cat_\infty$ in $\infty$-categories as well. An $\infty$-prederivator is an $n$-prederivator if it takes values in the subcategory of $n$-categories in $\CAT_\infty$, \ie those $\infty$-categories that are equivalent to their homotopy $n$-category.
\end{defn}

The K-theory of a higher derivator requires a slight cosmetic adjustment:
\begin{defn}[Definition~5.3, \cite{Rap19}]
Let $\D$ be a left pointed $\infty$-derivator. Then we define
\begin{equation*}
\K(\D):=\Omega|\Sdot^\simeq\D|
\end{equation*}
Here, $\Sdot\D$ is the full sub-$\infty$-category of $\D(\Ar[n])$ precisely as above. Thus $\Sdot\D$ is a simplicial object in (pointed) $\infty$-categories, and $\Sdot^\simeq\D$ is the restriction of that object to the maximal sub-$\infty$-groupoid. This is analogous to taking $N_\bullet i\Sdot\D$, and if $\D$ is a 1-derivator exactly the same. Working with quasicategories as our model of $\infty$-categories, the `nerve' part of the classical definition is built-in.
\end{defn}

This paper has been concerned with classical `1-derivators', but not too much is changed in the passage to higher derivators. We define a \emph{left pointed $\infty$-derivator} in an analogous way. The explicit definitions can be found in \cite[\S4]{Rap19}, but the summary is the same: a left pointed $\infty$-derivator has a pointed base, admits all left Kan extensions, and admits right Kan extensions along sieves. We let $\Der^{(n)}_K$ denote the 2-category of left pointed $n$-derivators, cocontinuous morphisms of $n$-derivators, and isomodifications, where we allow $1\leq n\leq \infty$. 

It is worth observing that, for a functor $u\colon J\to K$ in $\Dia$, the left adjoint $u_!$ to a pullback functor $u^\ast$ requires `more data' if this adjunction is occurring in $\CAT_\infty$, \eg the unit and counit of the adjunction needs to contain some higher homotopical information. But this entire theory is detected at the level of the homotopy 2-category of an $\infty$-category, as discussed at length in the work of Riehl-Verity \cite{RieVer21}.

In particular, for any functor $u\colon J\to K$ in $\Dia$ and any left pointed $\infty$-derivator $\D$, we obtain cocontinuous morphisms of $\infty$-derivators $u^\ast\colon \D^K\to \D^J$ and $u_!\colon \D^J\to \D^K$. If $i\colon A\to B$ is a sieve, we also obtain a cocontinuous morphism of $\infty$-derivators $i_\ast\colon \D^A\to \D^B$. We note that every morphism and natural transformation in the main theorems of this paper were purely `diagrammatic', \ie the morphisms are compositions of pullbacks and Kan extensions, and the natural transformations pullbacks or mates. Therefore the categories $\Der^{(n)}_K$ admit all the constructions in the key theorems of this paper, \eg the functor $\sigma$ in the proof of Theorem~\ref{thm:additivities}.

There was one more key construction that we needed: strictification. But Proposition~\ref{prop:strictification}, which strictifies arbitrary morphisms of derivators $\Phi\colon\D\to\E$, can be rewritten to higher derivators: going now into the (omitted above) details of the proof, for each $J\in\Dia$, the category $\widetilde\D(J)$ is defined to have objects triples $X\in \D(J)$, $Y\in\E(J)$, and an (incoherent) isomorphism $f\colon Y\to \Phi(X)$. The maps of $\widetilde\D(J)$ are pairs of maps in $\D(J)$ and $\E(J)$ compatible with the isomorphisms. The prederivator structure of $\widetilde\D$ comes from the prederivator structures on $\D$ and $\E$ along with the structure isomorphisms $\gamma^\Phi$ of the morphism $\Phi$. The morphisms $\widetilde \D\to\D$ and~$\widetilde \D\to \E$ are just the projections, which are necessarily strict. All of this data has no 1-categorical bias (up to replacing the word `isomorphism' with `equivalence') so can be promoted to any higher derivator. Consequently Lemma~\ref{lemma:functoriality} holds verbatim for left pointed $n$-derivators, $1\leq n\leq\infty$: writing $\|\D\|:=|\Sdot^\simeq\D|$, every cocontinuous morphism $\Phi$ induces $\|\Phi\|$ in $\Ho\cS$, functorially and invariantly under isomodification. This matters because Lemma~\ref{lemma:functoriality} is the one non-diagrammatic ingredient of Section~\ref{sec:additivity}: the constructions $\nabla_{\Phi,\Psi}$ and $\Theta$, and both identifications of the ends of the cylinder, are phrased entirely in terms of $\|-\|$, so once it is available the arguments of that section transfer with no further change.

Because the theory of left pointed $\infty$-derivators can be approached with the same sort of diagrammatic arguments as 1-derivators, Theorem~\ref{thm:additivities} and the proof of (2) using the method of \cite{Gra11} can be repeated verbatim, as all the constructions we perform work equally well for any higher derivator. Thus we can simultaneously conclude that additivity holds for $n$-derivator K-theory for $1\leq n\leq \infty$, which answers a question raised by Raptis in \cite[Remark~5.4]{Rap19}.

\begin{theorem}
Let $\D$ and $\E$ be left pointed $\infty$-derivators. Then for any cofibration morphism of $\infty$-derivators $\Xi\colon\D\to\E_\text{cof}$ with associated morphisms $S,T,Q\colon \D\to\E$, there exists a homotopy of maps of spaces
\begin{equation*}
\K(T)\simeq \K(S\vee Q).
\end{equation*}
\end{theorem}
For the arguments of \S5 to apply, the only further comment we need is that the maximal sub-$\infty$-groupoid $\Sdot^\simeq\D$ is (in the quasicategorical model) a simplicial set, so can go through the same machinery of Grayson.

\begin{remark}
Remark~\ref{rk:hspace} also carries over verbatim, and we spell it out because the notation $\K(S\vee Q)$ deserves the same explanation here as it did there. For a left pointed $\infty$-derivator $\E$, the coproduct morphism $\vee\colon\E^{e\sqcup e}\to\E$ is again the composite $(1,1)^\ast i_{\ul,!}i_!$ of Equation~\ref{eq:coprod}, and is again cocontinuous but not strict; by the $\infty$-categorical form of Lemma~\ref{lemma:functoriality} discussed above it induces a based map
\begin{equation*}
m\colon\big|\Sdot^\simeq\E\big|\times\big|\Sdot^\simeq\E\big|\cong\big|\Sdot^\simeq\E^{e\sqcup e}\big|\longrightarrow\big|\Sdot^\simeq\E\big|
\end{equation*}
in $\Ho\cS$, defined already on the delooping because $\vee$ acts on the whole simplicial object $\Sdot^\simeq\E$, and with the basepoint as a two-sided homotopy unit. Looping, $\Omega m$ is a multiplication on $\K(\E)$; interchange with loop concatenation holds on the nose exactly as in Remark~\ref{rk:hspace}, and the Eckmann--Hilton argument identifies the two operations on $[\K(\D),\K(\E)]$ and shows they are commutative. Hence the theorem above may equivalently be stated as
\begin{equation*}
[\K(T)]=[\K(S)]+[\K(Q)]\in[\K(\D),\K(\E)],
\end{equation*}
which is the $\infty$-categorical form of the statement of additivity that Waldhausen records in \cite[Proposition~1.3.2(4)]{Wal85}. Nothing in this argument used $n=1$; it applies to left pointed $n$-derivators for all $1\leq n\leq\infty$.
\end{remark}

\begin{remark}\label{rk:cosmology}
Our argument purely in terms of Kan extensions and adjunctions could probably be rewritten purely in $\infty$-cosmological language, so that any `left pointed $\infty$-cosmos' (which would include left pointed $n$-derivators for $n\in \N\cup\{\infty\}$) admits an additive K-theory construction. Making this argument carefully with the correct vocabulary would be a productive exercise for a reader with more knowledge of this field.
\end{remark}

We now compare (higher) derivator K-theory to Waldhausen K-theory. The K-theory of stable $\infty$-categories as described in \cite{BluGepTab13} is equivalent to Waldhausen K-theory, at least when we restrict to those Waldhausen categories $\cC$ which give rise to a stable $\infty$-category after Hammock localisation (see \S7.2 ibid.). On the other hand, derivator K-theory is not equivalent to Waldhausen K-theory, even for $n$-derivator K-theories as defined by Raptis (with $n<\infty$). By \cite[Theorem~5.5]{Rap19}, the canonical comparison map $\mu^{(n)}\colon K_\infty(\cC)\to K\left(\D^{(n)}_\cC\right)$ is $(n+1)$-connected for left pointed $\infty$-categories, where $\D^{(n)}_\cC$ is the $n$-derivator associated to $\cC$. These generalize the 1-categorical map from Equation~\ref{eq:comparisonmap}, and moreover these $\mu^{(n)}$ are part of a tower of comparisons:
\begin{equation*}
\vcenter{\xymatrix{
&\K_\infty(\cC)\ar[d]_-{\mu^{(n)}}\ar[rd]^-{\mu^{(n-1)}}\ar@(r,u)[rrrd]^-{\mu^{(1)}=\mu}\\
\cdots\ar[r]&\K\left(\D_\cC^{(n)}\right)\ar[r]&\K\left(\D_\cC^{(n-1)}\right)\ar[r]&\cdots\ar[r]&\K(\D_\cC)
}}
\end{equation*}

This sets up a series of comparisons which we would like to explore in more detail. To do so, recall that the delooping of (higher) derivator K-theory means that the spaces $K\left(\D_\cC^{(n)}\right)$ are infinite loop spaces for all $n$.
\begin{lemma}
(Higher) Derivator K-theory extends to a functor of $\infty$-categories $N_\bullet\Der^{\text{str},(n)}_K\to~\Sp$, where $\Sp$ is the $\infty$-category of spectra.
\end{lemma}
\begin{proof}
Our domain is not all of $\Der_\text{K}$, but rather the wide subcategory where we restrict to only strict morphisms of derivators. As we have pointed out, any morphism in $\Der_\text{K}$ can be written as a span of strict morphisms, one of which is an equivalence, so that there is still a morphism in the homotopy category of $N_\bullet\Der_\text{K}$ for any (non-strict) cocontinuous morphism between left pointed derivators.

Muro-Raptis in \cite[Proposition~5.1.3]{MurRap17} prove that there is a simplicial functor $N_\bullet\Der_\text{K}^\text{str}\to \mathbf{SSet}$, where the category of simplicial sets is simplicially-enriched in the usual way: $\Hom_n(X,Y)=\Hom(X\times[n],Y)$. Our functor can be defined identically: on the objects of $\Der_\text{K}^{\text{str},(n)}$, we define $\K(\D)=\Sdot^\simeq\D$, \ie the maximal sub-$\infty$-groupoid before geometric realization and desuspending. The definition of this simplicial functor on higher simplices of $N_\bullet \Der_\text{K}^{\text{str},(n)}$ is identical to the proof in Muro-Raptis, as the only technology the authors use are pullback functors. See also \cite[\S5.2]{Rap19}.
\end{proof}

\begin{prop}
Derivator K-theory is an additive invariant of stable $\infty$-categories. Specifically, there is a functor $K_\D\colon\StCat_\infty^\text{ex}\to\Sp$ from the $\infty$-category of small stable $\infty$-categories and exact functors to the $\infty$-category of spectra that inverts Morita equivalences, preserves filtered colimits, and sends split-exact sequences to fibrations.
\end{prop}

\begin{proof}
How does one obtain a derivator from an $\infty$-category? Carlson in \cite[Definition~9]{Arl20} gives the following natural definition. For $Q$ a quasicategory, define the prederivator $\HO(Q)$ by
\begin{equation*}
\HO(Q)\colon J\mapsto \Ho\left(Q^{N_\bullet J}\right),\quad u\colon J\to K\mapsto \Ho(N_\bullet u^\ast)\colon\Ho\left(Q^{N_\bullet K}\right)\to\Ho\left(Q^{N_\bullet J}\right)
\end{equation*}
where $Q^{N_\bullet J}$ is the quasicategory of simplicial maps $N_\bullet J\to Q$. The action of $\HO(Q)$ on natural transformations is the only sensible one given the above. Carlson proves that all such prederivators satisfy Der1, Der2, and Der5. Moreover, if $Q$ admits (homotopy) limits and colimits, so does $\HO(Q)$. In particular, if $Q$ is a stable quasicategory, it admits all homotopy finite limits and colimits, so $\HO(Q)$ is a strong derivator on~$\Dirf$. Moreover, an exact functor of stable $\infty$-categories preserves all finite limits and colimits, so induces a cocontinuous morphism of the corresponding derivators. Therefore we have a map $\HO(\StCat_\infty^\text{ex})\to\Der_\text{K}$ on objects.

Carlson further proves in \cite[Corollary~20]{Arl20} that quasicategories embed simplicially fully-faithfully into the simplicial enrichment of (pre)derivators described above. In particular, part of Carlson's proof is that all functors of quasicategories give strict morphisms of derivators, so we do not have an invisible `strictification' step in the middle. This gives us a simplicial embedding of $\StCat_\infty^\text{ex}$ into $N_\bullet\Der_\text{K}$.

As derivator K-theory takes values in infinite loop spaces, \ie connective spectra, the (enriched) suspension spectrum functor $\Sigma_+^\infty\colon\mathbf{SSet}\to\Sp$ is fully faithful on the image of $K$ in $\mathbf{SSet}$. The total definition becomes
\begin{equation*}
\vcenter{\xymatrix{
K_\D\colon\StCat_\infty^\text{ex}\overset{\natural}\longrightarrow\StCat_\infty^\text{perf}\subset\textbf{QCat}\overset{\HO}\longrightarrow N_\bullet\Der_\text{K}\overset{K}\longrightarrow\mathbf{SSet}\to\Sp
}}
\end{equation*}

The functor $\natural$ is the idempotent completion functor, which also appears as the first step of the universal additive invariant in \cite{BluGepTab13}; this ensures that Morita equivalences are inverted. The category $\StCat_\infty^\text{perf}$ is just the full subcategory idempotent complete $\infty$-categories. That split exact sequences are sent to fibrations is exactly the additivity theorem.

Finally, we need to address filtered colimits. First, the idempotent completion functor is  left adjoint to the inclusion $\StCat_\infty^\text{perf}\subset \StCat_\infty^\text{ex}$ by (for example) the comments near \cite[Definition~2.14]{BluGepTab13}, so preserves all colimits. Therefore assume that $Q\colon I\to \Cat_\infty^\text{perf}$ is a filtered diagram. For any $K\in\Dirf$, we can consider the comparison map
\begin{equation*}
\colim_I\HO(Q_i)(K)\to\HO(\colim_I Q_i)(K)
\end{equation*}
which, if we unwind the definition, is
\begin{equation*}
\colim_I \Ho\left(Q_i^{N_\bullet K}\right)\to \Ho\left((\colim_I Q_i)^{N_\bullet K}\right)
\end{equation*}
The functor $\Ho\colon\Cat_\infty\to\Cat$ which gives the underlying category of a quasicategory is left adjoint the nerve functor, so preserves all colimits. Therefore we need only consider the comparison map between quasicategories
\begin{equation*}
\colim_I Q_i^{N_\bullet K}\to (\colim_I Q_i)^{N_\bullet K}
\end{equation*}
Here we may use the fact that $K\in\Dirf$, so that $N_\bullet K$ has only finitely many nondegenerate simplices. In particular, this is a compact object in simplicial sets, so the functor $(-)^{N_\bullet K}$ commutes with filtered colimits.

We therefore obtain an equivalence
\begin{equation*}
\colim_I\HO(Q_i)(K)\to\HO(\colim_I Q_i)(K)
\end{equation*}
for all $K\in\Dirf$, which assemble to an equivalence of the corresponding derivators. This equivalence passes through to K-theory, completing the argument.
\end{proof}

\begin{remark}
The above proposition should also hold for higher derivator K-theory, but for one step to be addressed in future work.

To define the $n$-derivator associated to a quasicategory $Q$ for $n>1$, for $J\in\Dia$ rather than taking the homotopy category of $Q^{N_\bullet J}$, one takes the homotopy $n$-category; see \cite[\S2.2]{Rap19} for more details. However, it is not immediately obvious that we still get a fully-faithful simplicial embedding of $\mathbf{QCat}$ into $\Der^{(n)}_K$ for $n>1$, as \cite[Theorem~19]{Arl20} concerns only the prederivator $\HO(Q)$ built from homotopy categories, \ie the case $n=1$. A more systematic approach as we considered in Remark~\ref{rk:cosmology} would likely clean up these issues.
\end{remark}

In light of this proposition, we should expect that the comparison map between Waldhausen and derivator K-theory that Muro-Raptis studied in \cite{MurRap17} agrees with the `universal trace map' $K_\infty\to K_\D$ guaranteed by the results of \cite[Theorem~10.3]{BluGepTab13}. In particular, the spectrum of natural transformations $\operatorname{Nat}(K_\infty,K_\D)$ is isomorphic to $K_\D(\Sp^\omega)$, the K-theory associated to the derivator of (compact) spectra, \ie the derivator K-theory of modules over the sphere spectrum $\bS$. 

We know that $\pi_0K_\infty(\bS)\cong\Z$, $\pi_1K_\infty(\bS)\cong \Z/2\Z$, and $\pi_2K_\infty(\bS)\cong\Z/2\Z$, which can be found in \cite{BluMan19}. Using the results of \cite{Rap19} (building on \cite{Mur08}), the comparison map $K_\infty\to K_\D$ is 2-connected. This means that $\pi_0 K_\D(\bS)\cong \Z$, so that the derivator trace map should correspond to an integer $x\in\Z$. In \cite[Theorem~10.6]{BluGepTab13}, the authors prove that the Dennis trace $K\to T\!H\!H$ corresponds to $1\in T\!H\!H_0(\bS)\cong\Z$, but this relies on earlier (classical) computations by Waldhausen. We know from \cite{MurRap17} that 1-derivator K-theory is initial among approximations to Waldhausen K-theory that invert equivalences of derivators. Those interested are encouraged to explore this perspective more closely and identify the integer corresponding to the `derivator trace'.

\bibliography{master-bibliography}
\bibliographystyle{alpha}

\end{document}